\documentclass[12pt]{article}

\usepackage{verbatim}
\usepackage{natbib}
\usepackage{amsfonts}
\usepackage{amssymb}
\usepackage{amsmath}
\usepackage{amsthm}
\usepackage{mathrsfs}
\usepackage[pdftex]{color,graphicx}
\usepackage{enumerate}
\usepackage{booktabs} %
\usepackage{hyperref}
\usepackage{todonotes}
\usepackage[margin=3.8cm]{geometry}
\usepackage{setspace}
\usepackage{pgfplots}
\usepackage{amstext}
\usepackage{setspace}

\newcommand{\bs}{\boldsymbol}
\def\bal#1\eal{\begin{align}#1\end{align}}
\def\bals#1\eals{\begin{align*}#1\end{align*}}
\def\be#1\ee{\begin{equation}#1\end{equation}}

\DeclareMathOperator*{\argmin}{argmin\;}
\DeclareMathOperator{\trace}{trace}
\DeclareMathOperator{\Var}{Var}
\DeclareMathOperator{\diag}{diag}

\newtheorem{definition}{Definition}
\newtheorem{theorem}{Theorem}
\newtheorem{corollary}{Corollary}
\newtheorem{lemma}{Lemma}

\newtheorem{assumption}{Assumption}

\title{From Cross-Validation to SURE: Asymptotic Risk of Tuned Regularized Estimators}
\author{Karun Adusumilli\footnote{Department of Economics, University of Pennsylvania. \href{mailto:akarun@sas.upenn.edu}{akarun@sas.upenn.edu}.} 
\and Maximilian Kasy\footnote{Corresponding author. Department of Economics, University of Oxford. \href{mailto:maximilian.kasy@economics.ox.ac.uk}{maximilian.kasy@economics.ox.ac.uk}. Maximilian Kasy was supported by the Alfred P. Sloan Foundation, under the grant ``Social foundations for statistics and machine learning.''}
\and Ashia Wilson\footnote{Department of Electrical Engineering and Computer Science, MIT.}}

\begin{document}
\maketitle
\onehalfspacing
\begin{abstract}
    We derive the asymptotic risk function of regularized empirical risk minimization (ERM) estimators tuned by $n$-fold cross-validation (CV). The out-of-sample prediction loss of such estimators converges in distribution to the squared-error loss (risk function) of shrinkage estimators in the normal means model, tuned by Stein's unbiased risk estimate (SURE).
    This risk function provides a more fine-grained picture of predictive performance than uniform bounds on worst-case regret, which are common in learning theory: it quantifies how risk varies with the true parameter.

    As key intermediate steps, we show that (i) $n$-fold CV converges uniformly to SURE, and (ii) while SURE typically has multiple local minima, its global minimum is generically well separated. Well-separation ensures that uniform convergence of CV to SURE translates into convergence of the tuning parameter chosen by CV to that chosen by SURE.
\end{abstract}

\clearpage
\section{Introduction}

\paragraph{Background}
The goal of supervised learning is to produce good predictions for new observations.\footnote{We would like to dedicate this paper to Gary Chamberlain, whose conversations provided the original inspiration for this project.} An important class of estimators for supervised learning are {regularized empirical risk minimization (ERM) estimators} that are {tuned using cross-validation (CV)}. 

{ERM estimators} minimize in-sample average prediction loss (empirical risk), among a given class of predictors. Examples include ordinary least squares and maximum likelihood.
ERM estimators are prone to overfitting if the class of predictors is large. Such estimators achieve low in-sample loss, but can perform poorly for new observations. To counter overfitting, {regularization} is used. Regularization adds a penalty term to the ERM objective; common penalties include the $L^{2}$ norm of parameters (in Ridge regression) and the $L^{1}$ norm (in Lasso regression). Adding a penalty avoids overfitting, by reducing the estimator variance, at the cost of introducing some bias, which might result in underfitting.

To achieve good performance, avoiding both overfitting and underfitting, the amount of penalization needs to be carefully {tuned}. This can be done by choosing weights for the penalty that minimize a {cross-validation} estimate of predictive loss. We  focus on n-fold CV, where predictions are evaluated for one hold-out observation at a time, and predictive loss is estimated by averaging evaluations over each of the $n$ observations.

\paragraph{Risk functions}
The present paper characterizes the behavior of estimators of this form by deriving an asymptotic approximation to their {risk function}. A large literature in learning theory characterizes such estimators by proving bounds on their worst-case regret---the supremum over data-generating processes (DGPs) of the difference between an estimator's risk and the expected loss of the best predictor in the given class. Such bounds provide strong robustness guarantees, but they might not be informative about the behavior of predictive algorithms for realistic DGPs. By focusing on the risk function, we obtain a more fine-grained characterization: the risk function tells us how expected predictive performance depends on the DGP, while worst-case regret bounds only characterize performance for the least favorable DGP.

Our asymptotic characterization relates the risk function of regularized ERM estimators tuned using CV to the risk function of the James-Stein (JS) shrinkage estimator \cite{james1961estimation}, and generalizations thereof.
JS shrinkage famously dominates maximum likelihood estimation (MLE) in the normal means setting: The risk function (mean squared error) of JS shrinkage, as characterized in \cite{stein1981estimation}, is lower than the risk of MLE, for any possible DGP. Figure~\ref{fig:JS-risk} illustrates. Our result suggests that this same risk improvement carries over, asymptotically, to CV-tuned penalized estimation in general parametric models. Like CV, SURE provides an unbiased estimate of mean squared error; the JS estimator approximately minimizes SURE over the shrinkage intensity.
\begin{figure}
    \caption{Risk function for JS-shrinkage, dimension 10}
    \label{fig:JS-risk}
    \begin{center}
    \begin{tikzpicture}
        \begin{axis}[
            xlabel={$\lVert \theta \rVert$},
            ylabel={$MSE$},
            xmin=0, xmax=6.5,
            ymin=0, ymax=1.1,
            xtick={1,2,3,4,5,6},
            width=7cm,
            height=5cm,
            axis lines=middle,
            xlabel style={below right},
            ylabel style={at={(axis description cs:-.13,0.5)}, anchor=south, rotate=90},
            legend pos=south east
        ]
        \addplot[mark=none, dashed] coordinates {
            (0, 1)
            (6, 1)
        };
        \addlegendentry{MLE};
        \addplot[mark=none,
        thick] coordinates {
            (0, 0.19797817448257113)
            (0.1, 0.2054582170314164)
            (0.2, 0.22829026983210313)
            (0.3, 0.26421634843194347)
            (0.4, 0.3099541401775098)
            (0.5, 0.3621252122332097)
            (0.6, 0.4174453765812922)
            (0.7, 0.4729872281492919)
            (0.8, 0.5264425242683316)
            (0.9, 0.5762522186927894)
            (1.0, 0.6215656856465035)
            (1.1, 0.6620956218921872)
            (1.2, 0.6979425114595172)
            (1.3, 0.7294329272061092)
            (1.4, 0.7569979113845149)
            (1.5, 0.7810957908666116)
            (1.6, 0.8021691945998736)
            (1.7, 0.820623915309514)
            (1.8, 0.8368205993035923)
            (1.9, 0.851073622847667)
            (2.0, 0.8636537353420901)
            (2.1, 0.8747924235899378)
            (2.2, 0.8846868172898379)
            (2.3, 0.8935044991617345)
            (2.4, 0.9013879119287014)
            (2.5, 0.9084582433406777)
            (2.6, 0.9148187724116754)
            (2.7, 0.9205577105437802)
            (2.8, 0.92575059262033)
            (2.9, 0.930462278908653)
            (3.0, 0.9347486264858748)
            (3.1, 0.9386578832298975)
            (3.2, 0.9422318505406067)
            (3.3, 0.9455068540717296)
            (3.4, 0.9485145554238454)
            (3.5, 0.9512826321915071)
            (3.6, 0.9538353490106405)
            (3.7, 0.9561940382675428)
            (3.8, 0.958377505822497)
            (3.9, 0.9604023743737454)
            (4.0, 0.9622833748491301)
            (4.1, 0.9640335943800212)
            (4.2, 0.965664687913424)
            (4.3, 0.9671870592926904)
            (4.4, 0.9686100166347064)
            (4.5, 0.9699419060103076)
            (4.6, 0.9711902267611615)
            (4.7, 0.9723617312329309)
            (4.8, 0.9734625112488388)
            (4.9, 0.9744980732717461)
            (5.0, 0.9754734038918732)
            (5.1, 0.9763930270194633)
            (5.2, 0.9772610539474514)
            (5.3, 0.9780812272706952)
            (5.4, 0.9788569594992685)
            (5.5, 0.9795913670785175)
            (5.6, 0.9802873004238419)
            (5.7, 0.9809473704900459)
            (5.8, 0.9815739723208001)
            (5.9, 0.9821693059609183)
            (6.0, 0.9827353950609422)
        };
        \addlegendentry{JS-shrinkage}
        \end{axis}
    \end{tikzpicture}
    \end{center}
\end{figure}

\paragraph{Main result}
Our main theorem states that the distribution of the out-of-sample prediction loss of CV-tuned regularized ERM estimators converges to the distribution of the squared-error loss of the corresponding SURE-tuned shrinkage estimator in the Gaussian limit experiment. In particular, the risk function (expected out-of-sample prediction loss, as a function of the true parameter) converges to the mean squared error of SURE-tuned penalized estimation in the normal means model. Two of our intermediate results are of independent interest: the uniform approximation of $n$-fold CV by SURE, and the generic well-separation of the global minimum of SURE. 

\paragraph{Key steps}
Let us briefly outline the three main parts of our proof.
First, we show that in large samples ERM estimators are approximately normally distributed, and that out-of-sample predictive loss is approximately equal to squared error loss. These are standard results, and we follow \cite{van2000asymptotic} in proving this step. We further demonstrate that this approximation carries over to penalized estimators, for fixed tuning parameters. 
Our asymptotic approximations are based on local-to-0 asymptotics:
As sample size $n$ increases, the parameter vector drifts to $0$ (which is the minimizer of the penalty term) at a rate of $1/\sqrt{n}$. This rate is such that both bias and variance remain non-negligible in the limit. The drifting-parameter framework is  natural for studying regularized estimators, because it is the regime in which the penalty has a first-order effect: if the true parameter were fixed away from~$0$, the penalty would become asymptotically irrelevant, while if it were exactly~$0$, no bias--variance tradeoff would arise.

Second, we show that n-fold CV, interpreted as a random function that maps tuning parameter values to estimates of predictive loss, converges uniformly to SURE. In this step of the proof, we build on the prior work of \cite{wilsonkasymackey2018}. This step involves an influence function approximation for leave-one-out estimators, and a second-order approximation to predictive risk. Uniformity of convergence over the space of tuning parameters is key to this step, and we need to carefully specify regularity conditions such that uniformity is guaranteed.

Third, we show that convergence of the CV criterion function for tuning is sufficient for convergence of its minimizer, the tuned parameter. This is non-trivial, because both CV and SURE typically have multiple local minima, and might have multiple global minima. We need to show that generically the global minimum is well-separated. We do this using separate arguments for $L^{1}$ and $L^{2}$ penalties, characterizing the shape and behavior of SURE in either case.

We should emphasize some limitations of our analysis: We do not consider penalties beyond $L^1$ and $L^2$, or $k$-fold CV with $k < n$; extending our results in these directions is left for future work. Our asymptotic approximations are furthermore not appropriate in the over-parametrized regime, when the number of parameters is of similar or larger magnitude than the number of observations.

\paragraph{Literature}
The analysis in this paper connects several lines of work in statistics, econometrics, and machine learning.
Leave-one-out cross-validation was formalized by \cite{Stone1974}; its asymptotic optimality for model selection was established by \cite{Li1987}. A closely related family of risk estimators includes Mallows' $C_p$ \citep{Mallows1973}, generalized cross-validation \citep{Golub1979}, and Stein's unbiased risk estimate \citep{stein1981estimation}. \cite{efron2004estimation} provides a unifying perspective, showing that these criteria all take the form of in-sample error plus a covariance penalty.
\cite{arlot2010survey} provide a comprehensive survey of cross-validation procedures and their theoretical properties. A noteworthy contrast with our results is provided by \cite{shao1993linear}, who shows that leave-one-out CV is asymptotically \emph{inconsistent} for model selection---it selects overfitted models with positive probability---while $k$-fold CV with $k = o(n)$ is consistent. Our result is complementary in focus: rather than asking which of a finite list of models has the best predictive ability, we characterize the limiting distribution and risk function of the estimator selected by leave-one-out CV, showing it converges to that of the SURE-tuned normal-means estimator.

Shrinkage estimation originates with \cite{james1961estimation} and was analyzed in depth by \cite{stein1981estimation}. SURE-based tuning of shrinkage was extended to wavelet thresholding by \cite{Donoho1995}. The two penalty families we study---Ridge \citep{Hoerl1970} and Lasso \citep{tibshirani1996regression}---are the most widely used forms of regularization in practice, and both are commonly tuned by cross-validation.
The close relationship between leave-one-out CV and covariance-penalty criteria such as SURE is further illuminated by \cite{zou2007degrees}, who show---using SURE as the analytical tool---that the effective degrees of freedom of the Lasso equals the number of nonzero fitted coefficients. 

Using local asymptotic frameworks to characterize decision problems was pioneered by \cite{le1972limits}; work using this approach is reviewed in \cite{hirano2020asymptotic}. The use of shrinkage asymptotics for parametric models in econometrics is discussed in \cite{hansen2016efficient}. We build directly on \cite{wilsonkasymackey2018}, who provide non-asymptotic deterministic guarantees for approximate cross-validation, showing that leave-one-out estimators can be well approximated by a single Newton step from the full-sample estimator.

\paragraph{Roadmap}
The remainder of this paper is structured as follows:
In Section \ref{sec:setup}, we introduce our model and assumptions, and define all relevant notation.
In Section \ref{sec:outline}, we first provide a heuristic outline of our proof, and then state a series of intermediate lemmas.
In Section~\ref{sec:numerical}, we illustrate our convergence results with several numerical experiments. 
Section~\ref{sec:conclusion} concludes.
Proofs of all lemmas are collected in the appendix.
In Appendix \ref{sec:verify_examples}, we verify the regularity conditions for leading examples.
In Appendix \ref{sec:influencefunctions}, we prove the lemmas corresponding to the first part of our argument, involving influence function approximations and asymptotic normality.
In Appendix \ref{sec:cvconvergence}, we prove the second part of the argument, namely the (uniform) approximation of n-fold CV by SURE.
In Appendices \ref{sec:ridge} and \ref{sec:lasso}, we show that the global minimizer of SURE is generically well-separated for both $L^{2}$ and $L^{1}$ penalties, thereby proving the third part.
In Appendix \ref{sec:risk}, we conclude our derivation, proving the convergence of risk for tuned estimators.

\section{Setup}
\label{sec:setup}

In the following, we first set up our estimators, and their asymptotic counterparts, in a series of definitions.
We then state the assumptions that will be invoked to justify our asymptotic approximations.

Throughout this paper, we consider the problem of estimating a parameter vector $\beta_n$ which is defined as the minimizer of expected loss $E[l(\beta, Z_n^1)]$. For each $n$, the random vectors $Z_n^i$ are i.i.d. draws from the distribution $\mu_n$, across $i$.
We use local-to-0 coordinates, where $\beta_n = \theta_0 / \sqrt{n}$ for sample size $n$, and we assume that $\theta_0$ remains fixed across $n$.
Any finite-sample object will be denoted by a subscript $n$; objects without subscripts correspond to the limiting experiment.

For prediction problems, typically $Z_n^i=(W_n^i,Y_n^i)$, for predictive features $W_n^i$ and outcomes $Y_n^i$.
Examples include linear OLS regression, where $l(\beta, Z_n^i) = (Y_n^i- W_n^i \cdot \beta)^2$, as well as the use of neural nets\footnote{With a small number of parameters relative to the sample size.} or other parametric models for classification, where $l(\beta, Z_n^i) = -\log(f(Y_n^i|W_n^i, \beta))$, and $f(Y_n^i|W_n^i,\beta)$ is the probability assigned to outcome $Y_n^i$ by the model.

\subsection{Definitions}

Definition~\ref{def:loss} introduces notation for loss functions and for limiting loss functions.
We evaluate estimates of $\theta$ in terms of their expected loss $\bar L_n(\theta, \theta_0)$. For supervised learning, $\bar L_n(\theta, \theta_0)$ is the \textit{out-of-sample} expected prediction error.

\begin{definition}[Loss, empirical loss, and expected loss]$\;$\\
  \label{def:loss}
  Given the loss function $l(\beta, z)$, define the following.
  \bals
  l_n(\theta, z) &= l(\theta / \sqrt{n} ,z) & \text{Loss function in local parameter}\\
  L_n(\theta) &= \sum_{i=1}^n l_n(\theta,Z_n^i) & \text{Empirical loss}\\
  \bar L_n(\theta, \theta_0) &=  E\left[ L_n(\theta) - L_n(\theta_0)\right] & \text{Expected loss}\\
  \bar L(\theta, \theta_0) &= \lim_{n\rightarrow \infty}  \bar L_n(\theta, \theta_0). & \text{Limiting expected loss}
  \eals
  We assume that the sequence of distributions $\mu_n$ is such that the limit in the last definition is well-defined.
\end{definition}

\paragraph{Scaling}
Note that we could have equivalently defined 
$$\bar L_n(\theta, \theta_0) =  n\cdot E\left[ l_n(\theta,Z_n^{n+1}) - l_n(\theta_0,Z_n^{n+1})\right],$$ that is, $\bar L_n(\theta, \theta_0)$ is the expected regret for out-of-sample predictions, multiplied by the sample size $n$.
The multiplication by $n$ is required because in Definition~\ref{def:loss}, we do \textit{not} scale empirical loss $L_n(\theta)$ by a factor $\tfrac1n$, and therefore $L_n(\theta)$ diverges. However, in the definition of the local parameter vector $\theta$ we have re-scaled the parameter vector $\beta$ by a factor of $\tfrac1{\sqrt{n}}$.
This implies that the Hessian (second derivative) of $L_n(\theta)$ with respect to $\theta$, if it exists, is given by
$$
\nabla_\theta^2  L_n(\theta) = \frac1n \sum_{i=1}^n \nabla_\beta^2 l(\theta/\sqrt{n},Z_n^i).
$$
We can thus expect that this Hessian converges, by a law of large numbers, under some additional regularity conditions.

\paragraph{Estimators of $\theta$ and their asymptotic counterparts}
We next specify a series of estimators, in Definition~\ref{def:estimators}.
We start with the standard ERM estimator $\hat \theta_n = \argmin_\theta L_n(\theta)$, and its  limiting counterpart $\hat \theta$, where the latter is normally distributed with mean $\theta_0$ and variance $\Sigma$.
We then consider the regularized versions of these estimators, that is, the penalized ERM estimator with penalty $\lambda  \cdot \pi(\theta)$, and its limiting counterpart.
Our choice of local parametrization ensures that a constant value of $\lambda$ along the sequence indexed by $n$ leads to a non-degenerate limit for the penalized ERM estimator, where neither variance nor bias of this estimator vanish.

In Definition~\ref{def:estimators}, we furthermore introduce leave-one-out loss, and the corresponding leave-one-out ERM and penalized ERM estimator. These will serve as the building blocks of n-fold cross-validation.
\begin{definition}[Estimators of $\theta_0$]
  \label{def:estimators}
We will consider the following estimators of $\theta_0$, for finite $n$, and in the limit experiment:
\bals
  \hat \theta_n &= \argmin_\theta L_n(\theta) & \text{ERM estimator}\\
  \hat \theta &\sim N(\theta_0, \Sigma) & \text{Limiting ERM estimator}\\
  \hat \theta_n^\lambda &= \argmin_\theta \left[L_n(\theta) + \lambda  \cdot \pi(\theta)\right] & \text{Penalized ERM estimator}\\
  \hat \theta^\lambda &= \argmin_\theta \left[\tfrac12\|\theta - \hat \theta\|^2 + \lambda  \cdot \pi(\theta)\right],  &\text{Limiting penalized ERM estimator}
\eals
where $\pi( \cdot )$ is convex and attains its minimum at $0$.\\
We furthermore consider the following leave-one-out (LOO) loss and estimators of $\theta_0$:
\bals
L_n^{-i}(\theta) &= \sum_{j\neq i} l_n(\theta,Z_n^j) & \text{LOO empirical loss}\\
\hat \theta_n^{-i} &= \argmin_\theta L_n^{-i}(\theta) & \text{LOO ERM estimator}\\
\hat \theta^{\lambda,-i}_n &= \argmin_\theta \left[L_n^{-i}(\theta) + \lambda  \cdot \pi(\theta)\right]. &\text{LOO penalized ERM estimator}
\eals
\end{definition}

\noindent We can rewrite the limiting penalized ERM estimator as
$$
\hat \theta^\lambda = \hat \theta + g^\lambda(\hat \theta),
$$
where
\bals
  g^\lambda(\theta) = \argmin_g \tfrac12\|g\|^2 + \lambda  \cdot \pi(\theta + g).
\eals
Denote $\nabla g^\lambda(\theta)$ the derivative of $g^\lambda(\theta)$ where it exists, and define $\nabla g^\lambda(\theta) = 0$ at points where $g^\lambda(\theta)$ is not differentiable.\footnote{This convention is adopted for convenience; we will use it to handle Lasso ($L^1$) penalties in the proof of Lemma~\ref{lem:convergence_cv} below.}

\paragraph{Estimators of risk}
The preceding definition introduced penalized estimators for given, fixed values of the tuning parameter $\lambda$.
We are interested in estimators which choose this tuning parameter in a data-dependent way, where $\lambda$ minimizes an estimator of risk.
For finite sample size $n$, we consider the n-fold crossvalidation (CV) criterion as an estimator of the risk of penalized ERM estimation.
For the limit experiment, we consider Stein's Unbiased Risk Estimator (SURE) as an estimator of the risk of the penalized limiting ERM estimator.
\begin{definition}[Estimators of risk]
  We consider the following estimators of risk for penalized ERM estimators with fixed tuning parameter $\lambda$.
  \bals
  CV_n(\lambda) &=\sum_i l_n(\hat \theta^{\lambda,-i}_n,Z_n^i), &\text{n-fold CV}\\
  SURE(\lambda, \hat \theta, \Sigma) &=  \tfrac12 \left[ \trace(\Sigma) + \|g^\lambda(\hat \theta)\|^2 + 2  \trace \left ( \nabla g^\lambda(\hat \theta)\cdot \Sigma\right )\right]. &\text{SURE}
  \eals
\end{definition}

\paragraph{Tuned estimators}
We can now formally define our tuned estimators. $\hat \theta_n^{*}$ is the penalized ERM estimator using a tuning parameter $\lambda_n^*$ which minimizes the n-fold CV estimator of risk.
$\hat \theta^{*}$ is the penalized limiting ERM estimator using a tuning parameter $\lambda^*$ which minimizes the SURE estimator of risk.
The tuning parameter is chosen from a set $\Lambda \subset \mathbb R$. Later, we will consider $\Lambda = \mathbb R^+ \cup \{\infty\}$ (for Ridge penalties), and $\Lambda$ arbitrary but finite (for Lasso penalties).
\begin{definition}[Tuned estimators of $\theta$]
  \bals
  \hat \theta_n^{*} & = \hat \theta_n^{\lambda_n^*}, & \text{Penalized ERM tuned using CV}\\
  &\lambda_n^* = \argmin_{\lambda \in \Lambda} CV_n(\lambda) \\
  \hat \theta^{*} &= \hat \theta^{\lambda^*}, & \text{Limiting penalized ERM tuned using SURE}\\
  &\lambda^* = \argmin_{\lambda \in \Lambda} SURE(\lambda, \hat \theta, \Sigma).
  \eals
\end{definition}

We evaluate estimators based on their expected loss for new data-points. For supervised learning, this corresponds to the out-of-sample expected prediction loss.
In this paper, we do not consider global criteria such as worst-case risk (maximizing over $\theta_0$) or Bayes risk (averaging over a prior distribution for $\theta_0$).
Instead, we are interested in the dependence of expected loss on the parameter $\theta_0$, which is captured by the risk function. Our notation makes this dependence on $\theta_0$ explicit. The risk function gives a more fine-grained picture of estimator performance, relative to global criteria such as worst-case risk, Bayes risk, or worst-case regret.

The parameter $\theta_0$ enters the following expressions both directly, as an argument of $\bar L_n$ and $\bar L$, and implicitly, via the distribution of $Z$ that the expectations are averaging over.
\begin{definition}[Risk functions]
\bals
  R_n(\theta_0) &= E\left[\bar L_n(\hat \theta_n^{*}, \theta_0) \right] &
  \text{Finite sample risk, tuned using CV}\\  
  R(\theta_0) &= E\left[\bar L(\hat \theta^{*}, \theta_0)\right] &
  \text{Limiting risk, tuned using SURE}.
\eals
\end{definition}

\subsection{Assumptions}

Having defined our estimators and evaluation criteria, we next specify the assumptions invoked in our asymptotic analysis.
Assumption~\ref{as:sequence} sets up a sequence of experiments, indexed by $n$. 
We assume that, for each $n$, the minimizer of expected loss $E[l(\beta, Z_n^i)]$ is given by $\theta_0/\sqrt{n}$. Put differently, the minimizer $\beta_n$ of expected loss drifts towards $0$. 
We furthermore assume that the variance $\Sigma$ of the score $\nabla_\beta l(\theta_0/\sqrt{n}, Z_n^i)$ remains constant along our sequence.
\begin{assumption}[Sequence of experiments]
  \label{as:sequence}
  For each $n$, the random vectors $Z_n^i$ are i.i.d. draws from the distribution $\mu_n$, across $i$.
  The distributions $\mu_n$ are such that $\theta_0$ and $\Sigma$ do not depend on $n$, where
  \bals
    \theta_0 &=  \argmin_\theta E[l(\theta /\sqrt{n}, Z_n^i)],\\
    \Sigma &= \Var\left(\nabla_\beta l(\theta_0/\sqrt{n}, Z_n^i)\right).
  \eals
In particular, the minimizer $\beta_n = \theta_0/\sqrt{n}$ of expected loss drifts to $0$ at rate $1/\sqrt{n}$.
We further assume that $\Sigma$ is positive definite.
\end{assumption}

The limiting Hessian $H = \nabla_\theta^2 \bar L(\theta, \theta_0)$ is typically non-degenerate because of our scaling of $\bar L_n(\theta, \theta_0)$ and of $\theta$. 
The following Assumption~\ref{as:smoothloss} is made for notational convenience. This assumption states that the Hessian $H $ is equal to the identity $I$. This is a coordinate normalization that can be imposed without loss of generality.\footnote{By suitable choice of coordinates we can normalize \textit{either} $H$, \textit{or} the asymptotic variance $\Sigma$ of Assumption~\ref{as:sequence}, \textit{or} the Hessian of the penalty function $\pi$ (when the latter exists),  {but not more than one of these three matrices}, in general.
After normalizing the Hessian, we can however diagonalize one more matrix, without loss of generality.}
\begin{assumption}[Normalized loss function]
  \label{as:smoothloss}
  $$\nabla_\theta^2 \bar L(\theta, \theta_0)|_{\theta = \theta_0} = I.$$
\end{assumption}

The last part of our proof requires showing that the global optimum of $SURE$ with respect to $\lambda$ is generically unique and well-separated.
We will prove this fact for both Ridge and Lasso penalties, using separate arguments for either case.
\begin{assumption}[Penalty function and grid for tuning]$\;$\\
  \label{as:ridgelasso}
  The penalty $\pi(\theta)$ and the set $\Lambda$ take one of the following two forms:
  \begin{enumerate}
    \item \textbf{Ridge}: $\pi(\theta) = \tfrac12 \theta  \cdot A^{-1}  \cdot \theta$, where $A$ is positive definite,\\
    and $\Lambda = \mathbb R^+ \cup \{\infty\}$.
    \item \textbf{Lasso}: $\pi(\theta) = \|A^{-1}  \cdot \theta\|_1$, where $A$ is an invertible matrix,\\ and $\Lambda \subset \mathbb R^+$ is finite.
  \end{enumerate}
\end{assumption}

The remaining assumptions state regularity conditions. The first item in Assumption~\ref{as:ERM_conditions} is a condition on the loss function which allows us to invoke results from empirical process theory. Similar assumptions are invoked in \cite{van2000asymptotic} when deriving the properties of M-estimators.
The second item in Assumption~\ref{as:ERM_conditions} is a weak high-level condition ruling out divergence of ERM estimators, which ensures the applicability of empirical process results.
\begin{assumption}[Conditions for convergence of the ERM estimator]$\:$
  \label{as:ERM_conditions}
  \begin{enumerate}
    \item \textbf{Lipschitz loss}\\
    The loss function $l(\beta, z)$ satisfies
    $$|l(\beta_1, z) - l(\beta_2, z)| \leq m(z)  \cdot \|\beta_1 -\beta_2\|, $$
    for all $\beta_1, \beta_2$ in a neighborhood of $0$, where $\sup_n \Var(m(Z_n^i)) \leq \infty$.
    Furthermore, $l(\beta, Z_n^i)$ is differentiable w.r.t. $\beta$, for all $\beta$ in a neighborhood of $\beta=0$, with probability $1$.
  
  \item \textbf{Stochastically bounded ERM estimator}\\
    The sequence $\hat \theta_n = \argmin_\theta L_n(\theta)$ is bounded in probability.
  \end{enumerate}
\end{assumption}

In order to demonstrate (uniform) convergence of $CV_n$ to $SURE$, we impose the following additional regularity conditions.
\begin{assumption}[Conditions for convergence of CV] $\;$
  \label{as:loss_fn_bounds}
 \begin{enumerate}
  \item \textbf{Conditions on loss}\\
  There exist $\mu>0,\nu<\infty$ independent of $n$
  such that $L_{n}(\theta)$ is $\mu$-strongly convex and has $\nu$-smooth
  Hessians with probability approaching 1 under $\mu_{n}$.\footnote{$L_{n}(\theta)$ is $\mu$-strongly convex if $\nabla^2 L_{n}(\theta) - \mu I$ is positive semi-definite for all $\theta$. A function is $L$-smooth if its gradients are Lipschitz continuous with Lipschitz constant $L$.} 

  \item \textbf{Conditions on scores}\\
  The function $\sqrt{n}\nabla_{\theta}l_{n}(\theta,Z_{n}^{i})$
  is Lipschitz continuous almost everywhere, i.e., there exists $B_{n}(Z_{n}^{i})$
  such that 
  \[
  \left\Vert \sqrt{n}\nabla_{\theta}l_{n}(\theta,Z_{n}^{i})-\sqrt{n}\nabla_{\theta}l_{n}(\theta^{\prime},Z_{n}^{i})\right\Vert \le B_{n}(Z_{n}^{i})\left\Vert \theta-\theta^{\prime}\right\Vert \ \forall\theta,\theta^{\prime}\in\Theta
  \]
  and $E_{\mu_{n}}\left[\left\Vert B_{n}(Z_{n}^{i})\right\Vert ^{2}\right]<\infty$.
  Additionally, there exists $M<\infty$ independent of $n$ such that
  \[
  \mathbb{E}_{\mu_{n}}\left[\left\Vert \sqrt{n}\nabla_{\theta}l_{n}\left(\theta_{0},Z_{n}^{i}\right)\right\Vert ^{4}\right]\le M.
  \]
  
  \item \textbf{Conditions on Hessians}\\
  The Hessian $\nabla_{\theta}^{2}l_{n}(\theta,Z_{n}^{i})$ is
  such that 
  \[
  \frac{1}{n}\sum_{i=1}^{n}\left\Vert n\nabla_{\theta}^{2}l_{n}\left(\theta_{0},Z_{n}^{i}\right)\right\Vert ^{2}=O_{\mu_{n}}(1).
  \]
  Furthermore, there exists $C_{n}(Z_{n}^{i})$ such that 
  \[
  \left\Vert n\nabla_{\theta}^{2}l_{n}(\theta,Z_{n}^{i})-n\nabla_{\theta}^{2}l_{n}(\theta^{\prime},Z_{n}^{i})\right\Vert \le C_{n}(Z_{n}^{i})\left\Vert \theta-\theta^{\prime}\right\Vert \ \forall\theta,\theta^{\prime}\in\Theta
  \]
  and $\sup_n E_{\mu_{n}}\left[C_{n}(Z_{n}^{i})^{2}\right]<\infty$.
  
 \item \textbf{Conditions on Fourth Derivatives}\\
 The fourth derivative tensor, $D_\theta^4 l_{n}(\theta,Z_{n}^{i})$, of $l_n(\cdot, Z_n^i)$ is such that   
 \[
\sup_n \mathbb{E}_{\mu_{n}}\left[\sup_{\theta\in\Theta}\left\Vert n^{2}D_\theta^{4}l_{n}\left(\theta,Z_{n}^{i}\right)\right\Vert ^{4}\right]\le M,
\]
for some $M <\infty$.
 
 \end{enumerate}
 \end{assumption}

In \autoref{sec:verify_examples}, we verify these assumptions for standard linear regression and generalized linear models.

\section{Main result and intermediate lemmas}
\label{sec:outline}

Our main goal in this paper is to prove the following result:
\begin{theorem}
  \label{theo:risk_convergence}
  $$
\bar L_n(\hat \theta_n^*, \theta_0) \rightarrow_d \tfrac12 \|\hat \theta^* - \theta_0\|^2.
$$
\end{theorem}
In words, the distribution of loss of the penalized ERM estimator tuned using n-fold cross-validation converges to the distribution of squared error of the corresponding shrinkage estimator in the normal means model, tuned by Stein's Unbiased Risk Estimate.
A special case of these limiting estimators are James-Stein shrinkage estimators, for which closed form characterizations of the risk function are known \citep{stein1981estimation}.
An immediate corollary of Theorem~\ref{theo:risk_convergence} is the convergence of risk functions, subject to possible truncation of tail events.\footnote{Truncation is necessary because, even for estimators such as linear OLS regression, risk is typically undefined, since some eigenvalues of the design matrix might be close to 0, so that the moments of $\hat\theta_n$ might not exist.}
\begin{corollary}
  \label{cor:risk_convergence}
Let $M>0$. Then
$$
E\left[ \min\left(\bar L_n(\hat \theta_n^*,\theta_0), M \right) \right] \rightarrow
E\left[ \min\left(\tfrac12 \|\hat \theta^* - \theta_0\|^2,M\right) \right].
$$
\end{corollary}

\subsection{Outline of proof}
We will build up our argument that proves Theorem~\ref{theo:risk_convergence} in a series of lemmas.
Before doing so, however, we first provide an intuitive sketch of our argument, while neglecting remainder terms. Subsequently, we will prove that these remainder terms are indeed asymptotically negligible.

\paragraph{Influence function approximation}
We start by noting that empirical risk is asymptotically equivalent to quadratic error loss relative to the sample mean $\tilde \theta_n$,
\bals
L_n(\theta) &\approx const.
+ \tfrac12 \|\theta-\tilde \theta_n\|^2,
&\tilde \theta_n &= \theta_0 + \tfrac{1}{\sqrt{n}} \sum_i X_n^i,& 
\eals
where $$X_n^i = -\nabla_\beta l(\theta_0/\sqrt{n}, Z_n^i)$$ is the influence function.
Recall that we have normalized the Hessian in Assumption~\ref{as:smoothloss}, which simplifies the expression for $\tilde \theta_n$.
This approximation of empirical risk immediately implies an asymptotic linear approximation of the empirical risk minimization (ERM) estimator, $\hat \theta_n \approx \tilde \theta_n$.
These are standard approximations that deliver asymptotic normality of ERM estimators; see for instance Theorem 5.21 in \cite{van2000asymptotic}.

We then get the corresponding approximation for the penalized ERM estimator, for fixed $\lambda$,
\bals
\hat \theta_n^\lambda &\approx \tilde \theta_n^{\lambda} = \tilde \theta_n + g^\lambda(\tilde \theta_n),
\eals
where we recall the definition $g^\lambda(\theta) = \argmin_g \tfrac12\|g\|^2 + \lambda  \cdot \pi(\theta + g)$.

\paragraph{Convergence of CV to SURE}
An analogous approximation holds for leave-one-out (LOO) loss. 
To obtain the LOO sample mean, the influence function $\tfrac{1}{\sqrt{n}} X_n^i$ is subtracted from the sample mean $\tilde \theta_n$, which gives
\bals
  L_n^{-i}(\theta) &\approx const.
  + \tfrac12 \|\theta-\tilde \theta_n^{-i}\|^2,&
  \textrm{where } \tilde \theta_n^{-i} &= \tilde \theta_n - \tfrac{1}{\sqrt{n}} X_n^i.
\eals
The penalized LOO estimator is then approximately given by
$$
\hat \theta^{\lambda,-i}_n \approx \tilde \theta_n^{-i} + g^\lambda(\tilde \theta_n^{-i}) \approx \tilde \theta_n^{\lambda} - \tfrac{1}{\sqrt{n}}(I+\nabla g^{\lambda}(\tilde \theta_n)) \cdot X_n^i.
$$
In the last step we have replaced $g^\lambda$ by its first-order Taylor approximation around $\tilde \theta_n$, at points $\tilde \theta_n$ where $g^\lambda$ is differentiable. 
(This approximation won't hold at kink-points of $g^\lambda$, which exist for Lasso penalties, in particular.)

The n-fold cross-validation estimator of the risk of $\hat \theta^{\lambda}_n$ can be approximated by
$$
  CV_n(\lambda) =\sum_i l_n(\hat \theta^{\lambda,-i}_n,Z_n^i)
  \approx const. + \tfrac{1}{2n} \sum_i  \|\hat \theta^{\lambda,-i}_n - \theta_0 - \sqrt{n} X_n^i\|^2.
$$
When we take this expression, plug in the approximate form of $\hat \theta^{\lambda,-i}_n$, multiply out the inner products, and omit terms which do not depend on $\lambda$, we obtain
\bals
CV_n(\lambda)
&\approx const. + \tfrac{1}{2n} \sum_i  \|\underbrace{\tilde \theta_n + g^\lambda(\tilde \theta_n) - \tfrac{1}{\sqrt{n}}(I+\nabla g^{\lambda}(\tilde \theta_n)) \cdot X_n^i}_{\approx\hat \theta^{\lambda,-i}_n } - \theta_0 - \sqrt{n} X_n^i\|^2\\
&\approx const.+\tfrac{1}{2n} \sum_i \|g^\lambda(\tilde \theta_n)\|^2
+ \tfrac1n \sum_i \langle \nabla g^{\lambda}(\tilde \theta_n) \cdot X_n^i, X_n^i \rangle\\
&\approx const. + \tfrac12\| g^\lambda(\hat \theta_n)\|^2  +
\trace  ( \nabla g^\lambda(\hat \theta_n)\cdot \hat\Sigma_n )\\
 &= const. + SURE(\lambda, \hat \theta_n, \hat \Sigma_n)\\
  &\approx const. + SURE(\lambda, \hat \theta_n, \Sigma),
\eals
where $\hat\Sigma_n$ is the sample second moment of $X_n$.
In the second line, $const.$ subsumes any terms that do not depend on $\lambda$, while the approximation omits terms that depend on $\lambda$ but are of order $1/\sqrt{n}$.
This approximation to $CV_n$ has the form of Stein's Unbiased Risk Estimate.
The first term in this approximation to $CV_n(\lambda)$ corresponds to the average in-sample error, the second term has the form of a covariance penalty \citep{efron2004estimation}.

\paragraph{Convergence of tuning parameter and tuned estimators}
We will need to show that this approximation is uniformly valid in $\lambda$.
We will furthermore need to show that uniform proximity of $CV_n$ to $SURE$ is enough to guarantee proximity of the corresponding optimized tuning parameters,
$$
  \argmin_\lambda CV_n(\lambda) \approx \argmin_\lambda SURE(\lambda, \hat \theta_n, \hat \Sigma_n).
$$
\begin{figure}[t]
  \caption{Examples of multi-modality of $SURE$}
  \label{fig:multimodality}
  \vspace{5pt}
  \begin{minipage}{0.5\textwidth}
    \centering
    Ridge\\
  \begin{tikzpicture}
    \begin{axis}[
        xlabel={$\lambda$},
        ylabel={SURE},
        axis x line=middle,
        axis y line=left,
        axis line style={-latex},
        width=5cm,
        height=5cm,
        ymin=1.65,
        ymax=2,
        xmin = -5,
        xlabel style={below right},
    ]
    \addplot[
        domain=0:49.8,
        samples=250,
        thick
    ] 
    {.5*((1 - 1 / (1 + x))^2 * 1.3893^2 + 2 * (1 / (1 + x)) +
     (1 - 40 / (40 + x))^2 * 1.5^2 + 2 * (40 / (40 + x)))};
    \end{axis}
\end{tikzpicture}
\end{minipage}%
\begin{minipage}{0.5\textwidth}
  \centering
  Lasso\\
\begin{tikzpicture}
    \begin{axis}[
        xlabel={$\lambda$},
        ylabel={SURE},
        axis x line=middle,
        axis y line=left,
        axis line style={-latex},
        width=5cm,
        height=5cm,
        ymin=1.8,
        ymax=4.2,
        xmin = -.1,
        xlabel style={below right},
    ]
    \addplot[
        domain=0:2.45,
        samples=491,
        thick
    ] 
    {.5*min(abs(x), 1/sqrt(8))^2 + (abs(x) <= 1/sqrt(8)) +
     .5*min(abs(x), 3/sqrt(8))^2 + (abs(x) <= 3/sqrt(8)) +
     .5*min(abs(x), 2)^2 + (abs(x) <= 2)};
    \end{axis}
\end{tikzpicture}
\end{minipage}
  \textit{Notes:} These plots show examples of multi-modality for SURE, for the case of $L^2$ penalties (Ridge) and $L^1$ penalties (Lasso). Both examples are reproduced from \cite{wilsonkasymackey2018}.
  \end{figure}
This latter step is non-trivial, because the criterion function $SURE(\lambda, \hat \theta_n, \hat \Sigma_n)$ typically has multiple local minima. For certain values of $\hat \theta_n$ this function furthermore has multiple \textit{global} minima. 
When $SURE$ has multiple global (near-)minima, uniform closeness of $CV$ to $SURE$ is not enough to ensure closeness of the minimizer of $CV$ to the minimizer of $SURE$.
The plots in Figure \ref{fig:multimodality}, which are reproduced from \cite{wilsonkasymackey2018},\footnote{The numerical values corresponding to these examples are as follows: 
(a) SURE for Ridge: $\hat \theta = (1.3893, 1.5)$, $L(\theta) =(\theta - \hat \theta)\diag(1,40)(\theta - \hat \theta)$, $\pi(\theta) =  \|\theta\|^2$,
(b) SURE for Lasso: $\hat \theta = \frac{1}{\sqrt{n}}(\sqrt{1/8}, \sqrt{9/8}, 2 )$, $\pi(\theta) = \sum |\theta_j|$.
} illustrate two numerical examples (realizations of $\hat \theta_n$ and of $\hat \Sigma_n$) for which $SURE$ indeed has multiple global minima.

Using separate arguments for Ridge (Appendix \ref{sec:ridge}) and Lasso (Appendix \ref{sec:lasso}), we will prove, however, that the global minimum of $SURE$ with respect to $\lambda$ is unique and well separated almost everywhere, in a suitable sense. Put differently, cases such as those represented in Figure \ref{fig:multimodality} are non-generic, such that they do not lead to a breakdown of convergence for the optimized tuning parameter.
The arguments proving that multiple global minima only occur on a set of measure $0$ are the most non-standard part of our proof

Well-separation ensures that the $\argmin$ functional is continuous at almost every realization of $\hat\theta$.
From these results we thus conclude that the mapping from $\hat \theta_n$ and $CV_n$ to the tuned estimate $\hat \theta_n^*$ is almost everywhere continuous. This allows us to invoke the continuous mapping and dominated convergence theorems, and to conclude the proof of Theorem~\ref{theo:risk_convergence}.

\subsection{Intermediate lemmas}

Let us now turn to a more formal exposition of our argument.
We will prove Theorem~\ref{theo:risk_convergence} in a series of Lemmas.
The lemmas are stated in this section, their proofs in the appendices.
All results impose the assumptions stated in Section~\ref{sec:setup}.
For transparency, \autoref{tab:assumptions} records which assumptions each result invokes directly, beyond those inherited through the results it builds on. 

\begin{table}[t]
\centering
\caption{Assumptions invoked by each result.}
\label{tab:assumptions}
\vspace{4pt}
\resizebox{\textwidth}{!}{%
\onehalfspacing
\begin{tabular}{lll}
\toprule
Result & Assumptions invoked & Builds on \\
\midrule
Lemma~\ref{lem:lipschitz} (Lipschitz $g^\lambda$) & convexity of $\pi$ & --- \\
Lemma~\ref{lem:ifrep} (Influence-function approx.) & 1, 2, 4.1, 4.2, 5.2 & Lemma~\ref{lem:lipschitz} \\
Lemma~\ref{lem:limitingloss} (Limiting squared-error loss) & 1, 2, 5.3 & --- \\
Corollary~\ref{cor:asymptotic_penalizederm} (Asymptotic distribution) & 1, 4.1 & Lemmas~\ref{lem:lipschitz},~\ref{lem:ifrep} \\
Lemma~\ref{lem:convergence_cv} (CV $\to$ SURE) & 1, 2, 3, 5.1--5.4 & Lemmas~\ref{lem:lipschitz},~\ref{lem:ifrep}, Cor.~\ref{cor:asymptotic_penalizederm} \\
Lemma~\ref{lem:convergencetuned} (Tuning-parameter convergence) & 1, 3 & Lemmas~\ref{lem:ifrep},~\ref{lem:convergence_cv}, Cor.~\ref{cor:asymptotic_penalizederm} \\
Theorem~\ref{theo:risk_convergence} (Risk convergence) & all & Lemmas~\ref{lem:limitingloss},~\ref{lem:convergencetuned} \\
Corollary~\ref{cor:risk_convergence} (Convergence of expected loss) & all & Theorem~\ref{theo:risk_convergence} \\
\bottomrule
\end{tabular}%
}
\end{table}

\begin{lemma}[Lipschitz $g^\lambda$]
  \label{lem:lipschitz}
  For any $\lambda \geq 0$, if $\pi( \cdot )$ is convex then $g^\lambda(\theta)= \argmin_g \tfrac12\|g\|^2 + \lambda  \cdot \pi(\theta + g)$ is Lipschitz with Lipschitz constant 1.
\end{lemma}

\begin{lemma}[Influence function approximation]
  \label{lem:ifrep}
  \be
  L_n(\theta) -  L_n(\theta_0) = \tfrac12 \|\theta - \tilde \theta_n\|^2  - \tfrac12 \|\tilde \theta_n - \theta_0\|^2 + \epsilon_n(\theta),
  \label{eq:loss_if}
  \ee
  where $\sup_{\theta: \| \theta \| \le C} \epsilon_n(\theta) = o_{\mu_n}(1)$ and  $\sup_{\theta: \| \theta \| \le C} \nabla \epsilon_n(\theta) = o_{\mu_n}(1)$ for any $C < \infty$, and
  \bals
  \tilde \theta_n &= \theta_0 + \tfrac{1}{\sqrt{n}} \sum_i X_n^i,  &
  X_n^i &= -\nabla_\beta l(\theta_0/\sqrt{n}, Z_n^i).
  \eals
  The ERM and penalized ERM estimators
  \bals
    \hat \theta_n &= \argmin_\theta L_n(\theta), &
  \hat \theta_n^\lambda &= \argmin_\theta \left[L_n(\theta) + \lambda  \cdot \pi(\theta)\right]
  \eals
  satisfy
  \bals
    \hat \theta_n &= \tilde \theta_n+ o_{\mu_n}(1),&
   \sup_\lambda  \Vert  \hat \theta_n^\lambda &- \tilde \theta_n - g^\lambda(\tilde \theta_n) \Vert = o_{\mu_n}(1).
   \eals
\end{lemma}

\begin{lemma}[Limiting squared error loss]
  \label{lem:limitingloss}
  The limiting expected loss is well defined and given by
  $$\bar L(\theta, \theta_0) =\tfrac12 \|\theta - \theta_0\|^2.$$
  Convergence of $\bar L_n(\theta, \theta_0)$ to this limit is uniform in any bounded neighborhood of $\theta_0$: $\sup_{\theta: \|\theta - \theta_0\| \leq C} |\bar L_n(\theta, \theta_0) - \bar L(\theta, \theta_0)| \to 0$ for all $C<\infty$.
\end{lemma}

\begin{corollary}[Asymptotic distribution for fixed tuning parameter]
  \label{cor:asymptotic_penalizederm}
  The ERM and penalized ERM estimators satisfy 
  \bals
    \hat \theta_n &\rightarrow^d \hat \theta \sim N(\theta_0, \Sigma),&
    \hat \theta_n^\lambda &\rightarrow^d \hat \theta + g^\lambda(\hat \theta).
  \eals
\end{corollary}

\begin{lemma}[Convergence of CV to SURE]
  \label{lem:convergence_cv}
  There exists a $\lambda$-independent (data-dependent) constant $c_n$ such that the n-fold crossvalidation criterion
  satisfies
  $$
    \sup_{\lambda \in \Lambda} \left|CV_n(\lambda) - c_n - SURE(\lambda, \hat \theta_n, \Sigma)\right| \rightarrow^{\mu_n} 0.
  $$
\end{lemma}

\begin{lemma}[Joint convergence of tuning parameter and tuned estimators]
  \label{lem:convergencetuned}
  $$
    (\lambda_n^*, \hat \theta_n) \rightarrow^d (\lambda^*, \hat \theta).
  $$
  and
  $$
  \hat \theta_n^{*} \rightarrow^d \hat \theta^{*}.
  $$
\end{lemma}

From Lemma \autoref{lem:convergencetuned}, we then show \autoref{theo:risk_convergence}.
The appendices prove each of these Lemmas in turn.

\paragraph{Discussion of assumptions}
Let us briefly discuss the role of our different assumptions, as referenced in \autoref{tab:assumptions}, in ensuring the validity of these results.
Each assumption controls a distinct source of error in our approximations. Assumption~\ref{as:sequence} places the problem in a local-to-zero regime, where the truth drifts at rate $1/\sqrt n$ while the score variance $\Sigma$ stays fixed and non-degenerate; this is what keeps the bias--variance tradeoff of penalization non-trivial in the limit, so that $CV$ and $SURE$ estimate a meaningful, non-degenerate risk. Assumption~\ref{as:smoothloss} normalizes the limiting Hessian to the identity; it carries no statistical content and merely puts the limit experiment into the normal-means form on which $SURE$ is built. Assumption~\ref{as:ridgelasso} restricts attention to Ridge and Lasso penalties, the two cases for which we can show that the $SURE$ objective has a unique, well-separated minimizer in $\lambda$ almost everywhere---the property that lets convergence of the \emph{criteria} carry over to convergence of the \emph{tuned} estimators (Lemma~\ref{lem:convergencetuned}). The two conditions in Assumption~\ref{as:ERM_conditions} govern the unregularized ERM estimator: the Lipschitz condition~(4.1) lets us invoke empirical-process arguments for its convergence, while the stochastic-boundedness condition~(4.2) rules out runaway estimates that would invalidate those arguments. Finally, the four conditions in Assumption~\ref{as:loss_fn_bounds} provide the quantitative control needed to equate $CV_n$ with $SURE$: strong convexity and smooth Hessians~(5.1) make each leave-one-out refit a small, $O(n^{-1/2})$ perturbation, so the degrees-of-freedom correction is well-behaved, and the score, Hessian, and fourth-derivative moment bounds~(5.2--5.4) control the successive terms in the Taylor expansion of the leave-one-out loss, ensuring the remainder vanishes uniformly in $\lambda$.

\section{Numerical experiments}
\label{sec:numerical}

We now illustrate our results numerically. The experiments have two purposes: to
verify that the approximations underlying our intermediate lemmas are accurate at
realistic sample sizes, and to show that the risk of a cross-validated estimator
indeed converges to the risk of its SURE-tuned limit. Replication code is
available at \url{https://github.com/maxkasy/cv_and_sure_simulations}.

\paragraph{Designs}
We compare each finite-sample experiment to its limiting normal-means
counterpart. In the finite-sample experiment we draw $n$ i.i.d.\ observations,
form the empirical risk minimizer $\hat\theta_n$, the $n$-fold cross-validation
criterion $CV_n(\lambda)$, the CV-tuned estimator $\hat\theta_n^*$, and its
out-of-sample loss $\bar L_n(\hat\theta_n^*,\theta_0)$. In the limiting
experiment we draw $\hat\theta\sim N(\theta_0,\Sigma)$ and compute the
SURE-tuned estimator $\hat\theta^*$ and its loss $\tfrac12\|\hat\theta^*-\theta_0\|^2$.

All designs are constructed directly in the normalized coordinates of
Assumption~\ref{as:smoothloss}, so that the limiting Hessian equals the identity,
 $H=I$.
We consider the same examples discussed in \autoref{sec:verify_examples} - linear regression and logisitic regression, for both Ridge and Lasso penalties: 

For
\emph{linear regression}, $l(\beta,z)=(y-w\cdot\beta)^2$ with
$W\sim N(0,\tfrac12 I_p)$ and $Y=W\cdot\theta_0/\sqrt n+U$,
$U\sim N(0,\tfrac12)$; this yields $H=2\,E[WW^\intercal]=I$ and $\Sigma=I$, and
the expected out-of-sample loss is exactly $\bar L_n(\theta,\theta_0)=\tfrac12\|\theta-\theta_0\|^2$.
For \emph{logistic regression}, $l(\beta,z)=-\log f(y\mid w,\beta)$ with
$W\sim N(0,4I_p)$ and $Y\sim\mathrm{Bernoulli}(\Lambda(W\cdot\theta_0/\sqrt n))$,
$\Lambda$ the logistic c.d.f.; here $b''(0)=\tfrac14$, so again $H=I$ and
$\Sigma=I$. The Ridge penalty is $\pi(\theta)=\tfrac12\theta\cdot A^{-1}\theta$
and the Lasso penalty $\pi(\theta)=\|A^{-1}\theta\|_1$, with $A$ diagonal.
Table~\ref{tab:sim_designs} lists the five resulting designs that we consider.

All five designs use $p=10$. Since $\Sigma=I$, each coordinate of $\hat\theta$ has
unit noise scale, and we choose the signal $\theta_0$ to be of comparable,
order-one magnitude. This is the regime in which the bias-variance tradeoff is
non-degenerate: Some shrinkage strictly improves on both the unregularized
estimator ($\lambda=0$) and full shrinkage ($\lambda=\infty$), so that the tuning
problem is not trivially solved at an endpoint. For the dense
designs (D1, D3) all coordinates carry the same unit signal; for the Lasso designs
(D2, D4) the signal is sparse, with three strong nonzero entries against a
background of exact zeros, which is the regime the Lasso is designed for.
Design~D5 is the anisotropic stress test: its penalty matrix $A$ has a wide,
geometric eigenvalue spread ($A_{jj}=40^{(j-1)/9}$, ranging from $1$ to $40$),
the higher-dimensional analogue of the anisotropic Ridge configuration that
produced the multi-modal $SURE$ examples of Figure~\ref{fig:multimodality} in
$p=2$. Genuine multi-modality of $SURE$ is a low-dimensional, knife-edge
phenomenon that does not survive in $p=10$; what the anisotropy produces here is a
$SURE$ objective that is poorly separated (nearly flat) near its minimum for a
sizeable fraction of draws. This probes the well-separation argument behind
Lemma~\ref{lem:convergencetuned} precisely where it is most demanding.

\begin{table}[t]
\centering
\footnotesize
\caption{Simulation designs. }
\label{tab:sim_designs}
\vspace{4pt}

\begin{tabular}{llllll}
\toprule
& Loss & Penalty & $p$ & $\theta_0$ & $A$ \\
\midrule
D1 & Linear & Ridge & 10 & $(1,\dots,1)$ & $I$ \\
D2 & Linear & Lasso & 10 & $(3,-3,2,0,\dots,0)$ & $I$ \\
D3 & Logit & Ridge & 10 & $(1,\dots,1)$ & $I$ \\
D4 & Logit & Lasso & 10 & $(3,-3,2,0,\dots,0)$ & $I$ \\
D5 & Linear & Ridge & 10 & $1.5\cdot(1,\dots,1)$ & $\mathrm{diag}(40^{(j-1)/9})$ \\
\bottomrule
\end{tabular}
\flushleft
$\Sigma=I$ and $p=10$ throughout (normalized coordinates). The sparse signal is $\theta_0=(3,-3,2,0,\dots,0)$; for D5, $A=\mathrm{diag}(40^{0},40^{1/9},\dots,40^{1})$ is a geometric grid spanning $[1,40]$.
\end{table}

For each design we vary the sample size $n$ and average over Monte Carlo
replications. The $n$-fold cross-validation criterion is computed exactly: in
closed form via leave-one-out residuals for the linear--Ridge designs (D1, D5),
and by explicit leave-one-out refitting otherwise. We use a grid of $14$ values
of $\lambda$. For Ridge, the grid is equally spaced in the compactified coordinate
$s=\lambda/(1+\lambda)\in[0,1]$, ranging from $\lambda=0$ (no shrinkage, $s=0$) to
$\lambda=\infty$ (full shrinkage, $s=1$), consistent with the compactified
$\Lambda=\mathbb R^+\cup\{\infty\}$ of Assumption~\ref{as:ridgelasso}. For Lasso,
the grid is equally spaced in $[0,\lambda_{\max}]$ with $\lambda_{\max}=3$.

\paragraph{Diagnostics}

We summarize each approximation by a single scale-free statistic, reported in the
columns of Table~\ref{tab:sim_results}, and write each as a discrepancy
$\Delta$ with a descriptive subscript. Let $\tilde\theta_n$ denote the
influence-function approximation of Lemma~\ref{lem:ifrep}.
\begin{enumerate}
\item \textbf{Influence-function approximation} (Lemma~\ref{lem:ifrep}):
$\Delta_{\mathrm{IF}} = E\|\hat\theta_n-\tilde\theta_n\|^2\big/E\|\hat\theta_n-\theta_0\|^2$,
the mean squared error of the linearization relative to the scale of the
estimand.

\item \textbf{Convergence of $CV$ to $SURE$} (Lemma~\ref{lem:convergence_cv}):
\[
\Delta_{\mathrm{CV}} = \frac{E\,\sup_{\lambda}\big|\,[CV_n(\lambda)-CV_n(\lambda_{\mathrm{ref}})]
-[SURE(\lambda)-SURE(\lambda_{\mathrm{ref}})]\,\big|}
{E\,\big[\max_\lambda SURE-\min_\lambda SURE\big]},
\]
the uniform gap between the two criteria, recentered at a reference
$\lambda_{\mathrm{ref}}$ to remove the $\lambda$-independent constant $c_n$, and
normalized by the range of $SURE$ over the grid. Here $SURE(\lambda)$ is
evaluated at the realized $\hat\theta_n$ and the true $\Sigma$.

\item \textbf{Agreement of the tuning parameter}
(Lemma~\ref{lem:convergencetuned}):
$\Delta_{\mathrm{tune}} = E\big|\,\ell(\lambda_n^*)-\ell(\lambda^{S})\,\big|\big/E\,\ell(\lambda^{S})$,
where $\ell(\lambda)=\tfrac12\|\hat\theta_n+g^\lambda(\hat\theta_n)-\theta_0\|^2$
and $\lambda_n^*$, $\lambda^{S}$ minimize $CV_n$ and $SURE(\cdot,\hat\theta_n,\Sigma)$
at the same $\hat\theta_n$. This is the relative excess loss incurred by tuning
with $CV$ rather than $SURE$ --- directly addressing whether the two minimizers
are close enough to matter.

\item \textbf{Convergence of risk} (Theorem~\ref{theo:risk_convergence},
Corollary~\ref{cor:risk_convergence}):
$\Delta_{R}=|R_n-R|$, the gap between the finite-sample risk
$R_n=E[\min(\bar L_n(\hat\theta_n^*,\theta_0),M)]$ and its limit
$R=E[\min(\tfrac12\|\hat\theta^*-\theta_0\|^2,M)]$, with truncation $M=20$.
\end{enumerate}

\paragraph{Results}

\begin{table}[t]
\centering
\footnotesize
\caption{Approximation diagnostics at $n=400$. }
\label{tab:sim_results}
\vspace{4pt}

\begin{tabular}{l c c c c c c}
\toprule
Design & $\Delta_{\mathrm{IF}}$ & $\Delta_{\mathrm{CV}}$ & $\Delta_{\mathrm{tune}}$ & $R_n$ & $R$ & $\Delta_{R}$ \\
& (Lem.~2) & (Lem.~4) & (Lem.~5) & (Thm.~1) & (limit) & $|R_n-R|$ \\
\midrule
D1\quad Linear / Ridge & 0.028 & 0.208 & 0.052 & 3.110 & 3.084 & 0.026 \\
D2\quad Linear / Lasso (sparse) & 0.028 & 0.226 & 0.111 & 3.885 & 3.789 & 0.096 \\
D3\quad Logit / Ridge & 0.032 & 0.309 & 0.055 & 3.062 & 3.075 & 0.013 \\
D4\quad Logit / Lasso (sparse) & 0.038 & 0.297 & 0.127 & 3.731 & 3.776 & 0.046 \\
D5\quad Linear / Ridge (anisotropic) & 0.028 & 0.152 & 0.036 & 4.827 & 4.733 & 0.094 \\
\bottomrule
\end{tabular}
\flushleft
Columns: $\Delta_{\mathrm{IF}}$ (Lemma~\ref{lem:ifrep}), $\Delta_{\mathrm{CV}}$ (Lemma~\ref{lem:convergence_cv}), $\Delta_{\mathrm{tune}}$ (Lemma~\ref{lem:convergencetuned}), the finite-sample and limiting risks $R_n,R$, and their gap $\Delta_{R}$ (Theorem~\ref{theo:risk_convergence}). Averages over $8000$ replications for the closed-form designs D1, D5 and $2000$ for the leave-one-out designs D2--D4.
\end{table}

Table~\ref{tab:sim_results} reports the four diagnostics at $n=400$, and
Figures~\ref{fig:sim_convergence} and~\ref{fig:sim_riskfunction} display their
behavior as $n$ grows.
\begin{figure}[t]
\centering
\includegraphics[width=\textwidth]{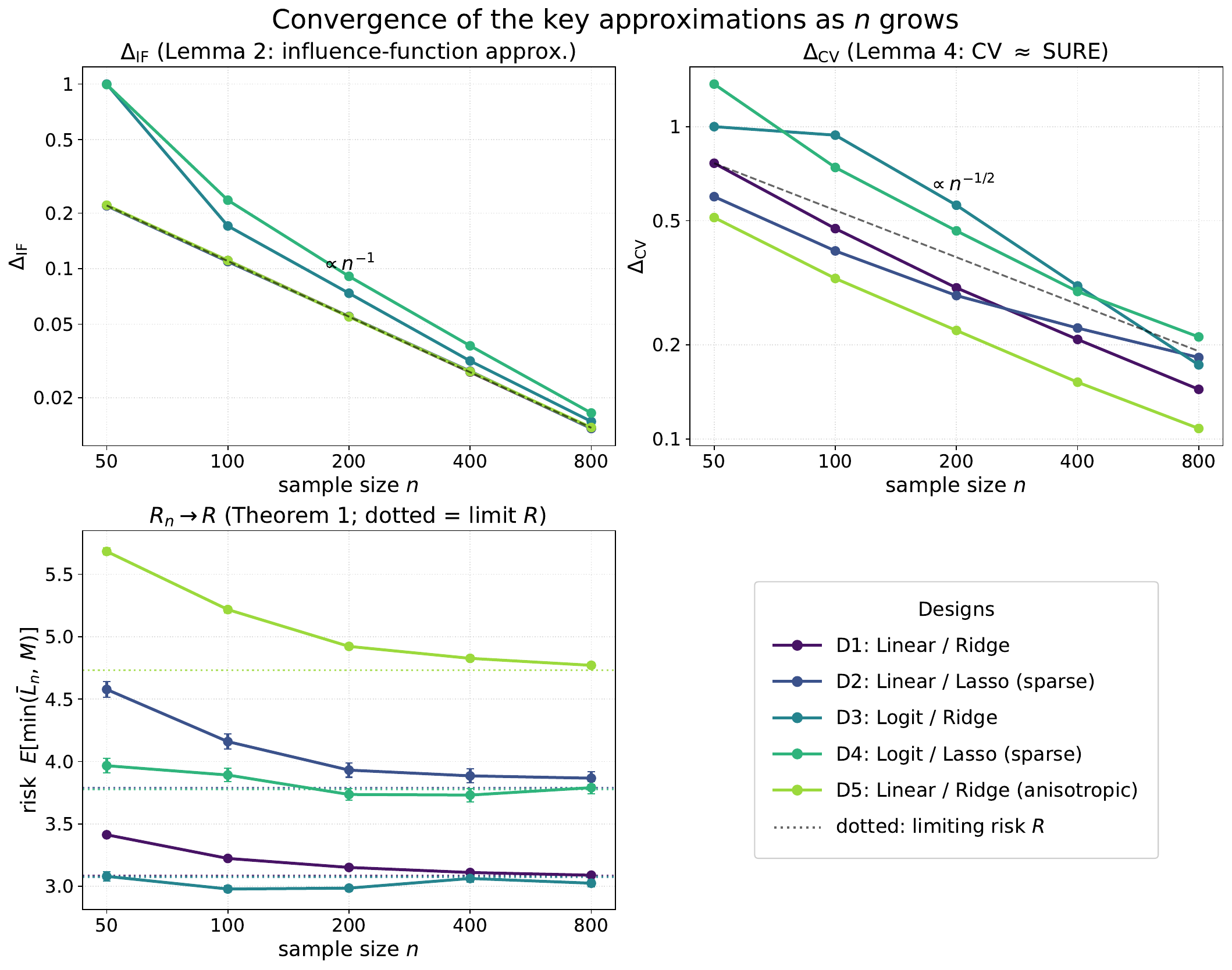}
\caption{\footnotesize Convergence of the approximations as $n$ grows (designs identified in the
legend panel, lower right). Top left: the influence-function error
$\Delta_{\mathrm{IF}}$ (Lemma~\ref{lem:ifrep}), $\propto n^{-1}$. Top right: the
cross-validation gap $\Delta_{\mathrm{CV}}$ (Lemma~\ref{lem:convergence_cv}),
$\propto n^{-1/2}$. Bottom left: the finite-sample risk $R_n$ converging to the
limit $R$ (dotted lines; Theorem~\ref{theo:risk_convergence}). }
\label{fig:sim_convergence}
\end{figure}
All four diagnostics are small across every design, and Figure~\ref{fig:sim_convergence}
shows that they vanish as $n$ grows. The influence-function error
$\Delta_{\mathrm{IF}}$ decreases at the parametric rate $n^{-1}$ --- for example
from $0.22$ at $n=50$ to $0.014$ at $n=800$ for design~D1, halving with each
doubling of $n$ --- confirming the linearization $\hat\theta_n\approx\tilde\theta_n$ of Lemma~\ref{lem:ifrep}. 
The
cross-validation gap $\Delta_{\mathrm{CV}}$ decreases at the slower rate
$n^{-1/2}$, as expected for a quantity governed by the sampling variability of the
leave-one-out perturbations; the convergence holds uniformly across the Ridge and
Lasso penalties and across the linear and logistic losses, confirming
Lemma~\ref{lem:convergence_cv}.

The tuning diagnostic $\Delta_{\mathrm{tune}}$ shows that selecting $\lambda$ by
cross-validation rather than by $SURE$ costs only a few percent of the tuned loss,
and shrinks toward zero as $n$ grows. It is largest for the two Lasso designs (D2,
D4: around $0.11$--$0.13$ at $n=400$), whose non-quadratic, kinked objective is the
least smooth, and smallest for the anisotropic design~D5 ($0.036$). The latter is
instructive: the wide eigenvalue spread of D5 makes $SURE$ nearly flat near its
minimum, so although the minimizer is poorly separated, the loss is correspondingly
insensitive to the exact choice of $\lambda$. Poor separation and large excess loss thus need not coincide. Consequently the finite-sample risk $R_n$ is close to its
limit $R$ already at $n=400$ (the gap $\Delta_{R}$ is at most a fraction of a unit,
against risk levels of $3$ to $5$), and converges to it as $n$ grows; by $n=800$
the two are nearly indistinguishable.

\begin{figure}[t]
\centering
\includegraphics[width=0.62\textwidth]{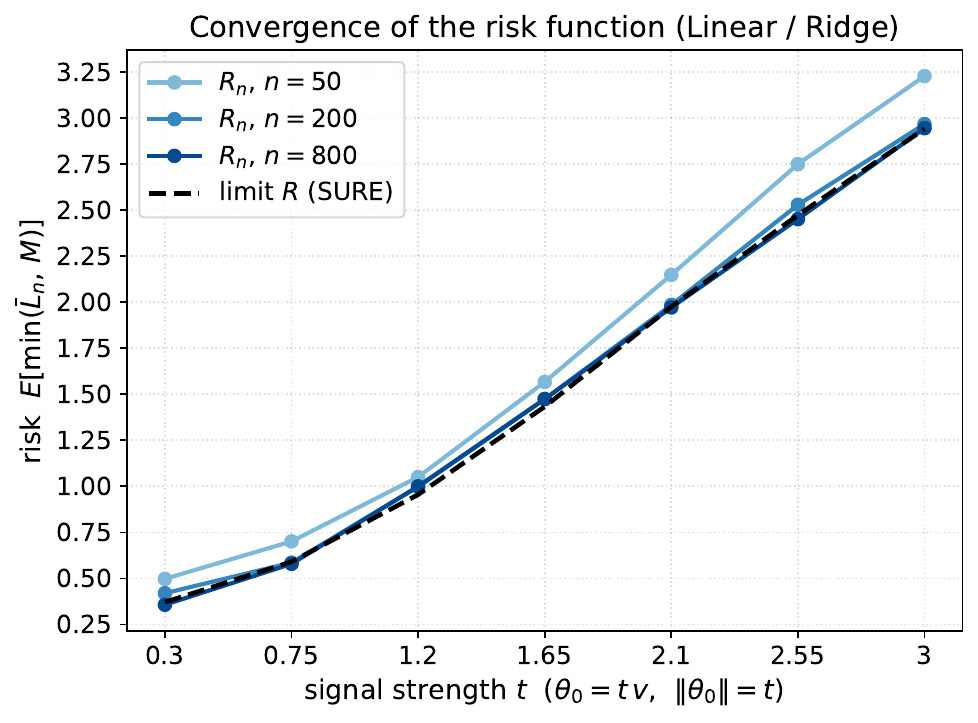}
\caption{\footnotesize Convergence of the risk function (linear--Ridge design). Along the slice
$\theta_0=t\cdot v$ (with $v$ a fixed unit vector, so $\|\theta_0\|=t$), the
finite-sample risk $R_n(\theta_0)$ of the CV-tuned estimator approaches the
SURE-tuned limit $R(\theta_0)$; by $n=200$ the curves are nearly
indistinguishable from the limit.}
\label{fig:sim_riskfunction}
\end{figure}

Figure~\ref{fig:sim_riskfunction} makes the convergence concrete for the
linear--Ridge design along a one-dimensional slice $\theta_0=t\cdot v$: the
finite-sample risk function $R_n(\theta_0)$ approaches the limiting risk function
$R(\theta_0)$ predicted by Theorem~\ref{theo:risk_convergence}, and the two are
nearly indistinguishable by $n=200$.
This is the practical content of our result:
the equivalence between cross-validation and SURE is useful at realistic sample
sizes, not only asymptotically.

Note also that the limiting risk $R(\theta_0)$ rises from about $0.4$ near the
origin toward---but always below---the risk of the \emph{unregularized}
estimator: at $\lambda=0$ the limiting estimator is $\hat\theta\sim
N(\theta_0,\Sigma)$, whose risk is $\theta_0$-independent and equal to
$E[\tfrac12\|\hat\theta-\theta_0\|^2]=\tfrac12\trace(\Sigma)=p/2=5$. The CV- and
SURE-tuned estimators thus strictly improve on the unregularized estimator across
the entire slice, and indeed in every design of Table~\ref{tab:sim_results}, where
the tuned limiting risks range from $3.1$ to $4.7$---all below $5$. This confirms the well-known dominance of James--Stein shrinkage over
the maximum-likelihood estimator (Figure~\ref{fig:JS-risk}).

\section{Conclusion}
\label{sec:conclusion}

We have shown that the out-of-sample prediction loss of regularized empirical risk minimization estimators tuned by $n$-fold cross-validation converges in distribution to the squared-error loss of the corresponding shrinkage estimator in the normal-means model, tuned by Stein's unbiased risk estimate. We conclude by discussing what this equivalence tells us.

For \emph{practitioners}, the result justifies using SURE as an asymptotically equivalent and computationally cheaper surrogate for $n$-fold cross-validation: SURE is available in closed form for both the Ridge and the Lasso penalty, and it avoids refitting the model $n$ times. Beyond this computational gain, the limiting risk function provides a uniform, robust performance guarantee for estimators tuned using CV. 

For \emph{theorists}, our result reduces the analysis of CV-tuned regularized ERM to a shrinkage problem in the normal-means model. This yields a characterization of the asymptotic \emph{risk function}, which discribes how risk varies with the underlying parameter $\theta_0$. This provides a rich characterization: Instead of a single uniform upper bound on regret, as in standard learning theory, the risk function allows to distinguish the parameter regions where regularization helps more or less. It also implies that the well-developed apparatus for the normal-means model (Stein's identity, James-Stein shrinkage, etc.) can be used to analyze the    asymptotic behavior of other data-driven tuning procedures.

Our analysis has some limitations that suggest directions for future work. Our framework is in particular fixed-dimensional and local-to-zero: the number of parameters is held fixed as $n$ grows, and the optimal parameter drifts to zero at the $\sqrt n$ rate. The high-dimensional case, in which the dimension grows with the sample size, is likely to behave differently. This high-dimensional case might provide a more accurate description of machine learning practice. In many such regimes one would expect the data-driven tuning parameter to be \emph{consistent} for its optimal value (as, for example, in \citealt{abadiekasy2017}); the variability of the tuning parameter could then be ignored, and the limiting risk function would simplify accordingly.

\clearpage
\bibliographystyle{apalike}
\bibliography{library}

\appendix
\clearpage
\section{Verifying the regularity conditions for leading examples}
\label{sec:verify_examples}

We now verify that our assumptions hold for two leading examples: linear regression with quadratic error loss, and smooth generalized linear models. Throughout, $Z_n^i=(W_n^i,Y_n^i)$ with regressors $W_n^i\in\mathbb R^p$ i.i.d across $n$ and $i$, and we write $Q=E[W_n^i W_n^{i\intercal}]$.

Assumptions~\ref{as:sequence}--\ref{as:ridgelasso} are easy to verify for the following examples: Assumption~\ref{as:sequence} holds immediately if an appropriate local-to-zero data-generating process is specified (with $\Sigma\succ0$ and constant in $n$ checked explicitly), Assumption~\ref{as:smoothloss} is the coordinate normalization of the Hessian (e.g.\ $2Q=I$ for linear regression), and Assumption~\ref{as:ridgelasso} is simply the choice of Ridge or Lasso penalty. We therefore focus on verifying the regularity conditions in Assumptions~\ref{as:ERM_conditions} and~\ref{as:loss_fn_bounds}.

\paragraph{Linear regression.}
Let $l(\beta,z)=(y-w\cdot\beta)^2$, with data generated as $Y_n^i = W_n^i\cdot\beta_n + U_n^i$, where $\beta_n=\theta_0/\sqrt n$, $E[U_n^i\mid W_n^i]=0$, and $E[(U_n^i)^2\mid W_n^i]\ge \underline\sigma^2>0$. Assume $Q\succ0$ and the uniform moment bounds $E\|W_n^i\|^8\le M$ and $E[(U_n^i)^8]\le M$. The score and Hessian are
\[
\nabla_\beta l(\beta,z) = -2w(y-w\cdot\beta),\qquad \nabla_\beta^2 l(\beta,z) = 2ww^\intercal,
\]
so the influence function is $X_n^i=-\nabla_\beta l(\beta_n,Z_n^i)=2W_n^i U_n^i$, with $\Sigma=\Var(X_n^i)=4\,E[W_n^iW_n^{i\intercal}(U_n^i)^2]\succ0$, as required by Assumption~\ref{as:sequence}. The Hessian is constant in $\beta$ with expectation $2Q$, and Assumption~\ref{as:smoothloss} is the coordinate normalization $2Q=I$, imposed without loss of generality.

\emph{Assumption~\ref{as:ERM_conditions}.} For item~1, on any ball $\|\beta\|\le\delta$ we have $\|\nabla_\beta l(\beta,z)\|\le 2\|w\|(|y|+\|w\|\delta)=:m(z)$, and $\sup_n\Var(m(Z_n^i))<\infty$ follows from the moment bounds (using $|Y_n^i|\le\|W_n^i\|\,\|\beta_n\|+|U_n^i|$ with $\|\beta_n\|\to0$); the loss is everywhere differentiable. For item~2, the ERM estimator is OLS, $\hat\theta_n=\theta_0+\big(\tfrac1n\sum_iW_n^iW_n^{i\intercal}\big)^{-1}\tfrac1{\sqrt n}\sum_iW_n^iU_n^i=\theta_0+O_{\mu_n}(1)$, since $\tfrac1n\sum_iW_n^iW_n^{i\intercal}\to Q\succ0$ and $\tfrac1{\sqrt n}\sum_iW_n^iU_n^i=O_{\mu_n}(1)$ by the central limit theorem; hence $\hat\theta_n$ is bounded in probability.

\emph{Assumption~\ref{as:loss_fn_bounds}.} Recall $l_n(\theta,z)=l(\theta/\sqrt n,z)$, so $\sqrt n\nabla_\theta l_n=\nabla_\beta l$, $n\nabla_\theta^2 l_n=\nabla_\beta^2 l$, and $n^2 D^4_\theta l_n=D^4_\beta l$, all evaluated at $\beta=\theta/\sqrt n$.
\begin{itemize}
\item \emph{Item 1.} $\nabla_\theta^2 L_n(\theta)=\tfrac1n\sum_i 2W_n^iW_n^{i\intercal}\to 2Q\succ0$, with eigenvalues bounded above and below with probability approaching one; hence $L_n$ is $\mu$-strongly convex and $\nu$-smooth for suitable $0<\mu<\nu$.
\item \emph{Item 2.} As $\nabla_\beta l$ is affine in $\beta$, $\|\sqrt n\nabla_\theta l_n(\theta,z)-\sqrt n\nabla_\theta l_n(\theta',z)\|=\|2ww^\intercal(\theta-\theta')/\sqrt n\|\le (2\|w\|^2/\sqrt n)\|\theta-\theta'\|$, so $B_n(z)=2\|w\|^2/\sqrt n$ with $E[B_n(Z_n^i)^2]=4E\|W_n^i\|^4/n<\infty$. Moreover $E\|\sqrt n\nabla_\theta l_n(\theta_0,Z_n^i)\|^4=16\,E[\|W_n^i\|^4(U_n^i)^4]\le 16\sqrt{E\|W_n^i\|^8\,E[(U_n^i)^8]}\le 16M$.
\item \emph{Item 3.} $n\nabla_\theta^2 l_n(\theta,z)=2ww^\intercal$ is constant in $\theta$, so $\tfrac1n\sum_i\|2W_n^iW_n^{i\intercal}\|^2=4\cdot\tfrac1n\sum_i\|W_n^i\|^4=O_{\mu_n}(1)$, and the Lipschitz bound holds with $C_n(z)\equiv0$.
\item \emph{Item 4.} The loss is quadratic, so $D^4_\beta l\equiv0$ and the bound holds with $M=0$.
\end{itemize}

\paragraph{Smooth generalized linear models.}
Let $l(\beta,z)=-\log f(y\mid w,\beta)$ for an exponential-family model with $\log f(y\mid w,\beta)=y\,(w\cdot\beta)-b(w\cdot\beta)+c(y)$ and log-partition function $b(\cdot)$, as in logistic or Poisson regression. Then $\nabla_\beta l=-(y-b'(w\cdot\beta))w$, $\nabla_\beta^2 l=b''(w\cdot\beta)\,ww^\intercal$, and $D^k_\beta l=(-1)^k b^{(k)}(w\cdot\beta)\,w^{\otimes k}$. Suppose the regressors are bounded, $\|W_n^i\|\le B$ almost surely, and $\sup_n E[(Y_n^i)^4]\le M$. The log partition function satisfies $b''>0$, so the loss is strictly convex; near the drifting truth $\beta_n\to0$ the curvature $b''(W\cdot\beta_n)$ is bounded below, giving $\nabla_\theta^2 L_n(\theta)\to E[b''(0)W_n^iW_n^{i\intercal}]\succ0$ and hence Assumption~\ref{as:loss_fn_bounds}(i). Because $b$ is smooth with derivatives bounded on the compact set $\{|w\cdot\beta|\le B\delta\}$ and the regressors are bounded, the score, Hessian, and fourth-derivative bounds (items~2--4) reduce to moments of $\|W_n^i\|$, which are finite by boundedness; Assumption~\ref{as:ERM_conditions} follows as above. (For logistic regression $b(t)=\log(1+e^t)$, with $0<b''\le 1/4$ and all $b^{(k)}$ bounded.)

\clearpage
\section{Proofs: Influence function approximations}
\label{sec:influencefunctions}

\subsection{Proof of Lemma \ref{lem:lipschitz} (Lipschitz $g^\lambda$)}

Fix $0 \leq \lambda < \infty$.
Recall that
$
    g^\lambda(\theta) = \argmin_g \tfrac1{2} \|g\|^2 + \lambda \cdot  \pi(\theta + g).
$
If $\pi( \cdot )$ is convex, then so is the objective function on the right.
This implies that there exists a sub-gradient $\nabla \pi$ of $\pi$ such that the first order condition
$$
g +  \lambda \cdot  \nabla \pi(\theta + g) = 0
$$
holds for $g = g^\lambda(\theta)$.
Consider two values $\theta_1, \theta_2$ of $\theta$, and the corresponding solutions $g_1, g_2$ and sub-gradients $\nabla \pi_1, \nabla \pi_2$, as well as the differences $\Delta \theta = \theta_2 - \theta_1$ and $\Delta g = g_2 - g_1$.
Taking the difference of the first order condition across the two values yields.
$$
\Delta g + \lambda  \cdot \left[\nabla \pi_2 - \nabla \pi_1\right] = 0.
$$
Convexity of $\pi$ implies
$$
\langle \nabla \pi_2 - \nabla \pi_1, \Delta \theta + \Delta g \rangle \geq 0.
$$
Combining the last two equations yields
$$
\langle \Delta g, \Delta \theta + \Delta g \rangle = - \lambda  \cdot \langle \nabla \pi_2 - \nabla \pi_1, \Delta \theta + \Delta g \rangle \leq 0,
$$
and thus (using Cauchy-Schwartz to get the second inequality),
$$
\|\Delta g\|^2 \leq \langle \Delta g, - \Delta \theta\rangle \leq \|\Delta g\|  \cdot \|\Delta \theta\|,
$$
so that
$$
\|\Delta g\| \leq \|\Delta \theta\|.
$$
This proves that $g^\lambda(\theta)$ is Lipschitz with Lipschitz-constant 1.
\hfill$\qed$

\subsection{Proof of Lemma \ref{lem:ifrep} (Influence function approximation)}

    Assume first that Equation~\eqref{eq:loss_if} holds, so that $L_n(\theta) -  L_n(\theta_0) = \tfrac12 \|\theta - \tilde \theta_n\|^2 - \tfrac12 \|\tilde \theta_n - \theta_0\|^2 + \epsilon_n(\theta)$.
    We show, under this assumption, that $\sup_\lambda \|\hat \theta_n^\lambda-\tilde \theta_n^\lambda\|^2 = o_{\mu_n}(1)$. Leveraging convexity of $\pi$ and Lipschitz continuity of $g^\lambda$, we first bound the corresponding difference in penalized squared error, which allows us to bound the difference in squared error, and finally the difference between the estimators themselves.

    \paragraph{Bounding the difference in penalized squared error loss}
    Define 
    $$
      \tilde \theta_n^\lambda = 
      \argmin_\theta \left[\tfrac12 \|\theta - \tilde \theta_n\|^2 + \lambda  \cdot \pi(\theta)\right] =
      \tilde \theta_n+ g^\lambda(\tilde \theta_n).
    $$
    By definition,
    $$
    \hat \theta_n^\lambda = \argmin_\theta \left[L_n(\theta) + \lambda  \cdot \pi(\theta)\right]
    $$
    and thus
    $$L_n(\hat \theta_n^\lambda) + \lambda  \cdot \pi(\hat \theta_n^\lambda) \leq 
    L_n(\tilde \theta_n^\lambda) + \lambda  \cdot \pi(\tilde \theta_n^\lambda).$$
    Substituting for $L_n( \cdot )$ on both sides of this inequality, using Equation~\eqref{eq:loss_if} applied to both $\theta = \hat \theta_n^\lambda$ and $\theta = \tilde \theta_n^\lambda$, and rearranging yields
    \bal
    \left[\tfrac12 \|\hat \theta_n^\lambda-\tilde \theta_n\|^2 + \lambda  \cdot \pi(\hat \theta_n^\lambda)\right] - \left[\tfrac12 \|\tilde \theta_n^\lambda-\tilde \theta_n\|^2 + \lambda  \cdot \pi(\tilde \theta_n^\lambda)\right]
    &\leq \epsilon_n(\tilde \theta_n^\lambda) - \epsilon_n(\hat \theta_n^\lambda). \label{eq:epsilonbound}
    \eal

    \paragraph{Bounding the difference in squared error loss}
    We next prove the following claim: Convexity of $\pi$, and the definition $\tilde \theta_n^\lambda = 
      \argmin_\theta \left[\tfrac12 \|\theta - \tilde \theta_n\|^2 + \lambda  \cdot \pi(\theta)\right]$, imply  that, for any $\theta$, 
    \be
    \tfrac12 \|\theta-\tilde \theta_n^\lambda\|^2\leq
    \left[\tfrac12 \|\theta-\tilde \theta_n\|^2 + \lambda  \cdot \pi(\theta)\right] -  \left[\tfrac12 \|\tilde \theta_n^\lambda-\tilde \theta_n\|^2 + \lambda  \cdot \pi(\tilde \theta_n^\lambda)\right].
    \label{eq:convexity}
    \ee
    To show \eqref{eq:convexity}, denote $a(\theta) = \tfrac12 \|\theta-\tilde \theta_n\|^2$ and $b(\theta) = \lambda  \cdot \pi(\theta)$.
    We can write 
    $$
    \tfrac12 \|\theta-\tilde \theta_n^\lambda\|^2 = a(\theta) - a(\tilde \theta_n^\lambda) - \nabla a(\tilde \theta_n^\lambda) \cdot (\theta  - \tilde \theta_n^\lambda).
    $$
    By convexity of $\pi$ and optimality of $\tilde \theta_n^\lambda$, there exists a subgradient $\nabla b$ of $b$ such that
    $$
    \nabla a(\tilde \theta_n^\lambda) +\nabla b(\tilde \theta_n^\lambda) = 0.
    $$
    Eliminating the common terms $a(\theta) - a(\tilde \theta_n^\lambda)$ on the left and right hand side, we can now rewrite \eqref{eq:convexity} as
    $$
    - \nabla a(\tilde \theta_n^\lambda) \cdot (\theta  - \tilde \theta_n^\lambda) \leq b(\theta) - b(\tilde \theta_n^\lambda).
    $$
    But since $-\nabla a(\tilde \theta_n^\lambda) = \nabla b(\tilde \theta_n^\lambda)$, this inequality holds by convexity of $b$ and the definition of a subgradient, and the claim follows.
    
    \paragraph{Bounding the distance between estimators}

    Combining two inequalities \eqref{eq:epsilonbound} and \eqref{eq:convexity}  yields 
    $$
    \tfrac12 \|\hat \theta_n^\lambda-\tilde \theta_n^\lambda\|^2 \leq \epsilon_n(\tilde \theta_n^\lambda) - \epsilon_n(\hat \theta_n^\lambda).
    $$
   It follows from Assumption~\ref{as:ERM_conditions} item 2, which states that $\hat \theta_n^\lambda$ is bounded in probability, and the Lipschitzness of $g^\lambda$ (Lemma~\ref{lem:lipschitz}), which implies $\|\tilde \theta_n^\lambda\| \leq 2 \|\tilde \theta_n\|$, that both $\hat \theta_n^\lambda$ and $\tilde \theta_n^\lambda$ are bounded in probability. Combined with equation (\ref{eq:loss_if}), we consequently obtain that with probability approaching 1 under $\mu_n$, there exists $C < \infty$ such that
      $$
  \sup_\lambda  \tfrac12 \|\hat \theta_n^\lambda-\tilde \theta_n^\lambda\|^2  \le \left( 2 \sup_{\| \theta \| \le C} \epsilon_n(\theta) \right) = o(1).
    $$
   
    The statement for the ERM estimator follows as a special case, where $\lambda = 0$.\\

    \paragraph{Proving Equation~\eqref{eq:loss_if}, using empirical process theory}
    It remains to show that that \eqref{eq:loss_if} holds, where $\sup_{\| \theta \| \le C} \epsilon_n(\theta) = o_{\mu_n}(1)$.
    This claim follows from a straightforward generalization of the proof of Lemma 19.31 in \cite{van2000asymptotic} to the case of drifting distributions.
    
    Applicability of arguments of Lemma 19.31 in \cite{van2000asymptotic} is guaranteed by the conditions in Assumption \ref{as:ERM_conditions}, item 1. In particular, pointwise convergence, for fixed $\theta$, follows from almost sure differentiability of $l$, by dominated convergence, given the uniform bound on the variance of $m(Z_n^i)$ in Assumption \ref{as:ERM_conditions}.
    To get uniform convergence across values of $\theta$ in any ball of radius $\delta$ around $0$, tightness needs to be shown.
    Tightness follows from a bound on the bracketing number of the class of functions $\{\sqrt{n} (l_n(\theta, \cdot ) - l_n(0,  \cdot )):\; \|\theta\|\leq \delta \}$.
    The bound in the proof of Lemma 19.31 in \cite{van2000asymptotic} applies verbatim, with a constant $C$ that does not depend on $n$, based on the uniform bound on the variance of $m(Z_n^i)$ in Assumption \ref{as:ERM_conditions}.
    The claim follows.

The claim that $\sup_{\theta: \| \theta \| \le C} \nabla \epsilon_n(\theta) = o_{\mu_n}(1)$ follows from the same argument, applied to $\nabla_\theta l_n(\theta, Z_n^i)$, using the condition on scores in item 2 of Assumption~\ref{as:loss_fn_bounds}. 
    \hfill$\qed$

\subsection{Proof of Lemma \ref{lem:limitingloss} (Limiting squared error loss)}

By Definition~\ref{def:loss}, $\bar L_n(\theta, \theta_0) =  E\left[ L_n(\theta) - L_n(\theta_0)\right]$, and $\bar L_n(\theta, \theta_0)$ is minimized at $\theta =\theta_0$.
By a second order Taylor expansion around $\theta =\theta_0$,
$$
\bar L_n(\theta, \theta_0) = \tfrac12 (\theta - \theta_0)\cdot\nabla_\theta^2 \bar L_n(\tilde\theta, \theta_0) \cdot (\theta - \theta_0).
$$
for some $\tilde \theta $ between $\theta $ and $\theta_0$.
By Assumption~\ref{as:smoothloss}, $\nabla_\theta^2 \bar L(\theta, \theta_0)|_{\theta = \theta_0} = I.$
By definition,
$$
\nabla_\theta^2 \bar L_n(\theta, \theta_0) = \nabla_\beta^2 E\left[l(\beta ,Z_n^i)\right]\big |_{\beta =\theta/\sqrt{n}}.
$$
The claim of the lemma then follows from continuity of the Hessian of $\nabla_\beta^2 E\left[l(\beta ,Z_n^i)\right]$ at $\beta =0$.
Continuity of the Hessian follows from item 3 of Assumption~\ref{as:loss_fn_bounds}:
\bals
&\left\Vert\nabla_\beta^2 E\left[l(\beta ,Z_n^i)\right]\big |_{\beta =\theta/\sqrt{n}} - \nabla_\beta^2 E\left[l(\beta ,Z_n^i)\right]\big |_{\beta =0}\right\Vert\\
\leq &  E\left[\left\Vert\nabla_\beta^2l(\beta ,Z_n^i)\big |_{\beta =\theta/\sqrt{n}} - \nabla_\beta^2 l(\beta ,Z_n^i)\big |_{\beta =0}\right\Vert\right]\\
\leq &E\left[C_{n}(Z_{n}^{i})\right] \cdot \frac{\Vert\theta\Vert}{\sqrt{n}},
\eals
where $\sup_n E\left[C_{n}(Z_{n}^{i})\right]<\infty$; this follows from $\sup_n E[ C_n(Z_n^i)^2] < \infty$ (Assumption~\ref{as:loss_fn_bounds}.3) via Jensen's inequality.

\hfill$\qed$

\subsection{Proof of Corollary \ref{cor:asymptotic_penalizederm} (Asymptotic distribution for fixed tuning parameter)}

Recall that $\Var(X_n^1) = \Sigma$ is constant in $n$, by assumption.
Note furthermore that Assumption \ref{as:ERM_conditions} (item 1) implies the Lindeberg condition $$E\left[\|X_n^1\|^2  \cdot \bs 1(\|X_n^1\| > \sqrt{n} M)\right] \rightarrow 0$$ for all $M>0$, since $\|X_n^1\| \leq m(Z_n^1)$: $X_n^i = -\nabla_\beta l(\theta_0/\sqrt{n}, Z_n^i)$, and the Lipschitz condition in Assumption~\ref{as:ERM_conditions}.1 together with a.e.\ differentiability implies $\|\nabla_\beta l(\beta, Z_n^i)\| \leq m(Z_n^i)$ at all points of differentiability, and the variance of the latter is uniformly bounded.
The Lindeberg-Feller central limit theorem (Proposition 2.27 in \citealt{van2000asymptotic}), applied to the triangular array $(X_n^i)$, therefore implies $\tilde \theta_n \rightarrow^d N(\theta_0, \Sigma)$.

The claims of Corollary \ref{cor:asymptotic_penalizederm} then follow from Lemma \ref{lem:ifrep}, and the continuous mapping theorem, where continuity of $g^\lambda$ follows from convexity of $\pi$, by Lemma \ref{lem:lipschitz}.

   \hfill$\qed$

\clearpage
\section{Proof of Lemma \ref{lem:convergence_cv}}
\label{sec:cvconvergence}

\subsubsection*{Step 0 (Preliminary observations):}

We start by stating some useful results for the proof. 
First, note that by Lemma 1 in \cite{wilsonkasymackey2018} and Assumption~\ref{as:loss_fn_bounds}(i), it follows
\begin{align}
\sup_{\lambda}\left\Vert \hat{\theta}_{n}^{\lambda,-i}-\hat{\theta}_{n}^{\lambda}\right\Vert  &= O_{\mu_{n}}\left(\frac{1}{\mu}\left\Vert \nabla_{\theta}l_{n}(\hat{\theta}_{n},Z_{n}^{i})\right\Vert \right) =  O_{\mu_{n}}(n^{-1/2}).\label{eq:requirement_1}
\end{align}
An analogous argument implies
\begin{align}
\sup_{\lambda}\left\Vert \tilde{\theta}_{n}^{\lambda,-i}-\tilde{\theta}_{n}^{\lambda}\right\Vert  & =O_{\mu_{n}}(n^{-1/2}),\label{eq:requirement_3}
\end{align} 
where 
\begin{align}
\tilde{\theta}^{\lambda, -i} := \argmin_\theta \left[ \tfrac12\|\theta - \tilde{\theta}_n^{-i} \|^2 + \lambda  \cdot \pi(\theta)\right].
\end{align}
Second, recall from Lemma \ref{lem:ifrep} that 
\begin{equation}
\sup_{\lambda}\left\Vert \hat{\theta}_{n}^{\lambda}-\tilde{\theta}_{n}^{\lambda}\right\Vert =o_{\mu_{n}}(1).\label{eq:pf:requirement_2}
\end{equation}
The next set of results concern the properties of $g^{\lambda}(\cdot)$. Since $g^{\lambda}(\cdot)$ is Lipschitz continuous by Lemma \ref{lem:lipschitz}, it is differentiable
almost everywhere (by Rademacher's theorem). In particular, there exists an $R^{\lambda}(\cdot;\theta)$ such that 
\begin{equation}
g^{\lambda}(\theta+\delta)=g^{\lambda}(\theta)+\nabla g^{\lambda}(\theta)^{\intercal}\delta+R^{\lambda}(\delta;\theta),\label{eq:gradient_expansion_of_g_lambda}
\end{equation}
and
\begin{equation}
\lim_{\left\Vert \delta\right\Vert \to0}\frac{\left\Vert R^{\lambda}(\delta;\theta)\right\Vert }{\left\Vert \delta\right\Vert }=0\ \textrm{for each \ensuremath{\lambda} and (Lebesgue) almost every }\theta.\label{eq:pf:7}
\end{equation}
In fact, under Assumption \ref{as:ridgelasso}, we can strengthen  (\ref{eq:pf:7}) to:
 \begin{equation}
    \lim_{\left\Vert \delta\right\Vert \to0}\sup_{\lambda\in\Lambda}\frac{\left\Vert R^{\lambda}(\delta;\theta)\right\Vert }{\left\Vert \delta\right\Vert }=0\ \textrm{for (Lebesgue) almost every }\theta.
 \label{eq:pf:R_bound}   
\end{equation}
For Ridge, (\ref{eq:pf:R_bound}) is immediate, since $R^\lambda(\delta, \theta) = 0$. For Lasso, it follows from Lemma~\ref{lem:locallinear}, which implies $R^\lambda(\delta, \theta) = 0$ for $\delta$ small enough, except on a set of $\theta$ values with Lebesgue measure 0.

For values of $\theta$ where $\nabla g^{\lambda}(\theta)$ does not
exist, we somewhat arbitrarily set $\nabla g^{\lambda}(\theta)=0$
and define $R^{\lambda}(\delta;\theta)=g^{\lambda}(\theta+\delta)-g^{\lambda}(\theta)$
for these values; that way (\ref{eq:gradient_expansion_of_g_lambda})
always holds. 
Observe that due to Lemma \ref{lem:lipschitz}, $\left\Vert g^{\lambda}(\theta+\delta)-g^{\lambda}(\theta)\right\Vert \le\left\Vert \delta\right\Vert $
and $\left\Vert \nabla g^{\lambda}(\theta)\right\Vert \le1$ (whenever
the gradient exists), so 
\begin{equation}
\sup_{\lambda,\theta}\left\Vert R^{\lambda}(\delta;\theta)\right\Vert \le2\left\Vert \delta\right\Vert .\label{eq:pf:8}
\end{equation}

\subsubsection*{Step 1:}
We first show that 
\begin{equation}
\textrm{CV}_{n}(\lambda)=\sum_{i=1}^{n}l_{n}\left(\hat{\theta}_{n}^{\lambda},Z_n^{i}\right)+\frac{1}{\sqrt{n}}\sum_{i=1}^{n}\left\langle \tilde{\theta}_{n}^{\lambda,-i}-\tilde{\theta}_{n}^{\lambda},\sqrt{n}\nabla_{\theta}l_{n}\left(\tilde{\theta}_{n}^{\lambda},Z_{n}^{i}\right)\right\rangle +o_{\mu_{n}}(1),\label{eq:CV_Taylor_expansion}
\end{equation}
uniformly over $\lambda$.
By a first order Taylor expansion, 
\begin{align*}
\textrm{CV}_{n}(\lambda) & =\sum_{i=1}^{n}l_{n}\left(\hat{\theta}_{n}^{\lambda,-i},Z_{n}^{i}\right)\\
 & =\sum_{i=1}^{n}l_{n}\left(\hat{\theta}_{n}^{\lambda},Z_{n}^{i}\right)+\frac{1}{\sqrt{n}}\sum_{i=1}^{n}\left\langle \hat{\theta}_{n}^{\lambda,-i}-\hat{\theta}_{n}^{\lambda},\sqrt{n}\nabla_{\theta}l_{n}\left(\hat{\theta}_{n}^{\lambda},Z_{n}^{i}\right)\right\rangle \\
 & \qquad+\frac{1}{\sqrt{n}}\sum_{i=1}^{n}r_{n}\left(\hat{\theta}_{n}^{\lambda,-i},\hat{\theta}_{n}^{\lambda},Z_{n}^{i}\right),
\end{align*}
where 
\[
\left|r_{n}\left(\hat{\theta}_{n}^{\lambda,-i},\hat{\theta}_{n}^{\lambda},Z_{n}^{i}\right)\right|\le B_{n}(Z_{n}^{i})\cdot\left\Vert \hat{\theta}_{n}^{\lambda,-i}-\hat{\theta}_{n}^{\lambda}\right\Vert ^{2}
\]
by Assumption~\ref{as:loss_fn_bounds}. Hence, by the Cauchy-Schwarz inequality, 
\begin{align*}
 & \left|\frac{1}{\sqrt{n}}\sum_{i=1}^{n}r_{n}\left(\hat{\theta}_{n}^{\lambda,-i},\hat{\theta}_{n}^{\lambda},Z_{n}^{i}\right)\right|\\
 & \le\left(\frac{1}{n}\sum_{i=1}^{n}\left|B_{n}(Z_{n}^{i})\right|^{2}\right)^{1/2}\left(\sum_{i=1}^{n}\left\Vert \hat{\theta}_{n}^{\lambda,-i}-\hat{\theta}_{n}^{\lambda}\right\Vert ^{4}\right)^{1/2}.
\end{align*}
The first term on the right hand side of the above expression is $O_{\mu_{n}}(1)$
by Assumption~\ref{as:loss_fn_bounds}, while the second term is $O_{\mu_{n}}(n^{-1/2})$
by (\ref{eq:requirement_1}) and the two requirements of Assumption~\ref{as:loss_fn_bounds} since 
\begin{align}\label{eq:pf:bound_on_sum}
& \sup_{\lambda}\sum_{i=1}^{n}\left\Vert \hat{\theta}_{n}^{\lambda,-i}-\hat{\theta}_{n}^{\lambda}\right\Vert ^{4} \nonumber \\
&\le \frac{1}{n^{2}\mu^{4}}\sum_{i=1}^{n}\left\Vert \sqrt{n}\nabla_{\theta}l_{n}(\hat{\theta}_{n},Z_{n}^{i})\right\Vert ^{4} \nonumber\\
&\le \frac{8}{n^{2}\mu^{4}}\sum_{i=1}^{n}\left\Vert \sqrt{n}\nabla_{\theta}l_{n}(\hat{\theta}_{n},Z_{n}^{i}) - \sqrt{n}\nabla_{\theta}l_{n}(\theta_0,Z_{n}^{i})\right\Vert ^{4} +
\frac{8}{n^{2}\mu^{4}}\sum_{i=1}^{n}\left\Vert \sqrt{n}\nabla_{\theta}l_{n}(\theta_0,Z_{n}^{i})\right\Vert ^{4} \nonumber\\
&= O_{\mu_{n}}(n^{-1}),
\end{align}
so the expression overall is $O_{\mu_{n}}(n^{-1/2})$.

We now show that, uniformly over $\lambda$, one can approximate
\begin{equation}
\frac{1}{\sqrt{n}}\sum_{i=1}^{n}\left\langle \hat{\theta}_{n}^{\lambda,-i}-\hat{\theta}_{n}^{\lambda},\sqrt{n}\nabla_{\theta}l_{n}\left(\hat{\theta}_{n}^{\lambda},Z_{n}^{i}\right)\right\rangle .\label{eq:pf:3}
\end{equation}
with
\[
\frac{1}{\sqrt{n}}\sum_{i=1}^{n}\left\langle \tilde{\theta}_{n}^{\lambda,-i}-\tilde{\theta}_{n}^{\lambda},\sqrt{n}\nabla_{\theta}l_{n}\left(\tilde{\theta}_{n}^{\lambda},Z_{n}^{i}\right)\right\rangle .
\]
To this end, we first argue that (\ref{eq:pf:3}) can be approximated
with
\[
\frac{1}{\sqrt{n}}\sum_{i=1}^{n}\left\langle \hat{\theta}_{n}^{\lambda,-i}-\hat{\theta}_{n}^{\lambda},\sqrt{n}\nabla_{\theta}l_{n}\left(\tilde{\theta}_{n}^{\lambda},Z_{n}^{i}\right)\right\rangle .
\]
By the Cauchy-Schwarz inequality, the approximation error is bounded
by 
\begin{equation}
\sup_{\lambda\in\Lambda}\left(\sum_{i=1}^{n}\left\Vert \hat{\theta}_{n}^{\lambda,-i}-\hat{\theta}_{n}^{\lambda}\right\Vert ^{2}\right)^{1/2}\cdot\sup_{\lambda\in\Lambda}\left(\frac{1}{n}\sum_{i=1}^{n}\left\Vert \sqrt{n}\nabla_{\theta}l_{n}\left(\hat{\theta}_{n}^{\lambda},Z_{n}^{i}\right)-\sqrt{n}\nabla_{\theta}l_{n}\left(\tilde{\theta}_{n}^{\lambda},Z_{n}^{i}\right)\right\Vert ^{2}\right)^{1/2}.\label{eq:pf:4}
\end{equation}
The first term in (\ref{eq:pf:4}) is $O_{\mu_{n}}(1)$ by (\ref{eq:requirement_1}) and Assumption~\ref{as:loss_fn_bounds} (the argument is analogous to \ref{eq:pf:bound_on_sum}).
The second term in (\ref{eq:pf:4}) is $o_{\mu_{n}}(1)$ under (\ref{eq:pf:requirement_2})
and Assumption~\ref{as:loss_fn_bounds}(ii). 

It then remains to show 
\begin{multline*}
\sup_{\lambda\in\Lambda}\left|\frac{1}{\sqrt{n}}\sum_{i=1}^{n}\left\langle \hat{\theta}_{n}^{\lambda,-i}-\hat{\theta}_{n}^{\lambda},\sqrt{n}\nabla_{\theta}l_{n}\left(\tilde{\theta}_{n}^{\lambda},Z_{n}^{i}\right)\right\rangle\right.\\
 \left.-\frac{1}{\sqrt{n}}\sum_{i=1}^{n}\left\langle \tilde{\theta}_{n}^{\lambda,-i}-\tilde{\theta}_{n}^{\lambda},\sqrt{n}\nabla_{\theta}l_{n}\left(\tilde{\theta}_{n}^{\lambda},Z_{n}^{i}\right)\right\rangle \right|=o_{\mu_{n}}(1).
\end{multline*}

By the Cauchy-Schwarz inequality, the expression on the left is bounded
by
\[
\sup_{\lambda\in\Lambda}\left(\sum_{i=1}^{n}\left\Vert \left(\hat{\theta}_{n}^{\lambda,-i}-\tilde{\theta}_{n}^{\lambda,-i}\right)-\left(\hat{\theta}_{n}^{\lambda}-\tilde{\theta}_{n}^{\lambda}\right)\right\Vert ^{2}\right)^{1/2}\cdot\sup_{\lambda}\left(\frac{1}{n}\sum_{i=1}^{n}\left\Vert \sqrt{n}\nabla_{\theta}l_{n}\left(\tilde{\theta}_{n}^{\lambda},Z_{n}^{i}\right)\right\Vert ^{2}\right)^{1/2}.
\]
The second term in the above expression is $O_{\mu_n}(1)$ by Assumption~\ref{as:loss_fn_bounds}(ii) (4th moment bound on scores) and the Lipschitz condition, since $\tilde{\theta}_n^\lambda$ is bounded in probability. At the end of this proof, we analyze the first term, showing that
$$
\sup_{\lambda\in\Lambda}\left(\sum_{i=1}^{n}\left\Vert \left(\hat{\theta}_{n}^{\lambda,-i}-\tilde{\theta}_{n}^{\lambda,-i}\right)-\left(\hat{\theta}_{n}^{\lambda}-\tilde{\theta}_{n}^{\lambda}\right)\right\Vert ^{2}\right) = o_{\mu_n}(1). 
$$
Combining the above results proves (\ref{eq:CV_Taylor_expansion}). 

\subsubsection*{Step 2:}

Next, we show that uniformly over $\lambda$, the term 
\begin{equation}
\frac{1}{\sqrt{n}}\sum_{i=1}^{n}\left\langle \tilde{\theta}_{n}^{\lambda,-i}-\tilde{\theta}_{n}^{\lambda},\sqrt{n}\nabla_{\theta}l_{n}\left(\tilde{\theta}_{n}^{\lambda},Z_{n}^{i}\right)\right\rangle \label{eq:pf:5}
\end{equation}
in (\ref{eq:CV_Taylor_expansion}) can be approximated
by 
\[
\frac{1}{\sqrt{n}}\sum_{i=1}^{n}\left\langle \tilde{\theta}_{n}^{\lambda,-i}-\tilde{\theta}_{n}^{\lambda}, -X_n^i\right\rangle .
\]
Indeed, under Assumption~\ref{as:loss_fn_bounds}(iii), we have 
\begin{align} \label{eq:pf:6}
 & \frac{1}{\sqrt{n}}\sum_{i=1}^{n}\left\langle \tilde{\theta}_{n}^{\lambda,-i}-\tilde{\theta}_{n}^{\lambda},\sqrt{n}\nabla_{\theta}l_{n}\left(\tilde{\theta}_{n}^{\lambda},Z_{n}^{i}\right)\right\rangle -\frac{1}{\sqrt{n}}\sum_{i=1}^{n}\left\langle \tilde{\theta}_{n}^{\lambda,-i}-\tilde{\theta}_{n}^{\lambda},-X_n^i\right\rangle \nonumber \\
 & =\frac{1}{n}\sum_{i=1}^{n}\left\langle \tilde{\theta}_{n}^{\lambda,-i}-\tilde{\theta}_{n}^{\lambda},n\nabla_{\theta}^{2}l_{n}\left(\theta_{0},Z_{n}^{i}\right)\cdot\left(\tilde{\theta}_{n}^{\lambda}-\theta_{0}\right)\right\rangle +\frac{1}{n}\sum_{i=1}^{n}\left\langle \tilde{\theta}_{n}^{\lambda,-i}-\tilde{\theta}_{n}^{\lambda},\Delta_{n}\left(\tilde{\theta}_{n}^{\lambda}-\theta_{0},Z_{n}^{i}\right)\right\rangle ,
\end{align}
where 
\[
\left\Vert \Delta_{n}\left(\tilde{\theta}_{n}^{\lambda}-\theta_{0},Z_{n}^{i}\right)\right\Vert \le C_{n}(Z_{n}^{i})\cdot\left\Vert \tilde{\theta}_{n}^{\lambda}-\theta_{0}\right\Vert ^{2},
\]
and the function $C_n(\cdot)$ is defined in  Assumption~\ref{as:loss_fn_bounds}(iii). 

By the Cauchy-Schwarz inequality, Assumption~\ref{as:loss_fn_bounds} and (\ref{eq:requirement_3}),
\begin{align*}
 & \sup_{\lambda\in\Lambda}\left|\frac{1}{n}\sum_{i=1}^{n}\left\langle \tilde{\theta}_{n}^{\lambda,-i}-\tilde{\theta}_{n}^{\lambda},n\nabla_{\theta}^{2}l_{n}\left(\theta_{0},Z_{n}^{i}\right)\cdot\left(\tilde{\theta}_{n}^{\lambda}-\theta_{0}\right)\right\rangle \right|\\
 & \le n^{-1/2}\cdot\sup_{\lambda\in\Lambda}\left(\sum_{i=1}^{n}\left\Vert \tilde{\theta}_{n}^{\lambda,-i}-\tilde{\theta}_{n}^{\lambda}\right\Vert ^{2}\right)^{1/2}\cdot\sup_{\lambda\in\Lambda}\left(\frac{1}{n}\sum_{i=1}^{n}\left\Vert n\nabla_{\theta}^{2}l_{n}\left(\theta_{0},Z_{n}^{i}\right)\right\Vert ^{2}\right)^{1/2}\cdot\sup_{\lambda\in\Lambda}\left\Vert \tilde{\theta}_{n}^{\lambda}-\theta_{0}\right\Vert \\
 & =n^{-1/2}\cdot O_{\mu_{n}}(1)\cdot O_{\mu_{n}}(1)\cdot O_{\mu_{n}}(1)=O_{\mu_{n}}(n^{-1/2}).
\end{align*}
This proves that the first term in the right hand side of (\ref{eq:pf:6})
is $o_{\mu_{n}}(1)$ uniformly over $\lambda$. By an analogous argument,
the second term in the right hand side of (\ref{eq:pf:6}) is also $o_{\mu_{n}}(1)$ uniformly over $\lambda$.

\subsubsection*{Step 3:}

It thus remains to show
\[
\frac{1}{\sqrt{n}}\sum_{i=1}^{n}\left\langle \tilde{\theta}_{n}^{\lambda,-i}-\tilde{\theta}_{n}^{\lambda},\sqrt{n}\nabla_{\theta}l_{n}\left(\theta_{0},Z_{n}^{i}\right)\right\rangle 
\]
is asymptotically equivalent to the degrees of freedom term in SURE. 

By the definition of $R^{\lambda}(\cdot)$, we may write 
\[
\tilde{\theta}_{n}^{\lambda,-i}-\tilde{\theta}_{n}^{\lambda} = -\frac{1}{\sqrt{n}} X_n^i - \frac{1}{\sqrt{n}} \nabla g^{\lambda}\left(\tilde{\theta}_{n}\right)^{\intercal} X_n^i + R^{\lambda}(\tilde{\theta}_{n}^{-i}-\tilde{\theta}_{n};\tilde{\theta}_{n}).
\]
We can thus expand 
\begin{align}
 & \frac{1}{\sqrt{n}}\sum_{i=1}^{n}\left\langle \tilde{\theta}_{n}^{\lambda,-i}-\tilde{\theta}_{n}^{\lambda},\sqrt{n}\nabla_{\theta}l_{n}\left(\theta_{0},Z_{n}^{i}\right)\right\rangle \nonumber \\
 & =\frac{1}{n}\sum_{i=1}^{n}\left\langle -X_n^i,\sqrt{n}\nabla_{\theta}l_{n}\left(\theta_{0},Z_{n}^{i}\right)\right\rangle +
 \frac{1}{n}\sum_{i=1}^{n}\left\langle \nabla g^{\lambda}\left(\tilde{\theta}_{n}\right)^{\intercal} X_n^i, -\sqrt{n}\nabla_{\theta}l_{n}\left(\theta_{0},Z_{n}^{i}\right)\right\rangle \nonumber \\
 & \ +\frac{1}{\sqrt{n}}\sum_{i=1}^{n}\left\langle R^{\lambda}(\tilde{\theta}_{n}^{-i}-\tilde{\theta}_{n};\tilde{\theta}_{n}), \sqrt{n}\nabla_{\theta}l_{n}\left(\theta_{0},Z_{n}^{i}\right)\right\rangle .\label{eq:pf:9}
\end{align}

The first term in (\ref{eq:pf:9}) is independent of $\lambda$ and
can therefore be neglected. 

The second term in (\ref{eq:pf:9}) is asymptotically equivalent to
the degrees of freedom term in SURE. 
Indeed, since $\tilde{\theta}_{n}^{-i}-\tilde{\theta}_{n} = \nabla_{\theta}l_{n}\left(\theta_{0},Z_{n}^{i}\right)$,
we can write
\begin{align*}
 &  \frac{1}{n}\sum_{i=1}^{n}\left\langle \nabla g^{\lambda}\left(\tilde{\theta}_{n}\right)^{\intercal} X_n^i, -\sqrt{n}\nabla_{\theta}l_{n}\left(\theta_{0},Z_{n}^{i}\right)\right\rangle \\
 & =\frac{1}{n}\sum_{i=1}^{n}\left\langle\nabla g^{\lambda}\left(\tilde{\theta}_{n}\right)^{\intercal} X_n^i, X_n^i\right\rangle \\
 & =\textrm{Tr}\left[\nabla g^{\lambda}\left(\tilde{\theta}_{n}\right)^{\intercal}\hat{\Sigma}_{n}\right],
\end{align*}
where 
\[
\hat{\Sigma}_{n}:=\frac{1}{n}\sum_{i=1}^{n} (X_n^i)  (X_n^i)^\intercal.
\]
But by the law of large of numbers, which can be applied here due
to Assumption~\ref{as:loss_fn_bounds}(ii), $\hat{\Sigma}_{n}=\Sigma+o_{\mu_{n}}(1)$.
We thus conclude that 
\[
 \frac{1}{n}\sum_{i=1}^{n}\left\langle \nabla g^{\lambda}\left(\tilde{\theta}_{n}\right)^{\intercal} X_n^i, -\sqrt{n}\nabla_{\theta}l_{n}\left(\theta_{0},Z_{n}^{i}\right)\right\rangle =
 \textrm{Tr}\left[\nabla g^{\lambda}\left(\tilde{\theta}_{n}\right)^{\intercal} \Sigma \right]+o_{\mu_{n}}(1),
\]
uniformly over $\lambda\in\Lambda$.

It remains to show the third term in (\ref{eq:pf:9}) is negligible,
i.e.,
\[
\frac{1}{\sqrt{n}}\sum_{i=1}^{n}\left\langle R^{\lambda}(\tilde{\theta}_{n}^{-i}-\tilde{\theta}_{n};\tilde{\theta}_{n}), \sqrt{n}\nabla_{\theta}l_{n}\left(\theta_{0},Z_{n}^{i}\right)\right\rangle=o_{\mu_{n}}(1).
\]
Recall that $\tilde{\theta}_{n}^{-i}-\tilde{\theta}_{n} = -X_n^i/\sqrt{n}$. Fix some $a\in(0,1/2)$ and $C<\infty$, and expand
\begin{align*}
 & \frac{1}{\sqrt{n}}\sum_{i=1}^{n}\left\langle R^{\lambda}(\tilde{\theta}_{n}^{-i}-\tilde{\theta}_{n};\tilde{\theta}_{n}), \sqrt{n}\nabla_{\theta}l_{n}\left(\theta_{0},Z_{n}^{i}\right)\right\rangle
  = \frac{1}{\sqrt{n}}\sum_{i=1}^{n}\left\langle R^{\lambda}(-X_n^i/\sqrt{n};\tilde{\theta}_{n}), -X_n^i\right\rangle \\
 & =\frac{1}{\sqrt{n}}\sum_{i=1}^{n}\left\langle R^{\lambda}(-X_n^i/\sqrt{n};\tilde{\theta}_{n})\cdot\mathbb{I}\left\{ \left\Vert X_n^i \right\Vert \ge Cn^{a}\right\} , -X_n^i \right\rangle \\
 & \ +\frac{1}{\sqrt{n}}\sum_{i=1}^{n}\left\langle R^{\lambda}(-X_n^i/\sqrt{n};\tilde{\theta}_{n})\cdot\mathbb{I}\left\{ \left\Vert X_n^i \right\Vert <Cn^{a}\right\} , -X_n^i \right\rangle \\
 & :=A_{n1}^{\lambda}+A_{n2}^{\lambda}.
\end{align*}
We analyze the terms $A_{n1}^{\lambda}$ and $A_{n2}^{\lambda}$ separately.
For the term $A_{n1}^{\lambda}$, observe that by (\ref{eq:pf:8}),
\[
\sup_{\lambda\in\Lambda}\vert A_{n1}^{\lambda}\vert  \le\frac{2}{n}\sum_{i=1}^{n}\left\Vert X_n^i \right\Vert^2 \cdot\mathbb{I}\left\{ \left\Vert X_n^i\right\Vert \ge Cn^{a}\right\}. 
\]
Consequently, under Assumption~\ref{as:loss_fn_bounds}(ii) and the given choice of $a$,
\begin{align*}
\mathbb{E}_{\mu_{n}}\left[\vert A_{n1}^{\lambda}\vert\right] & \le\frac{2}{C^{2}n^{2a}}\mathbb{E}_{\mu_{n}}\left[\left\Vert X_n^i\right\Vert ^{4}\right]\to0
\end{align*}
as $n\to\infty$. Thus, $\sup_{\lambda\in\Lambda}A_{n1}^{\lambda}=o_{\mu_{n}}(1)$. 

Next, we show $\sup_{\lambda\in\Lambda}A_{n2}^{\lambda}=o_{\mu_{n}}(1)$.
Due to exchangeability over $i$, this follows if we show that 
\begin{equation}
\lim_{n\to\infty}\sqrt{n}\mathbb{E}_{\mu_{n}}\left[\left|\left\langle R^{\lambda}(-X_n^i/\sqrt{n};\tilde{\theta}_{n})\cdot\mathbb{I}_{\Gamma_{i}} , -X_n^i \right\rangle \right|\right]=0,\label{eq:2}
\end{equation}
where we use $\mathbb{I}_{\Gamma_{i}}$ as a short-hand for $\mathbb{I}\left\{ \left\Vert X_n^i\right\Vert <Cn^{a}\right\} $.
Now, by the Cauchy-Schwarz inequality, 
\begin{align*}
 & \sqrt{n}\mathbb{E}_{\mu_{n}}\left[\left|\left\langle R^{\lambda}(-X_n^i/\sqrt{n};\tilde{\theta}_{n})\cdot\mathbb{I}_{\Gamma_{i}} , -X_n^i \right\rangle \right|\right]\\
 & \le\mathbb{E}_{\mu_{n}}^{1/2}\left[\sup_{\lambda\in\Lambda}\frac{\left\Vert R^{\lambda}\left(-X_n^i/\sqrt{n};\tilde{\theta}_{n}\right)\right\Vert ^{2}}{\left\Vert -X_n^i/\sqrt{n}\right\Vert ^{2}}\cdot\mathbb{I}_{\Gamma_{i}}\right]\cdot\mathbb{E}_{\mu_{n}}^{1/2}\left[\left\Vert X_n^i\right\Vert ^{4}\right].
\end{align*}
But $\mathbb{E}_{\mu_{n}}^{1/2}\left[\left\Vert X_n^i\right\Vert ^{4}\right] \le M < \infty$ under Assumption~\ref{as:loss_fn_bounds}(ii), so (\ref{eq:2}) would follow if we show that
\begin{equation}
\lim_{n\to\infty}\mathbb{E}_{\mu_{n}}\left[\sup_{\lambda\in\Lambda}\frac{\left\Vert R^{\lambda}\left(-X_n^i/\sqrt{n};\tilde{\theta}_{n}\right)\right\Vert ^{2}}{\left\Vert -X_n^i/\sqrt{n}\right\Vert ^{2}}\cdot\mathbb{I}_{\Gamma_{i}}\right]=0.\label{eq:3}
\end{equation}

To prove (\ref{eq:3}), observe that it is without loss of generality
to suppose $R^{\lambda}(\delta;\theta)=R^{\lambda}(\left\Vert \delta\right\Vert ;\theta)$,
i.e., that it depends only on $\left\Vert \delta\right\Vert $, and
that $R^{\lambda}(\left\Vert \delta\right\Vert ;\theta)/\left\Vert \delta\right\Vert $
is increasing in $\left\Vert \delta\right\Vert $. Otherwise, we can
simply define 
\[
\bar{R}^{\lambda}(\left\Vert \delta\right\Vert ;\theta)=\sup_{\left\Vert \delta^{\prime}\right\Vert \le\left\Vert \delta\right\Vert }\frac{R^{\lambda}(\delta^{\prime};\theta)}{\left\Vert \delta^{\prime}\right\Vert },
\]
and this would satisfy these two conditions while still retaining
the property (\ref{eq:pf:7}). Consequently, the left hand side of
(\ref{eq:3}) can be bounded as 
\begin{align*}
\mathbb{E}_{\mu_{n}}\left[\sup_{\lambda\in\Lambda}\frac{\left\Vert R^{\lambda}\left(\Vert X_n^i/\sqrt{n} \Vert;\tilde{\theta}_{n}\right)\right\Vert ^{2}}{\left\Vert X_n^i/\sqrt{n}\right\Vert ^{2}}\cdot\mathbb{I}_{\Gamma_{i}}\right] & \le\mathbb{E}_{\mu_{n}}\left[\sup_{\lambda\in\Lambda}\left\Vert \frac{R^{\lambda}\left(Cn^{a - \frac{1}{2}};\tilde{\theta}_{n}\right)}{Cn^{a - \frac{1}{2}}}\right\Vert ^{2}\right]\\
 & =\mathbb{E}_{\mu_{n}}\left[\sup_{\lambda\in\Lambda}\left\Vert B^{\lambda}(Cn^{a - \frac{1}{2}},\tilde{\theta}_{n})\right\Vert ^{2}\right],
\end{align*}
where 
\[
B^{\lambda}(\delta,\tilde{\theta}_{n}):=\frac{\left\Vert R^{\lambda}\left(\delta;\tilde{\theta}_{n}\right)\right\Vert }{\left\Vert \delta\right\Vert ^ {}}.
\]
We now bound 
\[
\mathbb{E}_{\mu_{n}}\left[\sup_{\lambda\in\Lambda}\left\Vert B^{\lambda}(Cn^{a - \frac{1}{2}},\tilde{\theta}_{n})\right\Vert ^{2}\right].
\]
Fix some $\epsilon>0$. By the requirement that $R^{\lambda}(\left\Vert \delta\right\Vert ;\theta)/\left\Vert \delta\right\Vert $
is increasing in $\left\Vert \delta\right\Vert $, along with the fact $a < 1/2$, there exists $\bar{n}$
large enough so that $B^{\lambda}(Cn^{a - \frac{1}{2}},\theta)\le B^{\lambda}(\epsilon,\theta)$
for each $n \ge \bar{n}$,  $\theta\in\mathbb{R}^{d}$ and $\lambda\in\Lambda$. Now,
it is straightforward to show 
\[
\tilde{\theta}_{n}\xrightarrow[\mu_{n}]{d}Z\sim N(\theta_{0},\Sigma).
\]
We then have
\begin{align*}
 & \lim_{n\to\infty}\mathbb{E}_{\mu_{n_{}}}\left[\sup_{\lambda\in\Lambda}\left\Vert B^{\lambda}(Cn^{a - \frac{1}{2}},\tilde{\theta}_{n})\right\Vert ^{2}\right]\\
 & \le\lim_{n\to\infty}\mathbb{E}_{\mu_{n_{}}}\left[\sup_{\lambda\in\Lambda}\left\Vert B^{\lambda}(\epsilon,\tilde{\theta}_{n})\right\Vert ^{2}\right]=\mathbb{E}\left[\sup_{\lambda\in\Lambda}\left\Vert B^{\lambda}(\epsilon,Z)\right\Vert ^{2}\right],
\end{align*}
where the equality follows from the properties of weak convergence
since equation (\ref{eq:pf:8}) implies $\sup_{\lambda\in\Lambda}\left\Vert B^{\lambda}(\epsilon,\theta)\right\Vert \le2$
uniformly over $\theta$. But (\ref{eq:pf:R_bound})
implies $\lim_{\epsilon^{\prime}\to0}\sup_{\lambda\in\Lambda}B^{\lambda}(\epsilon^{\prime},\theta)=0$
for every $\theta\in\mathbb{R}^{d}$ excluding a set of Lebesgue measure
$0$. Since the Gaussian distribution is absolutely continuous with respect to
the Lebesgue measure, it then follows by the dominated convergence
theorem that 
\[
\lim_{\epsilon^{\prime}\to0}\mathbb{E}\left[\sup_{\lambda\in\Lambda}\left\Vert B^{\lambda}(\epsilon^{\prime},Z)\right\Vert ^{2}\right]=0.
\]
Since $\epsilon>0$ was arbitrary, we conclude 
\[
\lim_{n\to\infty}\mathbb{E}_{\mu_{n}}\left[\sup_{\lambda\in\Lambda}\left\Vert B^{\lambda}(Cn^{a - \frac{1}{2}},\tilde{\theta}_{n})\right\Vert ^{2}\right]=0.
\]
This proves (\ref{eq:3}).\\

It remains to prove the claim, made in Step 2, that 
$$
\sup_{\lambda\in\Lambda}\left(\sum_{i=1}^{n}\left\Vert \left(\hat{\theta}_{n}^{\lambda,-i}-\tilde{\theta}_{n}^{\lambda,-i}\right)-\left(\hat{\theta}_{n}^{\lambda}-\tilde{\theta}_{n}^{\lambda}\right)\right\Vert ^{2}\right) = o_{\mu_n}(1). 
$$
For the remainder of this proof, we make a case distinction between Ridge penalties and Lasso penalties.
\subsubsection*{Step 4 (Ridge):}
We first specialize to the case of quadratic penalties. Define $\tilde{L}_{n}^{-i}(\theta)=\frac{1}{2}\left\Vert \theta-\tilde{\theta}_{n}^{-i}\right\Vert ^{2}$.
Observe that $\tilde{\theta}_{n}^{\lambda,-i}=\argmin_{\theta}\left\{ \tilde{L}_{n}^{-i}(\theta)+\lambda\pi(\theta)\right\} $.
Consequently, 
\begin{align*}
 & \nabla_{\theta}\left\{ \tilde{L}_{n}^{-i}(\hat{\theta}_{n}^{\lambda,-i})+\lambda\pi(\hat{\theta}_{n}^{\lambda,-i})\right\} \\
 & =\nabla_{\theta}\left\{ \tilde{L}_{n}^{-i}(\hat{\theta}_{n}^{\lambda,-i})+\lambda\pi(\hat{\theta}_{n}^{\lambda,-i})\right\} -\nabla_{\theta}\left\{ \tilde{L}_{n}^{-i}(\tilde{\theta}_{n}^{\lambda,-i})+\lambda\pi(\tilde{\theta}_{n}^{\lambda,-i})\right\} \\
 & =\left\{ \hat{\theta}_{n}^{\lambda,-i}-\tilde{\theta}_{n}^{\lambda,-i}\right\} +\lambda\nabla_{\theta}\left\{ \pi(\hat{\theta}_{n}^{\lambda,-i})-\pi(\tilde{\theta}_{n}^{\lambda,-i})\right\} .
\end{align*}
But $\nabla_{\theta}\left\{ L_{n}^{-i}(\hat{\theta}_{n}^{\lambda,-i})+\lambda\pi(\hat{\theta}_{n}^{\lambda,-i})\right\} =0$,
so we obtain
\[
\left\{ \hat{\theta}_{n}^{\lambda,-i}-\tilde{\theta}_{n}^{\lambda,-i}\right\} +\lambda\nabla_{\theta}\left\{ \pi(\hat{\theta}_{n}^{\lambda,-i})-\pi(\tilde{\theta}_{n}^{\lambda,-i})\right\} =\nabla_{\theta}\tilde{L}_{n}^{-i}(\hat{\theta}_{n}^{\lambda,-i})-\nabla_{\theta}L_{n}^{-i}(\hat{\theta}_{n}^{\lambda,-i}).
\]
In a similar vein, 
\[
\left\{ \hat{\theta}_{n}^{\lambda}-\tilde{\theta}_{n}^{\lambda}\right\} +\lambda\nabla_{\theta}\left\{ \pi(\hat{\theta}_{n}^{\lambda})-\pi(\tilde{\theta}_{n}^{\lambda})\right\} =\nabla_{\theta}\tilde{L}_{n}(\hat{\theta}_{n}^{\lambda})-\nabla_{\theta}L_{n}(\hat{\theta}_{n}^{\lambda}).
\]

When $\pi(\theta)$ is the ridge penalty $\frac{1}{2} \theta^\intercal A^{-1} \theta$, we have
\[
\left(\hat{\theta}_{n}^{\lambda,-i}-\tilde{\theta}_{n}^{\lambda,-i}\right)-\left(\hat{\theta}_{n}^{\lambda}-\tilde{\theta}_{n}^{\lambda}\right)= (I + \lambda A^{-1})^{-1}\left\{ \nabla_{\theta}\left(\tilde{L}_{n}^{-i}-L_{n}^{-i}\right)(\hat{\theta}_{n}^{\lambda,-i})-\nabla_{\theta}\left(\tilde{L}_{n}-L_{n}\right)(\hat{\theta}_{n}^{\lambda})\right\} .
\]

By a third order Taylor expansion,
\begin{align*}
 & L_{n}^{-i}(\theta)-L_{n}^{-i}(\theta_{0})\\
 & =\frac{1}{\sqrt{n}}\left\{ \sqrt{n}\nabla_{\theta}L_{n}^{-i}(\theta_{0})\right\} ^{\intercal}(\theta-\theta_{0})+\frac{1}{2n}(\theta-\theta_{0})^{\intercal}\left\{ n\nabla_{\theta}^{2}L_{n}^{-i}(\theta_{0})\right\} (\theta-\theta_{0})\\
 & \quad+\frac{1}{6n^{3/2}}D_\theta^{3}L_{n}^{-i}(\bar{\theta})[\theta-\theta_{0},\theta-\theta_{0},\theta-\theta_{0}],
\end{align*}
for some $\bar{\theta}$ between $\theta$ and $\theta_{0}$. At the
same time, 
\[
\tilde{L}_{n}^{-i}(\theta)-\tilde{L}_{n}^{-i}(\theta_{0})=\frac{1}{\sqrt{n}}\left\{ \sqrt{n}\nabla_{\theta}L_{n}^{-i}(\theta_{0})\right\} ^{\intercal}(\theta-\theta_{0})+\frac{1}{2}(\theta-\theta_{0})^{\intercal}(\theta-\theta_{0}).
\]
Therefore, by Assumption~\ref{as:loss_fn_bounds}(iv), which implies
\[
\mathbb{E}_{\mu_{n}}\left[\sup_{\theta\in\Theta}\left\Vert n^{2}D_\theta^{4}l_{n}\left(\theta,Z_{n}^{i}\right)\right\Vert ^{4}\right]\le M,
\]
we obtain
\begin{align*}
\nabla_{\theta}\left(\tilde{L}_{n}^{-i}-L_{n}^{-i}\right)(\theta) & =\left\{ \nabla_{\theta}^{2}L_{n}^{-i}(\theta_{0})-I\right\} (\theta-\theta_{0})+O_{\mu_{n}}(n^{-3/2}).
\end{align*}
In a similar vein, 
\begin{align*}
\nabla_{\theta}\left(\tilde{L}_{n}-L_{n}\right)(\theta) & =\left\{ \nabla_{\theta}^{2}L_{n}(\theta_{0})-I\right\} (\theta-\theta_{0})+O_{\mu_{n}}(n^{-3/2}).
\end{align*}

Taken together, we conclude 
\begin{align*}
 & \left\{ \nabla_{\theta}\left(\tilde{L}_{n}^{-i}-L_{n}^{-i}\right)(\hat{\theta}_{n}^{\lambda,-i})-\nabla_{\theta}\left(\tilde{L}_{n}-L_{n}\right)(\hat{\theta}_{n}^{\lambda})\right\} \\
 & =\left\{ \nabla_{\theta}^{2}L_{n}^{-i}(\theta_{0})-I\right\} \left(\hat{\theta}_{n}^{\lambda,-i}-\theta_{0}\right)-\left\{ \nabla_{\theta}^{2}L_{n}(\theta_{0})-I\right\} \left(\hat{\theta}_{n}^{\lambda}-\theta_{0}\right)+O_{\mu_{n}}(n^{-3/2})\\
 & =-\frac{1}{n}\left\{ n\nabla_{\theta}^{2}l_{n}(\theta_{0},Z_{n}^{i})\right\} \left(\hat{\theta}_{n}^{\lambda,-i}-\theta_{0}\right)+\left\{ \nabla_{\theta}^{2}L_{n}(\theta_{0})-I\right\} \left(\hat{\theta}_{n}^{\lambda,-i}-\hat{\theta}_{n}^{\lambda}\right)+O_{\mu_{n}}(n^{-3/2}).
\end{align*}
Hence, 
\begin{align*}
 & \sum_{i=1}^{n}\left\Vert \left(\hat{\theta}_{n}^{\lambda,-i}-\tilde{\theta}_{n}^{\lambda,-i}\right)-\left(\hat{\theta}_{n}^{\lambda}-\tilde{\theta}_{n}^{\lambda}\right)\right\Vert ^{2}\\
 & \le\frac{2}{n^{2}}\left(\sum_{i}\left\Vert n\nabla_{\theta}^{2}l_{n}(\theta_{0},Z_{n}^{i})\right\Vert ^{2}\cdot\left\Vert \hat{\theta}_{n}^{\lambda,-i}-\theta_0\right\Vert ^{2}\right)\\
 & \quad+2\left\Vert \nabla_{\theta}^{2}L_{n}(\theta_{0})-I\right\Vert ^{2}\cdot\left(\sum_{i}\left\Vert \hat{\theta}_{n}^{\lambda,-i}-\hat{\theta}_{n}^{\lambda}\right\Vert ^{2}\right)+O_{\mu_{n}}(n^{-3/2})\\
 & =o_{\mu_{n}}(1).
\end{align*}

\subsubsection*{Step 4 (Lasso):}
We finally prove that for the Lasso penalty, we again have that
$$
DD_{n} := \left(\sum_{i=1}^{n}\left\Vert \left(\hat{\theta}_{n}^{\lambda,-i}-\tilde{\theta}_{n}^{\lambda,-i}\right)-\left(\hat{\theta}_{n}^{\lambda}-\tilde{\theta}_{n}^{\lambda}\right)\right\Vert ^{2}\right) = o_{\mu_n}(1). 
$$
We consider the case of fixed $\lambda$; taking the supremum over $\lambda$ is trivial when $\Lambda$ is finite. We will also assume for notational simplicity that $A=I$, so that $h=\theta$; the general case (where $h = A^{-1}\cdot\theta$) follows immediately.

Let now
$$
D_{n}^{i} := \left(\hat{\theta}_{n}^{\lambda,-i}-\tilde{\theta}_{n}^{\lambda,-i}\right)-\left(\hat{\theta}_{n}^{\lambda}-\tilde{\theta}_{n}^{\lambda}\right).
$$
Denote $\eta = sign(\tilde{\theta}_{n}^{\lambda})$, and the set of active coordinates as $J=\{j: \eta_{j}\neq 0\}$. Define the following three event indicators, where the $sign$ function is applied component-wise, and $\rho_{n} > 0$ for some deterministic sequence $\rho_{n}$ such that $\rho_{n} \cdot n^{1/4} \to \infty$ and $\rho_{n}\to 0$:
$$
\begin{aligned}
A_{n} &= \mathbf{1}
\left(sign(\widehat{\theta}_{n}^{\lambda,-i})=
sign(\tilde{\theta}_{n}^{\lambda,-i}) = 
sign(\widehat{\theta}_{n}^{\lambda}) =
\eta \text{ for all }i\right),\\
B_{n} &= (1-A_{n})\cdot\mathbf{1}\Big(sign(\widehat{\theta}_{n}^{\lambda}) \neq \eta\text{ or }  sign(\tilde{\theta}_{n} +t + g^{\lambda}(\tilde{\theta}_{n}+t)) \neq \eta
 \text{ for some } \lVert t \rVert <\rho_{n} \Big)\\
C_{n} &= 1 - A_{n} - B_{n}. 
\end{aligned}
$$
Then, by construction, $A_{n} + B_{n} + C_{n} = 1$ for all $n,i$, so that
$$
DD_{n} = \underbrace{ A_{n}\cdot\sum_{i=1}^{n} \lVert D_{n}^{i} \rVert^{2}}_{\text{First sum}} + 
\underbrace{ B_{n}\cdot\sum_{i=1}^{n} \lVert D_{n}^{i} \rVert^{2}}_{\text{Second sum}} + 
\underbrace{ C_{n}\cdot\sum_{i=1}^{n} \lVert D_{n}^{i} \rVert^{2}}_{\text{Third sum}}. 
$$
To show that $DD_{n} \to 0$, we will consider each of these three sums separately.\\

\textbf{First sum}
Conditional on $A_{n}=1$, the signs for all estimators of $\theta$ under consideration coincide with $\eta$ . The first order conditions for the active coordinates for each of the estimators can therefore be written as
$$
\begin{aligned}
\nabla_{J} L_{n}(\widehat{\theta}_{n}^{\lambda}) +\lambda \cdot\eta_{J} &=0\\
\nabla_{J} L_{n}(\widehat{\theta}_{n}^{\lambda,-i}) - \nabla_{J} l_{n}(\widehat{\theta}_{n}^{\lambda,-i},Z_i^{n}) +\lambda \cdot\eta_{J}
&=0\\
(\tilde{\theta}_{n}^{\lambda}-\tilde{\theta}_{n})_{J} +\lambda \cdot\eta_{J} &=0\\
(\tilde{\theta}_{n}^{\lambda,-i}-\tilde{\theta}_{n})_{J} - \nabla_{J} l_{n}(\theta_{0},Z_i^{n}) +\lambda \cdot\eta_{J}
&=0.
\end{aligned}
$$
Taking differences of the first two and of the second two equations, we get
$$
\begin{aligned}
\nabla_{J} L_{n}(\widehat{\theta}_{n}^{\lambda}) -
\nabla_{J} L_{n}(\widehat{\theta}_{n}^{\lambda,-i}) &= \nabla_{J} l_{n}(\widehat{\theta}_{n}^{\lambda,-i},Z_i^{n}),\\
(\tilde{\theta}_{n}^{\lambda} - \tilde{\theta}_{n}^{\lambda,-i})_{J} &= \nabla_{J} l_{n}(\theta_{0},Z_i^{n}).
\end{aligned}
$$
The first of these equations can be rewritten as
$$
(\widehat{\theta}_{n}^{\lambda} -\widehat{\theta}_{n}^{\lambda,-i})_{J} = 
(\nabla_{J}^{2} L_{n}(\bar{\theta}_{n}^{i}))^{-1} \cdot
\nabla_{J} l_{n}(\widehat{\theta}_{n}^{\lambda,-i},Z_i^{n})
$$
for some intermediate point $\bar{\theta}^{i}$. Combining, we get $D_{n,J^{c}}^{i} = 0$ (for the inactive coordinates, the double difference vanishes) and
$$
\begin{aligned}
D_{n,J}^{i} &= \left[(\nabla_{J}^{2} L_{n}(\bar{\theta}_{n}^{i}))^{-1} \cdot
\nabla_{J} l_{n}(\widehat{\theta}_{n}^{\lambda,-i},Z_i^{n})\right] - \nabla_{J} l_{n}(\theta_{0},Z_i^{n})\\
&= \left[ (\nabla_{J}^{2} L_{n}(\bar{\theta}_{n}^{i}))^{-1} - I_{J}\right] \cdot
\nabla_{J} l_{n}(\widehat{\theta}_{n}^{\lambda,-i},Z_i^{n})
+ \left[ \nabla_{J} l_{n}(\widehat{\theta}_{n}^{\lambda,-i},Z_i^{n}) -  \nabla_{J} l_{n}(\theta_{0},Z_i^{n})\right] .
\end{aligned}
$$
and thus, conditional on $A_{n}=1$,
$$
\begin{aligned}
\sum_{i=1}^{n} \lVert D_{n}^{i} \rVert^{2} &= 
\sum_{i=1}^{n} \lVert D_{n,J}^{i} \rVert^{2} \\
&\leq 2 
\max_{i} \left\lVert   (\nabla_{J}^{2} L_{n}(\bar{\theta}_{n}^{i}))^{-1} - I_{J}\right\rVert^{2}\cdot 
\sum_i \left\lVert \nabla l_{n}(\widehat{\theta}_{n}^{\lambda,-i},Z_i^{n})\right\rVert^{2}\\
&+ 2 \sum_i \left \lVert \nabla l_{n}(\widehat{\theta}_{n}^{\lambda,-i},Z_i^{n}) -  \nabla l_{n}(\theta_{0},Z_i^{n})\right \rVert^{2}.
\end{aligned}
$$
Note that we dropped the $J$ subscript on gradients when taking the upper bound.
The $\max$ term goes to $0$ in probability by Lemma \ref{lem:ifrep}. The first sum is $O_{p}(1)$ given the bound on $4$th moments of the score in Assumption~\ref{as:loss_fn_bounds}.2. The second sum is $o_{p}(1)$ by the Lipschitz condition in Assumption~\ref{as:loss_fn_bounds}.2 and by $\widehat{\theta}_{n}^{\lambda,-i} - \theta_{0} = O_{p}\left( \frac{1}{\sqrt{ n }} \right)$.
It follows that $A_{n}\cdot\sum_{i=1}^{n} \lVert D_{n}^{i} \rVert^{2} = o_{p}(1)$.\\

\textbf{Second sum}
To control the second sum, we next show that $B_{n} = o_{\mu_{n}}(1)$.

Recall the following results that were shown previously:
\begin{itemize}
	\item By Lemma~\ref{lem:ifrep}, $\widehat{\theta}_{n}^{\lambda} - \tilde{\theta}_{n}^{\lambda} \to^{p} 0$.
	\item Also by Lemma~\ref{lem:ifrep}, $\sup_{\theta: \| \theta \| \le C} \nabla \epsilon_n(\theta) = o_{\mu_n}(1)$.
	\item By Corollary~\ref{cor:asymptotic_penalizederm}, $\widehat{\theta}_{n}^{\lambda} \to^{d} \widehat{\theta} + g^{\lambda}(\widehat{\theta})$, where $\widehat{\theta}\sim N(\theta_{0}, \Sigma)$.
	\item By Lemma~\ref{lem:locallinear} below, for all $\delta > 0$ there exists a $\gamma>0$ such that $g^\lambda(\theta)$ is linear on $S_\gamma(\hat \theta) = \{\theta:\; \|\theta - \hat\theta\| < \gamma\}$ with probability greater than $1-\delta$, where $\hat \theta \sim N(\theta_0, \Sigma)$.
\end{itemize}

We claim that the combination of these results implies that $B_{n}\to^{p}0$ as long as $\rho_{n} \to 0$.
To see this, define $\Theta_\gamma = \{\theta : g^\lambda \text{ is linear (affine) on } S_\gamma(\theta)\}$, for $\gamma >0$.
By Lemma~\ref{lem:locallinear}, $P(\widehat{\theta} \in \Theta_{\gamma}) > 1-\delta$ for $\gamma$ small enough. By Corollary~\ref{cor:asymptotic_penalizederm}, and the definition of convergence in distribution, we therefore get $P(\widehat{\theta}_{n} \in \Theta_{\gamma}) >1-\delta$ for $n$ large enough.

By Lemma~\ref{lem:ifrep}, $\lVert \widehat{\theta}_{n}^{\lambda} - \tilde{\theta}_{n}^{\lambda} \rVert <\gamma$ with probability greater than $1-\delta$ for $n$ large enough. This implies that $sign(\widehat{\theta}^{\lambda}_{n})_{J} = \eta_{J}$, because $|\tilde\theta_{n,j}^\lambda|$ is bounded away from $0$ for $j \in J$, by definition of $\Theta_\gamma$. This takes care of the active coordinates.

Let now $j \in J^{c}$ be one of the inactive coordinates. If $\tilde{\theta}_{n} \in \Theta_{\gamma}$, then necessarily $\lvert \tilde{\theta}_{j,n}\rvert + \gamma < \lambda$, by definition of $\Theta_{\gamma}$, since the mapping from $\tilde{\theta}_{j,n}$ to $\tilde{\theta}_{j,n}^\lambda$ has a kink at $\pm \lambda$.

Let $\theta,\theta'$ be equal in all coordinates except $j$, where $\theta_{j} = t$ and $\theta'_{j}=0$. Then, by Lemma~\ref{lem:ifrep}, by $\lvert\tilde{\theta}_{j,n}\rvert + \gamma < \lambda$, and by the Lipschitz continuity of $\epsilon(\theta)$ with constant $\gamma_{n}<\gamma$, for $n$ large enough, which follows again from Lemma~\ref{lem:ifrep}, 
$$
\begin{aligned}
&(L_{n}(\theta) + \lambda\lVert \theta \rVert_{1}) -
(L_{n}(\theta') + \lambda\lVert \theta' \rVert_{1}) \\
= &\tfrac{1}{2}(t-\tilde{\theta}_{j,n})^{2} + \epsilon(\theta) +\lambda \lvert t \rvert -\tfrac{1}{2}\tilde{\theta}_{j,n}^{2} - \epsilon(\theta')\\
\geq & \tfrac{1}{2} t^{2} + \lvert t \rvert \cdot \left(- (\lambda-\gamma) + \lambda -\gamma_{n}  \right) \\
= &\tfrac{1}{2} t^{2} + \lvert t \rvert \cdot \left(\gamma -\gamma_{n}  \right)\\
\geq &0,
\end{aligned}
$$
with equality only for $t=0$. It follows that the minimizer of $L_{n}(\theta) + \lambda\lVert \theta \rVert_{1}$ necessarily has $j$th component equal to zero. This takes care of the inactive coordinates.\\

\textbf{Third sum}
To control the third sum, we show that $C_{n} \to^{p} 0$. Since $\sum_{i=1}^{n}\lVert D_{n}^{i}\rVert^{2} = O_{p}(1)$ (which follows from $\lVert D_{n}^{i}\rVert \le \lVert \hat{\theta}_{n}^{\lambda,-i}-\hat{\theta}_{n}^{\lambda}\rVert + \lVert \tilde{\theta}_{n}^{\lambda,-i}-\tilde{\theta}_{n}^{\lambda}\rVert$ and the bounds (\ref{eq:requirement_1})--(\ref{eq:requirement_3}), by the same argument as for (\ref{eq:pf:bound_on_sum})), this implies $C_{n}\cdot\sum_{i}\lVert D_{n}^{i}\rVert^{2} = o_{p}(1)$.

Recall that $C_{n}=1$ requires: (a) $sign(\widehat{\theta}_{n}^{\lambda})=\eta$ and $g^{\lambda}$ is linear on $B_{\rho_{n}}(\tilde{\theta}_{n})$ (the negation of $B_n$'s condition), but (b) there exists some $i$ such that $sign(\widehat{\theta}_{n}^{\lambda,-i})\neq\eta$ or $sign(\tilde{\theta}_{n}^{\lambda,-i})\neq\eta$.
We show that on the event described in (a), neither type of sign disagreement occurs, with probability approaching~1.

\medskip
\textit{Preliminary: rate for the influence-function approximation error.}
We claim
\begin{align}
\sup_{\lambda}\lVert \widehat{\theta}_{n}^{\lambda}-\tilde{\theta}_{n}^{\lambda}\rVert = O_{\mu_{n}}(n^{-1/2}) = o_{\mu_n}(\rho_n), \label{eq:pf:if_rate}
\end{align}
where the second equality uses $\rho_{n}n^{1/2}=(\rho_{n}n^{1/4})\cdot n^{1/4}\to\infty$.
By Lemma~\ref{lem:ifrep}, $\widehat{\theta}_{n}^{\lambda}$ and $\tilde{\theta}_{n}^{\lambda}$ are respectively the minimizers of $\tfrac{1}{2}\lVert\theta-\tilde{\theta}_{n}\rVert^{2}+\epsilon_{n}(\theta)+\lambda\pi(\theta)$ and $\tfrac{1}{2}\lVert\theta-\tilde{\theta}_{n}\rVert^{2}+\lambda\pi(\theta)$.
By Lemma~1 in \cite{wilsonkasymackey2018} applied to the perturbation $\epsilon_{n}$,
\begin{equation}
\sup_{\lambda}\lVert \widehat{\theta}_{n}^{\lambda}-\tilde{\theta}_{n}^{\lambda}\rVert \;\le\; \frac{1}{\mu}\sup_{\lVert\theta\rVert\le C}\lVert\nabla\epsilon_{n}(\theta)\rVert. \label{eq:pf:wkm_eps}
\end{equation}
Since $\nabla L_{n}(\theta_{0})=\theta_{0}-\tilde{\theta}_{n}$ exactly (by the definition $\tilde{\theta}_{n}=\theta_{0}-\sum_{i}\nabla_{\theta}l_{n}(\theta_{0},Z_{n}^{i})$ in Lemma~\ref{lem:ifrep}) and $\nabla\epsilon_{n}(\theta)=\nabla L_{n}(\theta)-(\theta-\tilde{\theta}_{n})$, we have $\nabla\epsilon_{n}(\theta_{0})=0$ exactly. Therefore, for $\lVert\theta\rVert\le C$,
\[
\lVert\nabla\epsilon_{n}(\theta)\rVert \;=\; \lVert\nabla\epsilon_{n}(\theta)-\nabla\epsilon_{n}(\theta_{0})\rVert \;\le\; C\cdot\sup_{\lVert\theta'\rVert\le C}\lVert\nabla^{2}\epsilon_{n}(\theta')\rVert.
\]
Now $\nabla^{2}\epsilon_{n}(\theta)=\nabla^{2}L_{n}(\theta)-I$. The CLT applied to $n\nabla_{\theta}^{2}l_{n}(\theta_{0},Z_{n}^{i})$ (finite second moments from Assumption~\ref{as:loss_fn_bounds}(iii)) gives $\lVert\nabla^{2}L_{n}(\theta_{0})-I\rVert=O_{\mu_{n}}(n^{-1/2})$, using that $\lVert E_{\mu_{n}}[n\nabla_{\theta}^{2}l_{n}(\theta_{0},Z_{n}^{i})]-I\rVert=O(n^{-1/2})$ by Assumption~\ref{as:smoothloss} and the Lipschitz condition in Assumption~\ref{as:loss_fn_bounds}(iii). The same Lipschitz condition and the functional CLT extend this rate uniformly over bounded neighborhoods:
\begin{align}
\sup_{\lVert\theta\rVert\le C}\lVert\nabla\epsilon_{n}(\theta)\rVert \;=\; O_{\mu_{n}}(n^{-1/2}) \;=\; o_{\mu_n}(\rho_n). \label{eq:pf:grad_eps_rate}
\end{align}
Substituting into (\ref{eq:pf:wkm_eps}) establishes (\ref{eq:pf:if_rate}).

\medskip
\textit{Sign agreement for $\tilde{\theta}_{n}^{\lambda,-i}$.}
Since $\tilde{\theta}_{n}^{-i}-\tilde{\theta}_{n}=-\frac{1}{\sqrt{n}}X_{n}^{i}$ where $X_{n}^{i}=-\nabla_{\beta}l(\theta_{0}/\sqrt{n},Z_{n}^{i})$,
the event $\lVert \tilde{\theta}_{n}^{-i}-\tilde{\theta}_{n}\rVert >\rho_{n}$ for some $i$ has probability bounded by
$$
\sum_{i=1}^{n}P\!\left(\frac{1}{\sqrt{n}}\lVert X_{n}^{i}\rVert >\rho_{n}\right)
\le n\cdot\frac{E_{\mu_{n}}[\lVert X_{n}^{i}\rVert^{4}]}{(\rho_{n}\sqrt{n})^{4}}
=\frac{M}{(\rho_{n}n^{1/4})^{4}}\to 0,
$$
by the 4th-moment bound in Assumption~\ref{as:loss_fn_bounds}(ii) and $\rho_{n}\cdot n^{1/4}\to\infty$.
On the complement of this event, $\tilde{\theta}_{n}^{-i}\in B_{\rho_{n}}(\tilde{\theta}_{n})$ for every $i$, so  $g^{\lambda}$ is linear on a neighborhood of each $\tilde{\theta}_{n}^{-i}$, and therefore $sign(\tilde{\theta}_{n}^{\lambda,-i})=\eta$ for all~$i$.

\medskip
\textit{Sign agreement for $\widehat{\theta}_{n}^{\lambda,-i}$: active coordinates.}
On the event $C_{n}=1$, $g^{\lambda}$ is linear on $B_{\rho_{n}}(\tilde{\theta}_{n})$. For $j\in J$, this forces $\lvert\tilde{\theta}_{n,j}\rvert>\lambda+\rho_{n}$, hence $\lvert\tilde{\theta}_{n,j}^{\lambda}\rvert=\lvert\tilde{\theta}_{n,j}\rvert-\lambda>\rho_{n}$.
By (\ref{eq:pf:if_rate}), $\sup_{\lambda}\lVert\widehat{\theta}_{n}^{\lambda}-\tilde{\theta}_{n}^{\lambda}\rVert=o_{\mu_{n}}(\rho_{n})$, so with probability approaching~1,
\begin{equation}
\lvert\widehat{\theta}_{n,j}^{\lambda}\rvert \;\ge\; \rho_{n}/2 \qquad\text{for all }j\in J. \label{eq:pf:active_lb}
\end{equation}
It remains to show $\lvert\widehat{\theta}_{n,j}^{\lambda,-i}-\widehat{\theta}_{n,j}^{\lambda}\rvert<\rho_{n}/4$ for all $j\in J$ and all $i$ simultaneously.
By (\ref{eq:requirement_1}), $\sup_{\lambda}\lVert\widehat{\theta}_{n}^{\lambda,-i}-\widehat{\theta}_{n}^{\lambda}\rVert\le\frac{1}{\mu}\lVert\nabla_{\theta}l_{n}(\widehat{\theta}_{n},Z_{n}^{i})\rVert$.
The Lipschitz condition in Assumption~\ref{as:loss_fn_bounds}(ii) gives
\[
\lVert\nabla_{\theta}l_{n}(\widehat{\theta}_{n},Z_{n}^{i})\rVert
\;\le\;\frac{1}{\sqrt{n}}\lVert X_{n}^{i}\rVert + \frac{B_{n}(Z_{n}^{i})}{\sqrt{n}}\lVert\widehat{\theta}_{n}-\theta_{0}\rVert.
\]
A Chebyshev union bound on the first term (4th-moment bound in Assumption~\ref{as:loss_fn_bounds}(ii)) and a Chebyshev union bound on the second term (2nd-moment bound $E_{\mu_{n}}[B_{n}^{2}]<\infty$ in Assumption~\ref{as:loss_fn_bounds}(ii), together with $\widehat{\theta}_{n}-\theta_{0}=O_{\mu_{n}}(1)$) give
\begin{equation}
P\!\left(\exists\,i:\;\sup_{\lambda}\lVert\widehat{\theta}_{n}^{\lambda,-i}-\widehat{\theta}_{n}^{\lambda}\rVert>\frac{\rho_{n}}{4}\right)
\;\le\; \frac{C_{1}M}{(\rho_{n}n^{1/4})^{4}} + \frac{C_{2}E_{\mu_{n}}[B_{n}^{2}]}{n\rho_{n}^{2}}\to 0, \label{eq:pf:union_active}
\end{equation}
for absolute constants $C_{1},C_{2}$, using $(\rho_{n}n^{1/4})^{4}\to\infty$ and $n\rho_{n}^{2}\ge(\rho_{n}n^{1/4})^{2}\cdot n^{1/2}\to\infty$.
On the complement of (\ref{eq:pf:union_active}) and given (\ref{eq:pf:active_lb}), we have $\lvert\widehat{\theta}_{n,j}^{\lambda,-i}\rvert\ge\rho_{n}/4>0$ with sign~$\eta_{j}$, for all $j\in J$ and all~$i$.

\medskip
\textit{Sign agreement for $\widehat{\theta}_{n}^{\lambda,-i}$: inactive coordinates.}
It remains to show that $\widehat{\theta}_{n,j}^{\lambda,-i}=0$ for $j\in J^{c}$ and all~$i$, with probability approaching~1.
The argument parallels the one used for the second sum (showing $sign(\widehat{\theta}_{n}^{\lambda})=\eta$), now applied to $L_{n}^{-i}$ and $\widehat{\theta}_{n}^{\lambda,-i}$.

By Lemma~\ref{lem:ifrep},
$$
L_{n}^{-i}(\theta) = \tfrac{1}{2}\lVert \theta-\tilde{\theta}_{n}^{-i}\rVert^{2} + \epsilon_{n}^{(i)}(\theta) + c_n^{(i)},
$$
where
\begin{align*}
	\epsilon_{n}^{(i)}(\theta) &= \epsilon_{n}(\theta) - \left[l_{n}(\theta,Z_{n}^{i}) - l_{n}(\theta_{0},Z_{n}^{i}) - \langle\nabla_{\theta}l_{n}(\theta_{0},Z_{n}^{i}),\,\theta-\theta_{0}\rangle\right.\\
	& + \left.\tfrac{1}{2n}(\theta-\theta_{0})^{\intercal}\left\{n\nabla_{\theta}^{2}l_{n}(\theta_{0},Z_{n}^{i})\right\}(\theta-\theta_{0})\right].
\end{align*}
Consider a candidate minimizer $\theta$ of $L_{n}^{-i}(\theta)+\lambda\lVert\theta\rVert_{1}$ with $\theta_{j}=t\neq 0$ for some $j\in J^{c}$, and let $\theta^{\prime}$ agree with $\theta$ except $\theta_{j}^{\prime}=0$. On the event that $g^{\lambda}$ is linear on $B_{\rho_{n}}(\tilde{\theta}_{n})$, the KKT conditions imply $\lvert\tilde{\theta}_{n,j}\rvert +\rho_{n}<\lambda$ for $j\in J^{c}$. Therefore, by the same calculation as in the second sum,
$$
\left(L_{n}^{-i}(\theta)+\lambda\lvert t\rvert\right)-L_{n}^{-i}(\theta^{\prime})
\ge \tfrac{1}{2}t^{2}+\lvert t\rvert\!\left(\lambda-\lvert\tilde{\theta}_{n,j}^{-i}\rvert - \sup_{\lVert\theta\rVert\le C}\lVert\nabla\epsilon_{n}^{(i)}(\theta)\rVert\right).
$$
Since $\lvert\tilde{\theta}_{n,j}^{-i}\rvert\le\lvert\tilde{\theta}_{n,j}\rvert+n^{-1/2}\lVert X_{n}^{i}\rVert$ and the slack is $\lambda-\lvert\tilde{\theta}_{n,j}\rvert>\rho_{n}$ on the event under consideration, it suffices to show simultaneously for all~$i$:
\begin{equation}
n^{-1/2}\lVert X_{n}^{i}\rVert \;+\; \sup_{\lVert\theta\rVert\le C}\lVert\nabla\epsilon_{n}^{(i)}(\theta)\rVert \;<\; \rho_{n}. \label{eq:pf:slack_condition}
\end{equation}
Write $\nabla\epsilon_{n}^{(i)}(\theta)=\nabla\epsilon_{n}(\theta)-R_{3}(\theta,Z_{n}^{i})$, where
\[
R_{3}(\theta,Z_{n}^{i}):=\nabla_{\theta}l_{n}(\theta,Z_{n}^{i})-\nabla_{\theta}l_{n}(\theta_{0},Z_{n}^{i})-\tfrac{1}{n}\bigl\{n\nabla_{\theta}^{2}l_{n}(\theta_{0},Z_{n}^{i})\bigr\}(\theta-\theta_{0})
\]
is the second-order Taylor remainder of $\nabla_{\theta}l_{n}(\cdot,Z_{n}^{i})$ around $\theta_{0}$.
We bound the three contributions to (\ref{eq:pf:slack_condition}) in turn.

\smallskip\noindent
\textit{Control of $\nabla\epsilon_{n}$.}
By (\ref{eq:pf:grad_eps_rate}), $\sup_{\lVert\theta\rVert\le C}\lVert\nabla\epsilon_{n}(\theta)\rVert=O_{\mu_{n}}(n^{-1/2})=o_{\mu_{n}}(\rho_{n})$; this term is the same for all~$i$.

\smallskip\noindent
\textit{Simultaneous control of $R_{3}(\cdot,Z_{n}^{i})$.}
By the Lipschitz condition in Assumption~\ref{as:loss_fn_bounds}(iii),
$\lVert R_{3}(\theta,Z_{n}^{i})\rVert\le\frac{C_{n}(Z_{n}^{i})}{2n}\lVert\theta-\theta_{0}\rVert^{2}$,
so $\sup_{\lVert\theta\rVert\le C}\lVert R_{3}(\theta,Z_{n}^{i})\rVert\le\frac{C_{n}(Z_{n}^{i})C^{2}}{2n}$.
A Chebyshev union bound gives
\begin{equation}
P\!\left(\exists\,i:\;\sup_{\lVert\theta\rVert\le C}\lVert R_{3}(\theta,Z_{n}^{i})\rVert>\frac{\rho_{n}}{3}\right)
\;\le\;n\cdot P\!\left(C_{n}(Z_{n}^{i})>\frac{2n\rho_{n}}{3C^{2}}\right)
\;\le\;\frac{9C^{4}\sup_{n}E_{\mu_{n}}[C_{n}(Z_{n}^{i})^{2}]}{4\,n\rho_{n}^{2}}\to 0,
\label{eq:pf:union_R3}
\end{equation}
since $n\rho_{n}^{2}\ge(\rho_{n}n^{1/4})^{2}\cdot n^{1/2}\to\infty$ and $\sup_{n}E_{\mu_{n}}[C_{n}(Z_{n}^{i})^{2}]<\infty$ by Assumption~\ref{as:loss_fn_bounds}(iii).

\smallskip\noindent
\textit{Simultaneous control of $n^{-1/2}\lVert X_{n}^{i}\rVert$.}
By the same Chebyshev argument as for the first sign-agreement step,
\begin{equation}
P\!\left(\exists\, i:\; n^{-1/2}\lVert X_{n}^{i}\rVert>\frac{\rho_{n}}{3}\right)
\;\le\; \frac{81\,M}{(\rho_{n}n^{1/4})^{4}}\to 0. \label{eq:pf:union_score_inactive}
\end{equation}

On the complement of the three events above, (\ref{eq:pf:slack_condition}) holds for all~$i$: $n^{-1/2}\|X_n^i\| + \sup_{\|\theta\|\le C}\|\nabla\epsilon_n^{(i)}(\theta)\| \le \rho_n/3 + o_{\mu_n}(\rho_n) + \rho_n/3 < \rho_n$ for $n$ large enough. The slack is therefore strictly positive, so $\widehat{\theta}_{n,j}^{\lambda,-i}=0$ for all $j\in J^{c}$ and all~$i$.

\medskip
\textit{Conclusion.}
Combining the three cases above, we conclude that on the event described in condition~(a) (which has probability approaching~1 given $B_{n}\to^{p}0$), all signs agree: $sign(\widehat{\theta}_{n}^{\lambda,-i})=sign(\tilde{\theta}_{n}^{\lambda,-i})=\eta$ for every~$i$. This contradicts condition~(b), so $C_{n}\to^{p}0$. Since $\sum_{i}\lVert D_{n}^{i}\rVert^{2}=O_{p}(1)$, we obtain $C_{n}\cdot\sum_{i}\lVert D_{n}^{i}\rVert^{2}=o_{p}(1)$.

\medskip
Combining the bounds for all three sums, we conclude $DD_{n}=o_{\mu_{n}}(1)$, which establishes the claim deferred from Step~2.

\medskip
\textit{Conclusion of the proof.}
It remains to assemble the pieces. By Steps~1--3, uniformly over $\lambda\in\Lambda$,
\[
CV_{n}(\lambda)=\sum_{i=1}^{n}l_{n}(\hat{\theta}_{n}^{\lambda},Z_{n}^{i})+\textrm{Tr}\left[\nabla g^{\lambda}(\tilde{\theta}_{n})\cdot\Sigma\right]+c_{n}'+o_{\mu_{n}}(1),
\]
where $c_{n}'$ collects the $\lambda$-independent terms isolated in Step~3. For the
in-sample term, Lemma~\ref{lem:ifrep}---applied at $\theta=\hat{\theta}_{n}^{\lambda}$,
using $\sup_{\lambda}\|\hat{\theta}_{n}^{\lambda}-\tilde{\theta}_{n}^{\lambda}\|=o_{\mu_{n}}(1)$
with $\tilde{\theta}_{n}^{\lambda}=\tilde{\theta}_{n}+g^{\lambda}(\tilde{\theta}_{n})$,
and the Lipschitz property of $g^{\lambda}$ (Lemma~\ref{lem:lipschitz})---gives
\[
\sum_{i=1}^{n}l_{n}(\hat{\theta}_{n}^{\lambda},Z_{n}^{i})=L_{n}(\hat{\theta}_{n}^{\lambda})=c_{n}''+\tfrac12\|g^{\lambda}(\hat{\theta}_{n})\|^{2}+o_{\mu_{n}}(1),
\]
uniformly over $\lambda$, where $c_{n}''=L_{n}(\theta_{0})-\tfrac12\|\tilde{\theta}_{n}-\theta_{0}\|^{2}$
is $\lambda$-independent. Since $\hat{\Sigma}_{n}\to^{\mu_{n}}\Sigma$ and
$\nabla g^{\lambda}(\tilde{\theta}_{n})=\nabla g^{\lambda}(\hat{\theta}_{n})+o_{\mu_{n}}(1)$
(from $\hat{\theta}_{n}=\tilde{\theta}_{n}+o_{\mu_{n}}(1)$ and Lemma~\ref{lem:locallinear}),
collecting $c_{n}'$, $c_{n}''$, and the $\lambda$-independent $\tfrac12\,\textrm{trace}(\Sigma)$
into a single constant $c_{n}$ yields
\[
CV_{n}(\lambda)=c_{n}+\tfrac12\|g^{\lambda}(\hat{\theta}_{n})\|^{2}+\textrm{Tr}\left[\nabla g^{\lambda}(\hat{\theta}_{n})\cdot\Sigma\right]+o_{\mu_{n}}(1)=c_{n}+SURE(\lambda,\hat{\theta}_{n},\Sigma)+o_{\mu_{n}}(1),
\]
uniformly over $\lambda\in\Lambda$, where the last equality uses the definition
$SURE(\lambda,\hat{\theta},\Sigma)=\tfrac12[\textrm{trace}(\Sigma)+\|g^{\lambda}(\hat{\theta})\|^{2}+2\,\textrm{Tr}(\nabla g^{\lambda}(\hat{\theta})\cdot\Sigma)]$.
This completes the proof of Lemma~\ref{lem:convergence_cv}. \hfill$\qed$

\section{Characterizing SURE for Ridge}
\label{sec:ridge}

To prove Lemma~\ref{lem:convergencetuned} for the Ridge penalty $\tfrac12 \theta  \cdot A^{-1}  \cdot \theta$, we start by deriving a series of properties of the function $SURE(\lambda,\hat \theta, \Sigma) =  \tfrac12\left[\trace(\Sigma) + \|g^\lambda\|^2 + 2  \trace \left ( \nabla g^\lambda\cdot \Sigma\right )\right]$.
The properties derived in this section are purely analytic, not probabilistic, and concern the behavior of $SURE$ as a function, not any asymptotic limits.
Recall that $g^\lambda(\theta) = \argmin_g \tfrac12\|g\|^2 + \lambda  \cdot \pi(\theta + g)$,
and that $g^\lambda$ satisfies the first-order condition,
$$
  g^\lambda(\theta) = - \lambda  \cdot \nabla \pi(\theta + g^\lambda(\theta)),
$$
for a suitable sub-gradient $\nabla \pi$ of $\pi$.
Ridge corresponds to penalties of the form $\pi(\theta) = \tfrac12 \theta  \cdot A^{-1}  \cdot \theta$, where $A$ is positive definite.
Denote $C_\lambda = -(\tfrac1{\lambda} A + I)^{-1}$.
The first order condition for $g^\lambda(\theta)$ then implies
\bals
  g^\lambda(\theta) &= C_\lambda \cdot  \theta,&
  \nabla g^\lambda(\theta) &= C_\lambda.
\eals
and thus
\bals
SURE(\lambda, R, \nu) &= \tfrac12\left[\trace(\Sigma) +
  \|C_\lambda \cdot  \theta\|^2 +
2  \trace \left ( C_\lambda \cdot \Sigma\right )\right].
\eals

The following change of coordinates will be convenient for some of our arguments.
Denote $R_n = \|\hat \theta_n\|$ and $\nu_n = \hat \theta_n / R_n$, and similarly for $R$ and $\nu$.
In a slight abuse of notation, we shall write
$$
SURE(\lambda, R, \nu) = SURE(\lambda,\hat \theta, \Sigma) .
$$

The following lemma characterizes the behavior of $SURE$ for Ridge. Property 1 and supermodularity are derived directly from the expression for $SURE$. Properties 2a, 2b, and 3 are then consequences of supermodularity.
\begin{lemma}[Properties of $SURE$ for Ridge]$\;$
  \label{lem:sureridge}
  Suppose that $\pi(\theta) = \tfrac12 \theta  \cdot A^{-1}  \cdot \theta$, where $A$ is positive definite.
  Then the following holds:
  \begin{enumerate}
    \item For every point $\theta$, the function $SURE(\lambda, \theta)$ satisfies  
    $$
      \sup_{\lambda \in \Lambda} |SURE(\lambda, \theta') - SURE(\lambda, \theta)| \rightarrow 0
    $$
    as $\theta' \rightarrow \theta$.
    \item The function $SURE(\lambda, R, \nu)$ is strictly supermodular in $\lambda$ and $R$.
    This implies:
    \begin{enumerate}
      \item $\lambda(R, \nu) = \argmin_{\lambda \in \mathbb R^+} SURE(\lambda, R, \nu)$ is monotonically decreasing in $R$, given $\nu$.
      \item $\lambda(R, \nu)$ has at most countably many discontinuities, as a function of $R$, given $\nu$.
    \end{enumerate}
    \item Fix $\nu$ and $R$ such that $\lambda( \cdot )$ is continuous in $R$ at $(R,\nu)$, and let $\bar \lambda = \lambda(R, \nu)$. 
    Then supermodularity implies that the minimum of $SURE$ is well separated: For any $\epsilon>0$,
    $$
      \inf_{\lambda \in \mathbb R^+ \backslash [\bar \lambda - \epsilon,\bar \lambda + \epsilon]}  SURE(\lambda, R, \nu) -  SURE(\bar \lambda, R, \nu) >0.
    $$
  \end{enumerate}
\end{lemma}

\begin{proof}[Proof of Lemma~\ref{lem:sureridge} (Properties of $SURE$ for Ridge):]$\;$
  \begin{enumerate}
    \item We have
    \bals
      &2\left[SURE(\lambda, \theta') - SURE(\lambda, \theta)\right]\\
       &=  \|C_\lambda \cdot  \theta'\|^2 -  \|C_\lambda \cdot  \theta\|^2\\
       &\leq \|C_\lambda\|^2  \cdot \|\theta'-\theta\| \cdot \|\theta'+\theta\|\\
       &\leq \|\theta'-\theta\| \cdot \|\theta'+\theta\|,
    \eals
    where in the third line we have used Cauchy-Schwartz ($\lVert \theta' \rVert^{2} - \lVert \theta \rVert^{2} = \langle \theta'-\theta, \theta'+\theta\rangle\leq \lVert \theta'-\theta \rVert \cdot \lVert \theta'+\theta \rVert$), and in the last line we have used that $A$ is positive definite, which implies $\|C_\lambda\| \leq 1$.
    The claim follows.

    \item The first and last terms in the expression for $SURE$ do not depend on $\theta$.
    The middle term can be written as $\tfrac12 R^2 \cdot \|C_\lambda \cdot  \nu\|^2$,
    and thus
    \bals
    \frac{\partial^2}{\partial \lambda \partial R} SURE(\lambda, R, \nu) &= R  \cdot \frac{\partial}{\partial \lambda} \|C_\lambda \cdot  \nu\|^2 >0,
    \eals
    using again the positive definiteness of $A$.
    This implies that $SURE(\lambda, R, \nu)$ is strictly supermodular in $\lambda$ and $R$, i.e., whenever $\epsilon>0$ and $\delta >0$
    \bals
      &\left[SURE(\lambda + \epsilon, R,\nu) -SURE(\lambda, R,\nu)\right]\\
      &\hspace{20pt} - \left[SURE(\lambda + \epsilon, R- \delta, \nu) -SURE(\lambda, R- \delta ,\nu)\right] >0.
    \eals
    \begin{enumerate}
      \item Monotonicity of $\lambda(R, \nu)$ in $R$ follows from supermodularity of $SURE$, by Topkis's theorem.
      
      \item That $\lambda(R, \nu)$ is continuous in $R$ almost everywhere holds because monotonic functions are continuous almost everywhere: The set of discontinuity points is at most countable, by Theorem 4.30 of \cite{rudin1964principles}.
    \end{enumerate}

    \item For the given $R, \nu$, let $\delta$ be such that $|\lambda(R',\nu) - \lambda(R,\nu)| < \epsilon/2$ whenever $|R' - R| \leq \delta$; such a $\delta$ exists by continuity.
    
    Define
    $$A(\lambda_1, \lambda_2, R) = SURE(\lambda_1, R,\nu) -SURE(\lambda_2, R,\nu),$$ 
    let $\bar \lambda = \lambda(R, \nu)$
    and
    \bals
      \overline{\Delta} =  \min \Big(
        A(\bar \lambda + \epsilon, \bar \lambda + \epsilon/2, R) - 
          A(\bar \lambda + \epsilon, \bar \lambda + \epsilon/2, R - \delta)&,\\
        A(\bar \lambda - \epsilon/2, \bar \lambda - \epsilon, R + \delta) - 
          A(\bar \lambda - \epsilon/2, \bar \lambda - \epsilon, R)
        \Big).
    \eals
    By {strict} supermodularity, both of the ``double differences'' in this definition are positive, and thus $\overline{\Delta}>0$.
    The following figure illustrates the definition of $\overline{\Delta}$.
    The differences defining $A$ are taken over vertical segments for different $\lambda$ and fixed $R$. The double differences defining $\Delta$ are taken over of the grey rectangles in the figure:\\
\begin{center}
    \begin{tikzpicture}[scale=1.2] 
      \draw[->] (-1,-1) -- (4,-1) node[right] {$R$};
      \draw[->] (0,-2) -- (0,5) node[above] {$\lambda$};
      
      \draw[dashed, thin, gray] (1,-1) -- (1,5);
      \draw[dashed, thin, gray] (2,-1) -- (2,5);
      \draw[dashed, thin, gray] (3,-1) -- (3,5);
      \draw[dashed, thin, gray] (0,0) -- (4,0);
      \draw[dashed, thin, gray] (0,1) -- (4,1);
      \draw[dashed, thin, gray] (0,2) -- (4,2);
      \draw[dashed, thin, gray] (0,3) -- (4,3);
      \draw[dashed, thin, gray] (0,4) -- (4,4);
      
      \coordinate (R_minus_delta_lambda_plus_eps) at (1,4);
      \coordinate (R_lambda_plus_eps) at (2,4);
      \coordinate (R_minus_delta_lambda_plus_eps_half) at (1,3);
      \coordinate (R_lambda_plus_eps_half) at (2,3);
      \coordinate (R_lambda_minus_eps_half) at (2,1);
      \coordinate (R_plus_delta_lambda_minus_eps_half) at (3,1);
      \coordinate (R_lambda_minus_eps) at (2,0);
      \coordinate (R_plus_delta_lambda_minus_eps) at (3,0);
      
      \fill (R_minus_delta_lambda_plus_eps) circle (2pt);
      \fill (R_lambda_plus_eps) circle (2pt);
      \fill (R_minus_delta_lambda_plus_eps_half) circle (2pt);
      \fill (R_lambda_plus_eps_half) circle (2pt);
      \fill (R_lambda_minus_eps_half) circle (2pt);
      \fill (R_plus_delta_lambda_minus_eps_half) circle (2pt);
      \fill (R_lambda_minus_eps) circle (2pt);
      \fill (R_plus_delta_lambda_minus_eps) circle (2pt);
      
      \node at (2,-1.3) {$R$};
      \node at (1,-1.3) {$R - \delta$};
      \node at (3,-1.3) {$R + \delta$};
      \node at (-0.6,2) {$\bar \lambda$};
      \node at (-0.6,4) {$\bar \lambda + \epsilon$};
      \node at (-0.6,3) {$\bar \lambda + \epsilon/2$};
      \node at (-0.6,1) {$\bar \lambda - \epsilon/2$};
      \node at (-0.6,0) {$\bar \lambda - \epsilon$};
      
      \fill[gray,opacity=0.3] (1,4) rectangle (2,3);
      \fill[gray,opacity=0.3] (2,1) rectangle (3,0);
  \end{tikzpicture}\\
\end{center} 
    
    We claim that 
    $$
    \inf_{\lambda \in \mathbb R^+ \backslash [\bar \lambda - \epsilon,\bar \lambda + \epsilon]}  SURE(\lambda, R, \nu) -  SURE(\bar \lambda, R, \nu) \geq \overline{\Delta}.
    $$
    The following argument will correspond to the ``top left'' rectangle in the figure; the ``bottom right'' case is symmetric.
    Fix $\tilde{\lambda} > \bar \lambda + \epsilon$.

    \begin{itemize}
      \item Given our assumptions and definitions, $\tilde{\lambda} > \bar \lambda + \epsilon > \bar \lambda + \epsilon/2 > \lambda(R - \delta) \geq \bar \lambda$. Therefore, by supermodularity
      \bals
      A(\tilde{\lambda}, \lambda(R - \delta), R) - 
      A(\tilde{\lambda}, \lambda(R - \delta), R- \delta)\\
       \geq A(\bar \lambda + \epsilon, \bar \lambda + \epsilon/2, R)- 
      A(\bar \lambda + \epsilon, \bar \lambda + \epsilon/2, R- \delta).
      \eals
    
      \item By definition of $\overline \Delta$,
      $$
       A(\bar \lambda + \epsilon, \bar \lambda + \epsilon/2, R) - 
       A(\bar \lambda + \epsilon, \bar \lambda + \epsilon/2, R- \delta) \geq 
       \overline \Delta.
      $$
    
      \item By optimality of $\lambda(R - \delta)$ for $R- \delta$,
      $$
      A(\tilde{\lambda}, \lambda(R - \delta), R- \delta) \geq 0.
      $$
    
      \item Combining the preceding three items, we get
      $$
      A(\tilde{\lambda}, \lambda(R - \delta), R) \geq \overline \Delta.
      $$
      and thus
      $$
        SURE(\tilde{\lambda}, R, \nu) \geq SURE(\lambda(R - \delta), R, \nu) + \overline \Delta \geq SURE(\bar \lambda, R, \nu) + \overline \Delta.
      $$
      The argument for $\tilde{\lambda} < \bar \lambda - \epsilon$ is analogous, and the claim follows.
    \end{itemize}
  \end{enumerate}   
\end{proof}

\section{Characterizing SURE for Lasso}
\label{sec:lasso}

Lasso corresponds to penalties of the form 
$\pi(\theta) = \|A^{-1}  \cdot \theta\|_1$, where $A$ is an invertible matrix, and $\| \cdot \|_1$ is the $L_1$ norm.\footnote{For Lasso, it is \textit{not} without loss of generality to assume that $A$ is diagonal.}
Given $\lambda$, we can characterize $g^\lambda(\theta)$ as follows.
Denote $h^\lambda(\theta) = A^{-1}(\theta + g^\lambda(\theta))$.
The optimal $h^\lambda(\theta)$ solves
$$
  h^\lambda(\theta) = \argmin_h \tfrac12 \| A  \cdot h - \theta\|^2 + \lambda  \cdot \|h\|_1.
$$
The solution to this convex optimization problem is of the form
$$
  h^\lambda_J(\theta) = (A_J'A_J)^{-1} \cdot [A_J' \theta - \lambda \eta_J] .
$$
where $\eta_j = sign(h^\lambda_j)$, $J = \{j: \eta_j \neq 0\}$, and $A_J$ is the subset of columns corresponding to the index set $J$.
These conditions follow immediately from the first order conditions for $h^\lambda(\theta)$.

\begin{lemma}[Properties of $SURE$ for Lasso]
  \label{lem:surelasso}
Suppose that $\pi(\theta) = \|A^{-1}  \cdot \theta\|_1$, where $A$ is an invertible matrix, and $\| \cdot \|_1$ is the $L_1$ norm. Let $k= \dim(\theta)$.
Then the following holds:
\begin{enumerate}
  \item As a function of $\lambda$, for every $R,\nu$, the graph of $SURE(\lambda, R, \nu)$ consists of at most $3^k$ continuous segments indexed by $\eta \in \{-1, 0, 1\}^{k}$.
  On each of these segments $\eta$ and $J = \{j: \eta_j \neq 0\}$ are constant,
  \bals
  \nabla g^\lambda &= A_J \cdot (A_J'A_J)^{-1} \cdot A_J' - I,\\
  \|g^\lambda\|^2 &= \|\nabla g^\lambda \cdot \theta\|^2 + \lambda^2  \cdot \eta_J' (A_J'A_J)^{-1} \eta_J,
  \eals
  and
  $SURE(\lambda, R, \nu)$ is a monotonically increasing quadratic polynomial in $\lambda$ of the form
  $$
    SURE(\lambda, R, \nu) = const. + \tfrac12 \lambda^2  \cdot  \eta_J' (A_J'A_J)^{-1} \eta_J.
  $$
  \item $SURE$ for Lasso scales with $R$ as follows:
  $$
  SURE(R \cdot \lambda, R, \nu) =\tfrac12\left[\trace(\Sigma) + R^2  \cdot \|g^{\lambda}(\nu)\|^2 + 2  \trace \left ( \nabla g^{\lambda}(\nu)\cdot \Sigma\right )\right].
  $$
  Given $\nu$, let $\lambda_1,\lambda_2, \ldots, \lambda_m$ ($m\leq 3^k$) be the local minimizers of $SURE(\lambda, 1, \nu)$, corresponding to values of $\lambda$ where $\eta$ changes.
  The local minimizers of $SURE(\lambda, R, \nu)$ are then given by $R \cdot \lambda_1,R \cdot \lambda_2, \ldots, R \cdot \lambda_m$.

  \item Let  $\lambda(R, \nu) = \argmin_{\lambda \in \mathbb R^+} SURE(\lambda, R, \nu)$.
  Then, given $\nu$, $\lambda(R, \nu) = R  \cdot \lambda_{j(R)}$, where $j(R)$ is a monotonically decreasing mapping from $R$ to $1,2,\ldots, m$, and $\lambda_1,\lambda_2, \ldots, \lambda_m$ are as before.
  The graph $\lambda^*(R, \nu)$ thus follows a piecewise linear ``sawtooth'' pattern with at most $3^k$ jumps.

  \item Fix $\nu$ and $R$ such that $\lambda( \cdot )$ is continuous in $R$ at $(R,\nu)$, and let $\bar \lambda = \lambda(R, \nu)$ be such that $\eta \neq 0$. 
  Then the minimum of $SURE$ is well separated: For any $\epsilon>0$,
  $$
    \inf_{\lambda \in \mathbb R^+ \backslash [\bar \lambda - \epsilon,\bar \lambda + \epsilon]}  SURE(\lambda, R, \nu) -  SURE(\bar \lambda, R, \nu) >0.
  $$
\end{enumerate}
\end{lemma}
\begin{proof}[Proof of Lemma~\ref{lem:surelasso} (Properties of $SURE$ for Lasso):]$\;$
\begin{enumerate}
  \item Consider two values $\lambda_1,\lambda_2$ of $\lambda$ such that $\eta$ is the same for these two values.
  It follows from the optimality conditions for $h^{\lambda}$ that for any intermediate value of $\lambda$ between $\lambda_1,\lambda_2$, the optimal $h^\lambda$ is a linear interpolation between $h^{\lambda_1}$ and $h^{\lambda_2}$, and in particular $\eta$ remains the same in between $\lambda_1,\lambda_2$.
  For details, see \cite{mairal2012complexity}, Lemma 2.

  The vector $\eta \in \{-1,0,1\}^k$ can take $3^k$ possible values. 
  The gradient of $g^\lambda$ with respect to $\theta$ is a function of $J$, which is a function of $\eta$, but it does not depend on $\lambda$ otherwise:
  $$
  \nabla g^\lambda = A \cdot \nabla h^\lambda - I 
  = A_J \cdot (A_J'A_J)^{-1} \cdot A_J' - I.
  $$

  It follows that the penalty term $2 \trace \left ( \nabla g^\lambda\cdot \Sigma\right )$ in the expression for $SURE$ has at most $3^k-1$ jumps, as a function of $\lambda$ for fixed $\theta$, and is constant in between these jumps.\footnote{This bound can be refined, cf. \cite{mairal2012complexity}, but it is enough for our purposes.}

  Consider now the term $\|g^\lambda\|^2$, which is the other term in the expression for $SURE$.
  Fixing $\lambda$, and the corresponding set $J$ of active coordinates, we get that 
  $\|g^\lambda\|^2 = \|A_J h_J^\lambda - \theta\|^2$, which is minimized at $\lambda = 0$, holding $J$ fixed. 
  For $\lambda = 0$ we get 
  $$ \|A_J h_J^0 - \theta\|^2 = \|A_J  ((A_J'A_J)^{-1} A_J' - I) \cdot \theta\|^2 =  \|\nabla g^\lambda \cdot \theta\|^2.$$
  This is the sum of squared errors for an OLS regression of the elements of $\theta$ on the rows of $A_J$.

  Given $J$, $\|g^\lambda\|^2$ is a quadratic function of $\lambda$ with second derivative
  $
  \partial^2_\lambda \|A_J h_J^\lambda - \theta\|^2 = 2  \cdot \eta_J' (A_J'A_J)^{-1} \eta_J.
  $
  It follows that
  $$
  \|g^\lambda\|^2 = \|\nabla g^\lambda \cdot \theta\|^2 + \lambda^2  \cdot \eta_J' (A_J'A_J)^{-1} \eta_J.
  $$
  This last result implies that $\|g^\lambda\|^2$ is monotonically increasing in $\lambda$ on each segment defined by $\eta$. Since $g^\lambda$ is continuous in $\lambda$ (cf. Lemma 2 in \citealt{mairal2012complexity}), this also implies that $\|g^\lambda\|^2$  is monotonically increasing across $\lambda \in \mathbb R^+$; we will use this fact below.

  \item Multiplying the objective of the optimization problem $h^\lambda(\theta) = \argmin_h \tfrac12 \| A  \cdot h - \theta\|^2 + \lambda  \cdot \|h\|_1$ by a factor $1/R^2$ yields
  $$
  h^\lambda(R \cdot \nu) = \argmin_h \tfrac12 \| A  \cdot (h/R) - \nu\|^2 + \lambda/R  \cdot \|h/R\|_1 = R \cdot h^{\lambda/R}(\nu),
  $$
  and thus also 
  \bals
    g^\lambda(R \cdot \nu) &= R \cdot g^{\lambda/R}(\nu) \textrm{, and}\\
    \nabla g^\lambda(R \cdot \nu) &= \tfrac1R  \cdot \partial_\nu g^\lambda(R \cdot \nu) = \nabla g^{\lambda/R}(\nu).
  \eals
  which immediately implies
  $$
  SURE(R \cdot \lambda, R, \nu) =\tfrac12\left[\trace(\Sigma) + R^2  \cdot \|g^{\lambda}(\nu)\|^2 + 2  \trace \left ( \nabla g^{\lambda}(\nu)\cdot \Sigma\right )\right].
  $$

  Turning to the characterization of local minima, since $\|g^{\lambda}(\nu)\|^2$ is monotonically increasing in $\lambda$ (cf. item 1), the local minima of $SURE(R \cdot \lambda, R, \nu)$ are exactly the values of $\lambda$ where $2  \trace \left ( \nabla g^{\lambda}(\nu)\cdot \Sigma\right )$ jumps down. These values are independent of $R$, and the claim follows.

  \item That $\lambda(R, \nu) = R  \cdot \lambda_{j(R)}$ follows immediately from the preceding item.
  It remains to show that $j(R)$ is monotonically decreasing.
  To see this, consider any pair of values $j>j'$.
  Then $\|g^{\lambda_j}(\nu)\|^2 > \|g^{\lambda_{j'}}(\nu)\|^2$ by monotonicity of $\|g^{\lambda}(\nu)\|^2$ in $\lambda$ (cf. item 1), and we get that
  $$SURE(R \cdot \lambda_j, R, \nu) - SURE(R \cdot \lambda_{j'}, R, \nu) = \tfrac12 R^2  \cdot\left( \|g^{\lambda_j}(\nu)\|^2 - \|g^{\lambda_{j'}}(\nu)\|^2\right)$$
  is increasing in $R$. The claim follows.

  \item By the preceding argument, at a point of continuity in $R$
  $$
  \inf_{\lambda \in \{R \cdot \lambda_1,R \cdot \lambda_2, \ldots, R \cdot \lambda_m\} \backslash \bar \lambda}  SURE(\lambda, R, \nu) -  SURE(\bar \lambda, R, \nu) >0.
  $$
  The same holds for $\lambda$ to the right of any of the local minimizers $\{R \cdot \lambda_1,R \cdot \lambda_2, \ldots, R \cdot \lambda_{m}\} \backslash \bar \lambda$, since $SURE$ is monotonically increasing in $\lambda$ away from the local minimizers.

  It only remains to verify the condition for $\lambda$ immediately to the right of the global minimizer $\bar \lambda$.
  This holds because   $\|g^\lambda\|^2 = const. + \lambda^2  \cdot \eta_J' (A_J'A_J)^{-1} \eta_J$ is strictly monotonically increasing in $\lambda$ for $\lambda > 0$ and $\eta \neq 0$.
\end{enumerate}
\end{proof}

\begin{lemma}[Local linearity of $g^\lambda$]$\;$
  \label{lem:locallinear}
  \begin{enumerate}
    \item For all $\gamma > 0$  there exists an $\epsilon>0$ such that $g^\lambda(\theta)$ is linear on $B_\epsilon(\hat \theta) = \{\theta:\; \|\theta - \hat\theta\| < \epsilon\}$ with probability greater than $1-\gamma$, where $\hat \theta \sim N(\theta_0, \Sigma)$.
    \item For almost every point $\theta$, the function $SURE(\lambda, \theta)$ satisfies  
    $$
      \sup_{\lambda \in \Lambda} |SURE(\lambda, \theta') - SURE(\lambda, \theta)| \rightarrow 0
    $$
    as $\theta' \rightarrow \theta$.
  \end{enumerate}
\end{lemma}
\begin{proof}[Proof of Lemma~\ref{lem:locallinear}:]
To show the first claim, denote
$$
\Theta_\eta = \{\theta:\; sign(h^\lambda(\theta)) = \eta\}.
$$
By the KKT conditions for $h^\lambda$, the set $\Theta_\eta$ is convex for every $\eta \in \{-1,0,1\}^k$, and its boundary $\partial \Theta_\eta$ is a finite union of subsets of hyperplanes. Furthermore 
$\mathbb R^k = \bigcup_\eta \Theta_\eta $, and $h^\lambda(\theta)$ is linear in $\theta$ on each of the sets $\Theta_\eta$.

The claim of the lemma therefore follows if we can show that the probability of an $\epsilon$-band around the boundary $\partial \Theta_\eta$,
$$
\partial \Theta_\eta^\epsilon = \{\theta:\; d(\theta, \partial \Theta_\eta) < \epsilon\},
$$
where $d$ is the Euclidean distance, has vanishing measure for small $\epsilon$.
Because $\partial \Theta_\eta$ is a subset of a finite union of hyperplanes, $P(\hat \theta \in \partial \Theta_\eta) = 0$.
By the properties of probability measures, since
$$
\partial \Theta_\eta = \bigcap_{\epsilon>0} \partial \Theta_\eta^\epsilon
$$
we get
$$
0 = P(\hat \theta \in \partial \Theta_\eta) = \lim_{\epsilon \rightarrow 0} P(\hat \theta \in \partial \Theta_\eta^\epsilon).
$$
Therefore, for $\epsilon$ small enough, $\hat \theta$ is more than $\epsilon$ away from the boundary of any of the sets $\Theta_\eta$ with probability bigger than $1-\gamma$, and the claim follows.\\

We now turn to the second claim.
Fix $\lambda \in \Lambda$.
By the preceding argument, for almost all points $\theta$, $\theta$ is in the interior of $\Theta_\eta$, for some $\eta$. 
Thus $\theta' \in \Theta_\eta$, as well, for $\|\theta' - \theta\|$ small enough.
By the characterization of $SURE$ for Lasso in item 1 of Lemma~\ref{lem:surelasso}, given $\lambda$ and $\eta$, $\nabla g^\lambda$ is constant and $SURE = const. + \tfrac12\|g^\lambda\|^2 = const. + \tfrac12\|\nabla g^\lambda  \cdot \theta\|^2$. This is continuous in $\theta$, and thus
$|SURE(\lambda, \theta') - SURE(\lambda, \theta)| \rightarrow 0$ as $\theta' \rightarrow \theta$.
Since almost everywhere continuity of $SURE$ in $\theta$ thus holds for fixed $\lambda$, it also holds simultaneously for any finite set of $\lambda$, and the claim follows.
\end{proof}

\section{Convergence of risk}
\label{sec:risk}

The remainder of our proof will draw on some standard results, which we recall here, including the following  results from \cite{van2000asymptotic}:
\begin{enumerate}
  \item \textbf{Joint convergence (item (v) of Theorem 2.6)}\\
  If $W^1_n \rightarrow^d W^1$, and $W^2_n \rightarrow^p 0$, then 
  $W_n \rightarrow^d W$,\\
  where $W_n =(W^1_n, W^2_n)$,  $W= (W^1, 0)$.

  \item \textbf{Almost everywhere CMT (item (ii) of Theorem 2.3):}
  Suppose that $W_n\rightarrow^d W$ (converges in distribution), and that $s(W)$ is almost everywhere continuous.
  Then $s(W_n)\rightarrow^d s(W)$.
  
  \item \textbf{Uniform integrability and convergence of expectations (Theorem 2.20):}
  Suppose that $s(W_n)\rightarrow^d s(W)$, and that
  \be
    \lim_{M\rightarrow \infty} \limsup_{n\rightarrow \infty} E\left[|s(W_n)| \bs 1(|s(W_n)| >M)\right] = 0.
    \label{eq:unif_integr}
  \ee
  Then $E[s(W_n)]\rightarrow E[s(W)]$.

\end{enumerate}

\subsection{Convergence in distribution}

Consider some arbitrary function $c(\lambda)$ that is minimized to choose the tuning parameter $\lambda$ (in due time, we will substitute $CV_n$ for this function).
Define
\bals
\Delta(\lambda) &= c(\lambda) - SURE(\lambda, \theta, \Sigma).
\eals
We think of $\Delta$ as an element of the space of bounded functions on $\Lambda \subset \mathbb R^+$, endowed with the $\sup$ norm.
Define furthermore
\bals
w &= \left(\theta, \Delta\right),&
\|w\| &= \|\theta\| + \sup_{\lambda \in \Lambda} |\Delta(\lambda)|
\eals

We have proven for Ridge (in  Lemma~\ref{lem:sureridge}), and for Lasso (in  Lemma~\ref{lem:surelasso}), that the minimum of $SURE$ with respect to $\lambda$ is well separated for almost all $\theta$ (with the exception of Lasso when $\lambda$ is so large that $\hat \theta^\lambda = 0$).
This implies the following lemma.
The ``$\min$'' in the definition of $ \tilde{\lambda}$ serves as a tie-breaking rule in the case of non-uniqueness of the minimizer.

\begin{lemma}
  \label{lem:ascontinuous}
  For almost every $\theta$, the 
  mapping from $w = (\theta, \Delta)$ to $ g^{\tilde{\lambda}(\theta, \Delta)}$, where
  \bals
    \tilde{\lambda}(\theta, \Delta) &= \min\left(\argmin_{\lambda \in \Lambda} \left[SURE(\lambda, \theta, \Sigma) + \Delta(\lambda)\right]\right),
  \eals
  is continuous at $w=(\theta, 0)$ with respect to the norm $\|w\|$.
\end{lemma}
\begin{proof}
  We first prove the claim for Ridge, before discussing the necessary modifications for the argument to apply to Lasso.
  Fix $\nu$. By item 2a of Lemma~\ref{lem:sureridge}, for almost every $R$, $\lambda(R, \nu)$ is continuous in $R$.
  Fix such an $R$, and let $\theta = R \nu$.
  Fix $\epsilon >0$ and let 
  $$
  \gamma = \inf_{\lambda \in \mathbb R^+ \backslash [\bar \lambda - \epsilon,\bar \lambda + \epsilon]}  SURE(\lambda, R, \nu) -  SURE(\bar \lambda, R, \nu) >0,
  $$
  where $\bar \lambda = \tilde{\lambda}(\theta, 0) \in \argmin_\lambda  SURE(\lambda, R, \nu)$.
  The following figure illustrates the definition of $\gamma$:
  \begin{center}
  \begin{tikzpicture}[scale=1]
  
      \draw[->] (-.5,1) -- (5,1) node[right] {$\lambda$};
      \draw[->] (0,0.5) -- (0,4) node[left] {SURE};
      
      \coordinate (lambda_star) at (2.5,2);
      \coordinate (lambda_star_minus_epsilon) at (1.5,2);
      \coordinate (lambda_star_plus_epsilon) at (3.5,2);
      \coordinate (gamma) at (2.5,1);
      
      \draw[thick] plot[domain=1:4, smooth] (\x, {(\x-2.5)^2 + 1.5});

      \draw[dashed, thin, gray] (1.5,1) -- (1.5,4);
      \draw[dashed, thin, gray] (3.5,1) -- (3.5,4);
      \draw[dashed, thin, gray] (2.5,1) -- (2.5,4);
      
      \draw[dashed, thin, gray] (0,1.5) -- (5,1.5);
      \draw[dashed, thin, gray] (0,2.5) -- (5,2.5);
      
      \draw [decorate,decoration={brace,amplitude=10pt,mirror,raise=4pt},yshift=0pt] (5,1.5) -- (5,2.5) node [black,midway,xshift=0.6cm] {$\gamma$};

      \node at (2.5,.7) {$\bar \lambda$};
      \node at (1.5,.7) {$\bar \lambda - \epsilon$};
      \node at (3.5,.7) {$\bar \lambda + \epsilon$};
      \node at (5.5,3.2) {$SURE(\lambda, R, \nu)$};
      
  \end{tikzpicture}
\end{center}
  By item 3 of Lemma~\ref{lem:sureridge}, $\gamma > 0$.
  By item 1 of Lemma~\ref{lem:sureridge}, there exists a $\delta$ such that if $\|\theta' - \theta\|<\delta$ then $ \sup_{\lambda \in \mathbb R^+} |SURE(\lambda, \theta') - SURE(\lambda, \theta)| < \gamma/4$.\\

  Let $w' = (\theta', \Delta)$ and $w = (\theta, 0)$.
  Suppose that $\|w'- w\| < \min(\delta, \gamma/4)$, so that  $\|\theta' - \theta\|<\delta$ and  $\sup_\lambda |\Delta(\lambda)| < \gamma/4$.
  Denote $c(\lambda) = SURE(\lambda, \theta') + \Delta(\lambda)$, so that $\tilde{\lambda}(\theta', \Delta)$ is a minimizer of $c(\lambda)$.
  Then
\bals
   SURE(\tilde{\lambda}(\theta', \Delta) , \theta)\\
    < SURE(\tilde{\lambda}(\theta', \Delta) , \theta') + \tfrac14 \gamma\\
    < c(\tilde{\lambda}(\theta', \Delta)) + \tfrac 12 \gamma\\
    \leq c(\bar \lambda) + \tfrac12 \gamma \\
    < SURE(\bar \lambda , \theta') + \tfrac34 \gamma \\
    < SURE(\bar \lambda, \theta) + \gamma.
\eals
It follows that $|\tilde{\lambda}(\theta', \Delta) - \bar \lambda| < \epsilon$.
This proves that $\tilde{\lambda}(\theta, \Delta)$ is continuous at $(\theta, 0)$
The claim for Ridge follows, since continuity of $g^\lambda(\theta) = (\tfrac1{\lambda} A + I)^{-1} \cdot \theta$ in both $\lambda$ and $\theta$ is immediate.\\

Turning to Lasso, most of this argument holds verbatim, with the following modifications:
Consider first values of $\theta$ such that $\hat \theta^{\tilde \lambda} \neq 0$.
\begin{enumerate}
  \item By item 4 of Lemma~\ref{lem:surelasso}, $\gamma > 0$ for almost all such $\theta$.
  
  \item By item 2 of Lemma~\ref{lem:locallinear}, for \textit{almost} all $\theta$ there exists a $\delta$ such that if $\|\theta' - \theta\|<\delta$ then $ \sup_{\lambda \in \Lambda} |SURE(\lambda, \theta') - SURE(\lambda, \theta)| < \gamma/4$.
  
  \item By the characterization of $g^\lambda$ and $h^\lambda$ given at the  outset of Appendix~\ref{sec:lasso}, $g^\lambda(\theta)$ is continuous in both $\lambda$ and $\theta$. The claim thus follows for $\theta$ such that $\hat \theta^{\tilde \lambda} \neq 0$.
\end{enumerate}

Consider now values of $\theta$ such that $\hat \theta^{\tilde \lambda} = 0$.
For such values, $SURE$ is flat in $\lambda$ for values of $\lambda$ greater than $\tilde \lambda$, because $\hat \theta^{\lambda} = 0$ and $\nabla g^{\lambda} = 0$ for all such $\lambda$ (see Figure \ref{fig:multimodality} for an example). 
Because $SURE$ is flat to the right, continuity of $\tilde \lambda(\theta,\Delta)$ does not necessarily hold at $(\theta, 0)$; small perturbations of $\Delta$ can lead to large changes of $\tilde \lambda$.

By the same arguments used to prove item 4 of Lemma~\ref{lem:surelasso}, we obtain however (for almost all such $\theta$) that
$$
\inf_{\lambda < \lambda_m}  SURE(\lambda, R, \nu) -  SURE(\bar \lambda, R, \nu) >0,
$$
where $\lambda_m$ is defined as in Lemma~\ref{lem:surelasso}.
This, in combination with item 2 of Lemma~\ref{lem:locallinear} (for almost all $\theta$, $\sup_{\lambda \in \Lambda} |SURE(\lambda, \theta') - SURE(\lambda, \theta)| \rightarrow 0$ as $\theta' \rightarrow \theta$), implies that $\tilde \lambda(\theta', \Delta)$ is such that 
$g^{\tilde \lambda(\theta,0)}(\theta') = -\theta'$ for all $(\theta', \Delta)$ in a neighborhood of $(\theta,0)$, and the claim follows.
\end{proof}

\subsection{Proof of convergence in distribution}
We can now prove Lemma \ref{lem:convergencetuned}, drawing on our preceding Lemmas.

\begin{proof}[Proof of Lemma \ref{lem:convergencetuned}: ]$\;$
  \begin{itemize}
    \item By Lemma \ref{lem:ifrep},
    $$ \hat \theta_n \rightarrow^d \hat \theta \sim N(\theta_0, \Sigma).$$
    \item Let $\Delta_n(\lambda) = CV_n(\lambda) - c_n - SURE(\lambda, \hat \theta_n, \Sigma)$.
    By Lemma \ref{lem:convergence_cv},
    $$
      \sup_{\lambda \in \Lambda} \left|\Delta_n(\lambda)\right| \rightarrow^p 0.
    $$
    \item By joint convergence (\citealt{van2000asymptotic}, item (v) of Theorem 2.6), 
    $W_n = (\hat \theta_n, \Delta_n) \rightarrow^d W = (\hat \theta, 0)$.

    \item By Lemma \ref{lem:ascontinuous}, the mapping from $W_n$ to $\hat \theta_n + g^{\tilde{\lambda}(\hat \theta_n, \Delta_n)}(\hat \theta_n)$
    is almost surely continuous on the support of $W$.

    \item By definition, $\tilde{\lambda}(\hat \theta_n, \Delta_n) = \lambda_n^* = \argmin_\lambda CV_n(\lambda)$,\\
      and $\tilde{\lambda}(\hat \theta, 0) = \lambda^* = \argmin_\lambda SURE(\lambda, \hat \theta, \Sigma)$.

    \item The almost surely continuous mapping theorem (\citealt{van2000asymptotic}, Theorem 2.3) then implies
    $$
    \hat \theta_n + g^{\lambda_n^*}(\hat \theta_n) \rightarrow_d 
    \hat \theta + g^{\lambda^*}(\hat \theta) =
    \hat \theta^{\lambda^*}.
    $$
    \item By Lemma~\ref{lem:ifrep} $\hat \theta_n^* = \hat \theta_n+ g^{\lambda_n^*}(\hat \theta_n) + o_p(1)$, and the claim follows.
  \end{itemize}
\end{proof}

\subsection{Convergence of loss and risk}
\begin{proof}[Proof of Theorem \ref{theo:risk_convergence}:]$\;$\\
By Lemma~\ref{lem:limitingloss},
$$\bar L_n(\theta, \theta_0) \to \tfrac12 \|\theta - \theta_0\|^2$$
uniformly in any bounded neighborhood of $\theta_0$.
By Lemma \ref{lem:convergencetuned},  
$$
\hat \theta_n^* \rightarrow_d \hat \theta^*.
$$
Combining these two results gives
$$
\bar L_n(\hat \theta_n^*,\theta_0) \rightarrow_d \tfrac12 \|\hat \theta^* - \theta_0\|^2.
$$
The distributional convergence claim of Theorem \ref{theo:risk_convergence} follows.
The claim of Corollary \ref{cor:risk_convergence} is then immediate.
\end{proof}

\end{document}